\documentclass[11pt]{article}
\usepackage{bcc24}

\usepackage{ifthen,latexsym,amssymb,amsmath,bbm,fixmath}
\usepackage[shortlabels]{enumitem}
\usepackage{hyperref}
\bcctitle{Borel Combinatorics of Locally Finite Graphs}
\bccshorttitle{Borel Combinatorics of Locally Finite Graphs}

\bccname{Oleg Pikhurko\footnote{Supported by Leverhulme Research Project Grant RPG-2018-424.}}
\bccshortname{Oleg Pikhurko} 

\bccaddressa{Mathematics Institute and DIMAP\\
	University of Warwick\\
	Coventry CV4 7AL, UK}

\newtheorem{remark}[example]{Remark}

\newcommand{\hide}[1]{}
\newcommand{\beq}[1]{\begin{equation}\label{#1}}
	\newcommand{\eeq}{\end{equation}}
\newcommand{\qed}{\null\hfill\endmark}

\newcommand{\C}[1]{{\protect\mathcal{#1}}}

\newcommand{\I}[1]{{\mathbbm #1}}

\newcommand{\V}[1]{\mathbold{#1}}
\newcommand{\RR}[1]{\mathsf{#1}} 

\newcommand{\e}{\varepsilon}

\renewcommand{\mid}{:}
\renewcommand{\ge}{\geqslant}
\renewcommand{\le}{\leqslant}
\renewcommand{\preceq}{\preccurlyeq}

\newcommand{\actson}{{\curvearrowright}}
\newcommand{\dd}{\,\mathrm{d}}

\newcommand{\bex}{\begin{example}\rm }
	\newcommand{\eex}{\end{example}}

\newcommand{\cqed}{\nolinebreak\mbox{\hspace{5 true pt}%
		\rule[-0.85 true pt]{2.0 true pt}{8.1 true pt}}}
\newcommand{\bcpf}[1][Proof of Claim]{\smallskip\noindent{\it #1}}
\newcommand{\ecpf}{\cqed \medskip}

\newcommand{\ra}{\mbox{$\Rightarrow$}}
\newcommand{\DIAG}{\mathrm{Diag}}
\newcommand{\lex}{\mathrm{Lex}}
\newcommand{\induced}[2]{#1\upharpoonright #2}
\newcommand{\dist}{\mathrm{dist}}
\newcommand{\Ga}{\Gamma}
\newcommand{\ga}{\gamma}
\newcommand{\Local}{{\sf LOCAL}}
\newcommand{\graph}[1]{#1}
\newcommand{\Det}{\RR{Det}}
\newcommand{\Rand}{\RR{Rand}}
\newcommand{\var}{\mathrm{supp}}
\newcommand{\replace}[2]{{#1}\leadsto{#2}}
\newcommand{\dom}{\pi_0}
\newcommand{\boxdim}{\operatorname{dim}_{\C M}}
\newcommand{\OptionalCG}{}
\newcommand{\OptionalG}{}
\newcommand{\RT}{\RR C}
\newcommand{\CH}{{\C H}}

\newcommand{\Co}[1]{\cite[#1]{Cohn13mt}}
\newcommand{\Ts}[1]{\cite[#1]{Tserunyan19idst}}
\newcommand{\Ke}[1]{\cite[#1]{Kechris:cdst}}
\newcommand{\TW}[1]{\cite[#1]{TomkowiczWagon:btp}}
\newcommand{\Lo}[1]{\cite[#1]{Lovasz:lngl}}

\renewcommand{\rho}{t}

\begin{document}
\makebcctitle

\begin{abstract}
	We provide a gentle introduction, aimed at non-experts, to Borel combinatorics
	that studies definable graphs on topological spaces.
	This is an emerging field on the borderline between combinatorics and descriptive set theory  with deep connections to many other areas.
	
	After giving some background material, we present in careful detail some basic tools and results on the existence of Borel satisfying assignments: Borel versions of greedy algorithms and augmenting procedures, local rules, Borel transversals, etc. Also, we present the construction of Andrew Marks of acyclic Borel graphs for which the greedy bound $\Delta+1$ on the Borel chromatic number is best possible. 
	
	In the remainder of the paper we briefly discuss various topics such as relations to \Local\ algorithms, measurable versions of Hall's marriage theorem and of the Lov\'asz Local Lemma, applications to equidecomposability, etc.
\end{abstract}

	\section{Introduction}
	\label{se:Intro}

	\emph{Borel combinatorics}, also called \emph{measurable combinatorics} or \emph{descriptive (graph) combinatorics}, is an emerging and actively developing field that studies graphs and other combinatorial structures on topological spaces that are ``definable'' from the point of view of descriptive set theory. It 
	is an interesting blend of descriptive set theory and combinatorics that also has deep connections to measure theory, probability, group actions, ergodic theory, etc. We refer the reader to the survey by Kechris and Marks~\cite{KechrisMarks:survey}.
	
	Many recent advances in this field came from adopting various proofs and concepts of finite combinatorics to the setting of descriptive set theory. However, this potentially very fruitful interplay is not yet fully explored. One of the reasons is that
	a fairly large amount of background is needed in order to understand the proofs (or even the statements) of some results of this field. Thus the purpose of this paper is give a gentle introduction to some basic concepts and tools, aimed at non-experts, as well as to provide some pointers to further reading for those who would like to learn more.

	In order to give quickly a flavour of what types of objects and questions we will consider, here is one simple but illustrative example.
	
	\begin{Def}[Irrational Rotation Graph $\C R_\alpha$]
		\label{ex:Rotation} Let $\alpha\in\I R\setminus\I Q$ be irrational. Let $\C R_\alpha$ be the graph whose vertex set is the half-open interval $[0,1)$ of reals, where $x,y\in [0,1)$ are adjacent if their difference is equal to $\pm\alpha$ modulo $1$.
	\end{Def}

	Another way to define $\C R_\alpha$ is to consider the transformation $T_\alpha:[0,1)\to[0,1)$ which maps $x$ to $x+\alpha\pmod 1$; then the edge set consists of all unordered pairs $\{x,T_\alpha(x)\}$ over $x\in [0,1)$. Thus $\C R_\alpha$ can be viewed as the graph coming from an action of the group $\I Z$ on $[0,1)$ with the generator $1\in\I Z$ acting via~$T_\alpha$.  The name comes from identifying the interval $[0,1)$ with the unit circle $\I S^1\subseteq \I R^2$ via $x\mapsto (\cos (2\pi x),\sin(2\pi x))$ where $T_\alpha$ corresponds to the rotation by angle $2\pi\alpha$. Since $\alpha$ is irrational, $\C R_\alpha$ is a 2-regular graph with each component being a \emph{line} (a doubly-infinite path).
	
	Under the Axiom of Choice, $\C R_\alpha$ is combinatorially as trivial as a single line. For example, if we want to properly 2-colour its vertices then we can take a \emph{transversal} $S$ (a set containing exactly one vertex from each component) and colour every $x\in [0,1)$ depending on the parity of the distance from $x$ to $S$ in $\C R_\alpha$. However, this proof is non-constructive.
	
	So what is the chromatic number of $\C R_\alpha$ if we want each colour class, as a subset of the real interval $[0,1)$, be ``definable''? Of course, we have to agree first which sets are ``definable''. We would like the family of such sets to be closed under Boolean operations and under countable unions/intersections (i.e.\ to be a $\sigma$-algebra) so that various constructions, including some that involve passing to a limit after countably many iterations, do not take us outside the family. There are three important $\sigma$-algebras on $V=[0,1)$. 
	
	One, denoted by $\C B$, consists of \emph{Borel} sets and is, by definition, the smallest $\sigma$-algebra that contains all open sets, which for $[0,1)$ is equivalent to containing all intervals or just ones with rational endpoints. We can build Borel sets by starting with open sets and iteratively adding complements and countable unions of already constructed sets. Then each Borel set appears after $\mathbold{\beta}$-many iterations for some countable ordinal~$\mathbold{\beta}$ and thus
	can be ``described'' with countably many bits of information (motivating the name of descriptive set theory).
	
	The other two $\sigma$-algebras can be built by taking all Borel sets which can be additionally modified by adding or removing any ``negligible'' set of elements. For example, from the topological point of view any \emph{nowhere dense} set $X$ (that is, a set whose closure has empty interior) is ``negligible'': every non-empty open set $U$ contains a ``substantial'' part (namely, some non-empty open $W\subseteq U$) that avoids $X$ completely. We also consider \emph{meager} sets (that is, countable unions of nowhere dense sets) as ``negligible''. Now, the corresponding $\sigma$-algebra $\C T$ (of all sets $X$ with the symmetric difference 
	$X\bigtriangleup E$ being meager for some Borel $E$) consists precisely of the sets that have the \emph{property of Baire} and thus may be regarded as ``topologically nice''.
	
	The definition of our third $\sigma$-algebra $\C B_\mu$ of \emph{$\mu$-measurable} sets depends on a measure $\mu$ on Borel sets (which we take to be the Lebesgue measure in our Example~\ref{ex:Rotation}).  A set $X$ is called \emph{$\mu$-measurable} (or just \emph{measurable} when $\mu$ is understood) if there is a 
	Borel set $E$ such that the symmetric difference $X\bigtriangleup E$ is \emph{null} (that is, is contained in a Borel set of measure 0). For $[0,1)$ and the Lebesgue measure, a set is null if and only if, for every $\e>0$, it can be covered by countably many intervals whose sum of lengths is at most~$\e$. The measure $\mu$ extends in the obvious way to a measure on~$\C B_\mu\supseteq\C B$.

	Clearly, $\C B$ is a sub-family of $\C T$ and of $\C B_\mu$ but the last two $\sigma$-algebras are not compatible, with the notions of a ``negligible'' set being quite strikingly different. For example, one can partition the interval $[0,1)$ into two Borel sets one of which is meager and the other
	has Lebesgue measure 0,
	see e.g.\ \cite[Theorem 1.6]{Oxtoby:mc}.
	
	As it turns out, it is impossible to find a proper 2-colouring $[0,1)=X_0\cup X_1$ of the graph $\C R_\alpha$ from Example~\ref{ex:Rotation} with $X_0,X_1\in\C B_\mu$. Indeed, since the colours must alternate on each line, we have that the measure-preserving map $T_\alpha$ swaps $X_0$ and $X_1$ so these sets, if measurable, have Lebesgue measure $1/2$ each. However, by $T_\alpha(T_\alpha(X_0))=X_0$ this would contradict the fact that the composition $T_\alpha\circ T_\alpha=T_{2\alpha}$ is \emph{ergodic}, meaning that every invariant measurable set has measure 0 or~1. (For two different proofs of the last property, via Fourier analysis and via Lebesgue's density theorem, see e.g.~\cite[Proposition~4.2.1]{VianaOliveira:fet}.)%
	\hide{In fact, any transversal $S$ for $\I R_\alpha$ is a version of a Vitali's set and the Lebesgue measure cannot be extended to $S$ if we want to keep it invariant under rotations and countably additive: since countably many rotations of $S$, namely, $T_\alpha^n(S)$, $n\in\I Z$, are pairwise disjoint subsets of the measure-1 interval $[0,1)$, the measure of $S$ must be 0 and, since they cover $[0,1)$, the measure of $S$ must be positive.}
	Thus, although there are no edges between the lines of $\C R_\alpha$, we have to be careful with what we do on different lines so that the combined outcome is measurable.

	A similar argument using the \emph{generic ergodicity} of $T_{2\alpha}$ (namely, that every invariant set with the property of Baire is meager or has meager complement) shows that $\C R_\alpha$ cannot be 2-coloured with both parts in~$\C T$, see e.g.\ \Ts{Example 21.6} for a derivation.
	
	On the other hand, it is easy to show that we can properly 3-colour $\C R_\alpha$ with Borel colour classes $X_0$, $X_1$ and $X_2$.
	For example, we can let $X_2:=[0,c)$ for any $c>0$ which makes $X_2$ independent (namely, $c\le \min\{\alpha\pmod 1,1-\alpha\pmod 1\}$ is enough)
	and then colour every $x\in [0,1)\setminus X_2$ by the parity of the minimum integer $n\ge 1$ with $T_\alpha^n(x)\in X_2$. (Such $n$ exists by the irrationality of $\alpha$.) If $Y_n$ denotes the set of vertices for which the $n$-th iterate of $T_\alpha$ is the first one to hit $X_2$, then $Y_0=X_2$ and, for each $n\ge 1$, we have that $Y_n=T_\alpha^{-n}(X_2)\setminus(Y_0\cup\dots\cup Y_{n-1})$ is a finite union of half-open intervals by induction on $n$ and thus is Borel. We see that, for each of the $\sigma$-algebras $\C B$, $\C B_\mu$ and $\C T$,
	the minimum number of colours in a ``definable'' proper colouring of $\C R_\alpha$ happens to be the same, namely~$3$.

	To keep this paper of reasonable size, we will concentrate on results, where the assignments that we construct have to be Borel. Furthermore, except a few places where it is explicitly stated otherwise, we will restrict ourselves to \emph{locally finite} graphs, where every vertex has finitely many neighbours (but we usually do not require that the degrees are uniformly bounded by some constant). This is already a very rich area. Also, results of this type often form the proof basis for other settings. For example, if we want to find a proper colouring whose classes are measurable (resp.\ have the property of Baire) then one common approach is to build a Borel colouring and then argue that the set of vertices in connectivity components with at least one conflict (two adjacent vertices of the same colour)
	is null (resp.\ meager); in such situations, we are allowed to fix such components in an arbitrary fashion, e.g.\ by applying the Axiom of Choice. (Everywhere in this paper we assume that the Axiom of Choice holds.) 
	
	The following quotation of Lubotzky~\cite[Page~xi]{Lubotzky:dgegim} applies almost verbatim to this paper: \emph{``Generally speaking, I tried to write it in a form of something I wish had existed when, eight years ago, I made my first steps into these subjects without specific background in any of them.''}
	
	This paper is organised as follows. Some notation that is frequently used in the paper is collected in Section~\ref{se:Notation}. Then Section~\ref{se:Background} lists some basic facts about Borel sets. We assume that the reader is familiar with the fundamentals of topology. However, we try to carefully state all required, even rather basic results on Borel sets and functions, with references to complete proofs. Section~\ref{se:BorelGraphs} defines Borel graphs, our main object of study. 
	
	Section~\ref{se:LocallyFinite} presents some basic results for locally finite Borel graphs as follows.
	Section~\ref{se:Chi} contains the proofs of some classical results of 
	Kechris, Solecki and Todorcevic~\cite{KechrisSoleckiTodorcevic99} on Borel chromatic numbers and maximal independent sets. In Section~\ref{se:Local}, we prove that each ``locally defined'' labelling is a Borel function, a well-known and extremely useful result.
	Section~\ref{se:Smooth} presents one application of this result, namely that if one can pick exactly one vertex from each connectivity component in a Borel way then every locally checkable labelling problem 
	(\emph{LCL} for short) that is satisfiable admits a global Borel solution. Section~\ref{se:ElekLippner} present another important application namely that, for every LCL, we can carry all augmentations supported on vertex sets of size at most $r$ in a Borel way so that none remains. 
	
	The above mentioned results from Sections~\ref{se:Local}--\ref{se:ElekLippner} are frequently used and well-known to experts.
	However, detailed and accessible proofs of these results are hard to find in the literature. So, in order to fill this gap, the author carefully states and proves rather general versions of these results,  deviating from the philosophy of the rest of this paper of presenting just a simple case that conveys main ideas. The reader may prefer to skip the proofs from Sections~\ref{se:Local}--\ref{se:ElekLippner} and move to Sections~\ref{se:Marks16}--\ref{se:Equidec} that discuss more interesting results.

	Section~\ref{se:Marks16} presents
	the surprising result of Marks~\cite{Marks16} that the upper bound 
	$\Delta+1$ on the Borel chromatic number from Section~\ref{se:Chi} (that comes from an easy greedy algorithm) is in fact best possible even for acyclic graphs.

	At this point we choose to give brief pointers to some other areas, namely, Borel equivalence relations (Section~\ref{se:CBER}),  assignments with the property of Baire (Section~\ref{se:Baire}), and $\mu$-measurable assignments (Section~\ref{se:Measurable}). 
	
	With this background, however brief, we can discuss various results, in particular those that connect Borel combinatorics to measures and the property of Baire.
	Section~\ref{se:Bernshteyn20arxiv} presents the recent result of Bernshteyn~\cite{Bernshteyn20arxiv} that efficient \Local\ algorithms can be used to find Borel satisfying assignments of the corresponding LCLs. 
	Section~\ref{se:Borel+} discusses some ``purely Borel'' existence results where  measures come up in the proofs (but not in the statements).  Section~\ref{se:SubExp} discusses the class of graphs with subexponential growth for which one can prove some very general Borel results. Finally, Section~\ref{se:Equidec} presents
	some applications of descriptive combinatorics to \emph{equidecomposability} (where we try to split two given sets into congruent pieces).

	\section{Notation}
	\label{se:Notation}

	Here we collect some notation that is used in the paper. In order to help the reader to get into a logician's mindset we use some conventions from logic and set theory.
	
	We identify a non-negative integer $k$ with the set $\{0,\dots,k-1\}$, following the recursive definition of natural numbers as $0:=\emptyset$ being the empty set and $k+1:=k\cup\{k\}$. Thus $i\in k$ is a convenient shorthand for $i\in\{0,\dots,k-1\}$. The set of natural numbers is denoted by $\omega:=\{0,1,\dots\}$ and our integer indices usually start from 0. 
	
	Let $\pi_i$ denote the projection from a product $\prod_{j\in r} X_j$ of sets to its $(i+1)$-st coordinate~$X_i$.
	We identify a function $f:X\to Y$ with its graph $\{(x,f(x))\mid x\in X\}\subseteq X\times Y$. Thus the \emph{domain} and the \emph{image} of $f$ can be written respectively as $\pi_0(f)$ and $\pi_1(f)$. Also, the restriction of a function $f:X\to Y$ to $Z\subseteq X$ is $\induced{f}{Z}:=f \cap (Z\times Y)$, which is a function from $Z$ to~$Y$. For sets $X$ and $Y$, the set of all functions $X\to Y$ is denoted by $Y^X$. For functions $f_i:X\to Y_i$, $i\in 2$, let $(f_0,f_1):X\to Y_0\times Y_1$ denote the function that maps $x\in X$ to $(f_0(x),f_1(x))$.

	Let $X$ be a set.  The \emph{diagonal} $\DIAG_X$ is the set $\{(x,x)\mid x\in X\}\subseteq X^2$.
	We identify the elements of $2^X$ with the subsets of $X$, where a function $f:X\to 2$ corresponds to the preimage $f^{-1}(1)\subseteq X$. 
	A \emph{total order} on $X$ is a partial order on $X$ in which every two elements are compatible, that is, a transitive and antisymmetric subset $\preceq$ of $X^2$ such that for every $x,y\in X$ it holds that $x\preceq y$ or $y\preceq x$. The set $X$ is \emph{countable} if it admits an injective map into~$\omega$ (in particular, every finite set is countable); then its cardinality $|X|$ is the unique $k\in \omega\cup\{\omega\}$ such that there is a bijection between $X$ and~$k$.


	Let $G$ be a \emph{graph} by which we  mean a pair $(V,E)$, where $V$ is a set and $E$ is a subset of $V^2\setminus \DIAG_V$ which is \emph{symmetric}, i.e.\ for every $x,y\in V$, $(x,y)\in E$ if and only if $(y,x)\in E$. When it is convenient, we may work with unordered pairs, with $\{x,y\}\in E$ translating into $(x,y),(y,x)\in E$, etc.
	Note that we do not allow multiple edge nor loops.
	Of course, when we talk about matchings, edge colourings, etc, we mean subsets of~$E$, functions on $E$, etc, which are symmetric (in the appropriate sense). When the graph $G$ is understood, we use $V$ and $E$ by default to mean its vertex set and its edge set, and often remove the reference to $G$ from notation (except in the statements of theorems and lemmas, all of which we try to state fully).

	The usual definitions of graph theory apply. A \emph{walk} (of length $i\in\omega$) is a sequence $(x_0,\dots,x_i)\in V^{i+1}$ such that $(x_j,x_{j+1})\in E$ for every $j\in i$. A \emph{path} is a walk in which vertices do not repeat. Note that the length of a walk (or a path) refers to the number of edges.
	For $x,y\in V$, their \emph{distance} $\dist_G(x,y)$ is the shortest length of a path connecting $x$ to $y$ in $G$ (and is $\omega$ if no such path exists). A \emph{rooted graph} is a pair $(G,x)$ where $G$ is a graph and $x$ is a vertex of~$G$.

	The \emph{neighbourhood} of a set $A\subseteq V$ is 
	$$
	N_G(A):=\{y\in V\mid \exists\, x\in A\ (x,y)\in E\}.
	$$
	We abbreviate $N_G(\{x\})$ to $N_G(x)$. Note that $N_G(x)$ does not include~$x$. We denote
	the degree of $x$ by $\deg_G(x):=|N_G(x)|$. The graph $G$ is \emph{locally finite} (resp.\ \emph{locally countable}) if the neighbourhood of every vertex is finite (resp.\ countable).
	
	For $r\in\omega$, the \emph{$r$-th power $G^r$} of $G$ is the graph on the same vertex set $V$ where two distinct vertices are adjacent if they are at distance at most $r$ in~$G$. Also, the \emph{$r$-ball} around $A$,
	$$
	N^{\le r}_G(A):=N_{G^r}(A)\cup A,
	$$
	consists of those vertices of $G$ that are at distance at most $r$ from~$A$. A set of vertices $X\subseteq V$ is called \emph{$r$-sparse} if the distance between any two distinct vertices of $X$ is strictly larger than $r$ (or, equivalently, if $X$ is  an independent set in $G^r$).

	The subgraph \emph{induced} by $X\subseteq V$  in $G$ is 
	$\induced{G}{X}:=(X,E\cap X^2)$.  We call a set $X\subseteq V$ \emph{connected} if the induced subgraph $\induced{G}{X}$ is connected, that is, every two vertices of $X$ are connected by a path with all vertices in~$X$. By a \emph{(connectivity) component} of $G$ we mean a maximal connected set $X\subseteq V$.
	Let
	\beq{eq:CE}
	\C E_{G}:=\{(x,y)\in V^2\mid \dist(x,y)<\omega\},
	\eeq 
	denote the \emph{connectivity relation} of $G$. In other words, $\C E_G$ is the transitive closure of~$E\cup \DIAG_V$. Clearly, it is an equivalence relation whose classes are the components of~$G$. The \emph{saturation} of a set $A\subseteq V$ is 
	$$
	[A]_{G}:=\{x\in V\mid \dist(x,A)<\omega\},
	$$ 
	that is, the union of all components intersecting~$A$. In particular, the component of a vertex $x\in V$ is $[x]_{G}:=[\{x\}]_{G}$.

	Finally, we will need the following generalisation of Hall's marriage theorem to (infinite) graphs by Rado~\cite{Rado42}, whose proof (that relies on the Axiom of Choice) can also be found in e.g.\ \TW{Theorem C.2}.
	
	\begin{theorem}[Rado~\cite{Rado42}]
		\label{th:Rado42} Let $G$ be a bipartite locally finite graph. If $|N(X)|\ge |X|$ for every finite set $X$ inside a part then $G$ has a perfect matching.\qed
	\end{theorem}

	\section{Some Background and Standard Results}
	\label{se:Background}
	
	Here we present some basic results from analysis and descriptive set theory that we will use in this paper.
	We do not try to give any historic account
	of these results. Instead we just refer to their modern proofs (using sources which happen to be most familiar to the author).

	\subsection{Algebras and $\sigma$-Algebras}
	
	An \emph{algebra} on a set $X$ is a non-empty family $\C A\subseteq 2^X$ of subsets of $X$ which is closed under Boolean operations inside $X$ (for which it is enough to check that $A\cup B,X\setminus A\in \C A$ for every $A,B\in\C A$). In particular, the empty set and the whole set $X$ belong to~$\C A$.
	
	A  \emph{$\sigma$-algebra} on set $X$ is an algebra on $X$ which also is
	closed under countable unions. It follows that it is also closed under countable intersections.  Also, if we have a countable sequence of sets $A_0,A_1,\dots$ in a $\sigma$-algebra $\C A$, then $\liminf_{n} A_n$ (resp.\ $\limsup_{n} A_n$), the set of elements that belong to all but finitely sets $A_n$ (resp.\ belong to infinitely many sets $A_n$), is also in $\C A$. Indeed, we have, for example, that
	$$
	\textstyle\liminf_{n} A_n=\cup_{n\in\omega} \cap_{m=n}^\infty A_m,
	$$
	belongs to $\C A$ as the countable union of $\cap_{m=n}^\infty A_m\in\C A$.

	For an arbitrary family $\C F\subseteq 2^X$, let $\sigma_X(\C F)$ denote the $\sigma$-algebra on $X$ generated by $\C F$, that is, $\sigma_X(\C F)$ is the intersection of all $\sigma$-algebras on $X$ containing~$\C F$. (Note that the intersection is taken over a non-empty set since $2^X$ is an example of such a $\sigma$-algebra.)
	
	\subsection{Polish Spaces}
	
	Although the notion of a Borel set could be defined for general topological spaces, the theory becomes particularly nice and fruitful when the underlying space is Polish. Let us discuss this class of spaces first.
	
	By a \emph{topological} space we mean a pair $(X,\tau)$, where $\tau$ is a topology on a set $X$ (specifically, we view $\tau\subseteq 2^X$ as the collection of all open sets). When the topology $\tau$ on $X$ is understood, we usually just write $X$. The space $X$ is \emph{separable} if there is a countable subset $Y\subseteq X$ which is \emph{dense} (that is,
	every non-empty $U\in\tau$ intersects~$Y$). Also, $X$ is \emph{metrizable} if there is a metric $d$ which is \emph{compatible} with~$\tau$ (namely, $U\in \tau$ if and only if for every $x\in U$ there is real $r>0$ with the \emph{$r$-ball} $\{y\in X\mid d(x,y)<r\}$ lying inside~$U$). Furthermore, if $d$ can be chosen to be \emph{complete} (i.e.\ every Cauchy sequence converges to some element of $X$) then we call $X$ \emph{completely metrizable}. We call a topological space \emph{Polish} if it is separable and completely metrizable. 
	For more details on Polish spaces, we refer to, for example, Cohn~\cite[Chapter 8.1]{Cohn13mt}, Kechris~\cite[Chapter~3]{Kechris:cdst}, or Tserunyan~\cite[Part 1]{Tserunyan19idst}.
	
	This class of spaces was first extensively studied by Polish mathematicians (Sier\-pi\'n\-ski, Kuratowski, Tarski and others), hence the name. It has many nice properties, in particular being closed under many topological operations (such as passing to closed or more generally $G_\delta$ subsets, or taking various constructions like countable products or countable disjoint unions, etc). Also, Polish spaces satisfy various results crucial in descriptive set theory (such as the Baire Category Theorem). 
	So this class is the primary setting for Borel combinatorics. All topological spaces that we will consider in this survey are assumed to be Polish. As it is customary in this field, we do not fix a metric.
	
	Two basic examples of Polish spaces
	are the integers $\I Z$ (with the \emph{discrete topology} where every set is open) and the real line $\I R$ (with the usual topology, where a set is open if and only if it is a union of open intervals). Many other Polish spaces can be obtained from these; in fact, every Polish space is homeomorphic to a closed subset of $\I R^{\omega}$ (\cite[Theorem 4.17]{Kechris:cdst}).
	
	Also, it can be proved (without assuming the Continuum Hypothesis) that every uncountable Polish space has the same cardinality as e.g.\ the set of reals:

	\begin{theorem}\label{th:|Polish|} Each Polish space is either countable or has continuum many points.
	\end{theorem}
	
	\begin{proof} See e.g.\ \Ke{Corollary~6.5} or \Ts{Corollary~4.6}.\end{proof}

	\subsection{The Borel $\sigma$-Algebra of a Polish Space}
	
	Let $X=(X,\tau)$ be a Polish space.
	Its \emph{Borel $\sigma$-algebra} is $\C B(X,\tau):=\sigma_X(\tau)$, the $\sigma$-algebra on $X$ generated by all open sets.
	When the meaning is clear, we will usually just write $\C B(X)$ or $\C B$ instead of $\C B(X,\tau)$.

	The definition of the Borel $\sigma$-algebra behaves well with respect to various operations on Polish spaces. For example, the Borel $\sigma$-algebra of the product space $X\times Y$ is equal to the product of $\sigma$-algebras $\C B(X)\times\C B(Y)$:
	\beq{eq:BProduct}
	\C B(X\times Y)=\C B(X)\times \C B(Y),
	\eeq
	see e.g.~\cite[Proposition~8.1.7]{Cohn13mt} (and \Co{Section 5.1} for an introduction to products of $\sigma$-algebras).

	A function $f:X\to Y$ is \emph{Borel (measurable)} if the preimage of every Borel subset of $Y$ is a Borel subset of $X$. It is easy to see (\Co{Proposition 2.6.2}) that it is enough to check this condition for any family $\C A\subseteq \C B(Y)$ that generates $\C B(Y)$, e.g.\ just for open sets.
	Since we identify the function $f$ with its graph $\{(x,f(x))\mid x\in X\}\subseteq X\times Y$, the following result justifies why it is more common to call a Borel measurable function just ``a Borel function''.
	
	\begin{lemma}\label{lm:BorelF} A function $f:X\to Y$ between Polish spaces is Borel measurable if and only if it is
		a Borel subset of~$X\times Y$.\end{lemma}
	\begin{proof} See e.g.\ \cite[Proposition 8.3.4]{Cohn13mt}.\end{proof}

	Although there are many non-equivalent choices of a Polish topology $\tau$ on an infinite countable set $X$, each of them has $2^X$ as its Borel $\sigma$-algebra. Indeed, every subset $Y$ of $X$, as the countable union $Y=\cup_{y\in Y} \{y\}$ of closed sets, is Borel.
	A remarkable generalisation of this trivial observation is the following.
	
	\begin{theorem}[Borel Isomorphism Theorem]
		\label{th:BIT}
		Every two Polish spaces $X,Y$ of the same cardinality are \emph{Borel isomorphic}, that is, there is a bijection $f:X\to Y$ such that the functions $f$ and $f^{-1}$ are Borel.
		(In fact, by Theorem~\ref{th:LNU}, it is enough to require that just one of the bijections $f$ or $f^{-1}$ is Borel.)
	\end{theorem} 
	\begin{proof} See e.g.\  \cite[Theorem 8.3.6]{Cohn13mt} or \cite[Theorem 13.10]{Tserunyan19idst}.\end{proof}

	One crucial property of the Borel $\sigma$-algebra is that it can be generated by a countable family of sets.
	
	\begin{lemma}\label{lm:J}
		For every Polish space $X$, there is a countable family $\C J=\{J_n\mid n\in\omega\}$ with $\C B(X)=\sigma_X(\C J)$. \end{lemma}
	\begin{proof} Fix a countable dense subset $Y\subseteq X$ and a compatible metric $d$ on $X$. Let $\C J$ consist of all open balls in the metric $d$ around points of $Y$ with rational radii. Of course, $\C J\subseteq \C B$. Every open set is a union of elements of~$\C J$ and this union is countable as $\C J$ is countable. Thus $\sigma_X(\C J)$ contains all open sets and, therefore, has to be equal to~$\C B$.\end{proof}

	\begin{remark}\label{rm:GenAlg}
		We can additionally require in Lemma~\ref{lm:J} that $\C J$ is an algebra on~$X$. Indeed, given any generating countable family $\C J\subseteq \C B$, we can enlarge it by adding all Boolean combinations of its elements. This does not affect the equality $\sigma_X(\C J)=\C B$ (as all added sets are in $\C B$) and keeps the family countable.\ecpf
	\end{remark}
	
	For example, for $X:=\I R$, $d(x,y):=|x-y|$, and $Y:=\I Q$, our proof of Lemma~\ref{lm:J} gives the family of open intervals with rational endpoints. If we want to have a generating algebra for $\C B(\I R)$ as in Remark~\ref{rm:GenAlg}, it may be more convenient to use half-open intervals in the first step, namely to take
	$$
	\C J:=\{[a,b)\mid a,b\in\I Q\} \cup \{[a,\infty)\mid a\in \I Q\}\cup \{(-\infty,b)\mid b\in\I Q\}.
	$$
	Then the algebra generated by $\C J$ will precisely consists of finite unions of disjoint intervals in~$\C J$.

	Note that, for every distinct $x,y\in X$, there is an open (and thus Borel) set $U$ that \emph{separates} $x$ from $y$, that is,  $U$ contains $x$ but not $y$ (e.g.\ an open ball around $x$ of radius $d(x,y)/2$ for some compatible metric~$d$). Thus Borel sets separate points. In fact,  we have the following stronger property.
	
	\begin{lemma}\label{lm:Sep} Let $X$ be a Polish space and let $\C J$ be an algebra on $X$ that generates~$\C B(X)$. Then for every pair of disjoint finite sets $A,B\subseteq X$ there is $J\in\C J$ with $A\subseteq J$ and $B\cap J=\emptyset$.\end{lemma}
	
	\begin{proof} 
		For distinct $a,b\in X$, there is some $J_{a,b}\in\C J$ that contains exactly one of $a$ and~$b$ (as otherwise $\C B(X)=\sigma_X(\C J)$ cannot separate $a$ from $b$) and, by passing to the complement if necessary, we can assume that $J_{a,b}$ separates $a$ from $b$.
		For $a\in A$, the set $J_a:=\cap_{b\in B}J_{a,b}$ belongs to $\C J$, contains $a$ and is disjoint from $B$. Thus $J:=\cup_{a\in A} J_a$ satisfies the claim.\end{proof}

	The \emph{lexicographic order} $\le_\lex$ on $2^{\omega}$ is defined so that $(x_i)_{i\in\omega}$ comes before $(y_i)_{i\in\omega}$ if 
	$x_i<y_i$
	where $i\in\omega$ is the smallest index with $x_i\not= y_i$.
	
	\begin{lemma}\label{lm:I} For every Polish space $X$ there is a Borel injective map
		$I:X\to 2^{\omega}$. Furthermore, given any such map $I$, if we define $\preceq$ to consist of those pairs $(x,y)\in X^2$ with $I(x)\le_\lex I(y)$, then we obtain a total order on~$X$ which is Borel (as a subset of $X^2$).
	\end{lemma}
	\begin{proof} By Lemma~\ref{lm:J}, fix  a countable family $\C J$ that generates $\C B$ and define
		\beq{eq:I}
		I(x):=(\I 1_{J_0}(x),\I 1_{J_1}(x),\dots),\quad x\in X,
		\eeq
		where the \emph{indicator function} $\I 1_J$ of $J\subseteq X$ assumes value 1 on $J$ and 0 on~$X\setminus J$. In other words, $I(x)$ records which sets $J_i$ contain~$x$. The Borel $\sigma$-algebra on the product space $2^{\omega}$ is generated by the sets of the form $Y_{i,\sigma}:=\{x\in 2^{\omega}: x_i=\sigma\}$ for $i\in\omega$ and $\sigma\in 2$ (as these form a pre-base for
		the product topology on $2^{\omega}$). Since $I^{-1}(Y_{i,1})=J_i$ and $I^{-1}(Y_{i,0})=X\setminus J_i$ are in $\C B$, the map $I$ is Borel. Also,  $I$ is injective by Lemma~\ref{lm:Sep}.
		
		By the injectivity of $I$, we have that $\preceq$ is a total order on $X$. Also, the lexicographic order $\le_{\lex}$ is a closed (and thus Borel) subset of $(2^{\omega})^2$. Indeed, if $x\not\le_{\lex} y$ then,
		with $i\in\omega$ being the maximum index such that $(x_0,\dots,x_{i-i})=(y_0,\dots,y_{i-1})$,  every pair $(x',y')$ that coincides with $(x,y)$ on the first $i+1$ indices (an open set of pairs) is not in $\le_{\lex}$. Thus $\preceq$, as the preimage under the Borel map
		$I\times I: X\times X\to 2^{\omega}\times 2^{\omega}$ that sends $(x,y)$ to $(I(x),I(y))$,
		is a Borel subset of $X^2$. (Our claim that the map $I\times I$ is Borel can be easily derived from~\eqref{eq:BProduct}.)\end{proof}

	If $X$ is a Borel subset of $[0,1)$ then another, more natural, choice of $I$  in Lemma~\ref{lm:I} is to map $x\in[0,1)$ to the digits of its binary expansion where we do not allow infinite sequences of trailing~$1$'s. Then $\preceq$ becomes just the standard order on~$[0,1)$.

	While continuous images of Borel sets are not in general Borel, this is true all for countable-to-one Borel maps for which, moreover, a Borel right inverse exists (with the latter property called \emph{uniformization} in descriptive set theory).
	
	\begin{theorem}[Lusin-Novikov Uniformization Theorem]\label{th:LNU} Let $X,Y$ be Polish spaces and let $f:X\to Y$ be a Borel map.
		Let $A\subseteq X$ be a Borel set such that every $y\in Y$ has countably many preimages in $A$ under~$f$. 
		Then  $f(A)$ is a Borel subset of $Y$. Moreover, there are countably many Borel maps $g_n: f(A)\to A$, $n\in\omega$, each being a right inverse to $f$ (i.e.\ the composition $f\circ g_n$ is the identity function on $f(A)$), such that for every $x\in A$ there is $n\in \omega$ with $g_n(f(x))=x$.\end{theorem}
	
	\begin{proof} 
		This theorem, in the case when $X=Z\times Y$ for some Polish $Z$ and $f=\pi_1$ is the projection on $Y$,  can be found in e.g.\ \Ke{Theorem 18.10} or~\Ts{Theorem 13.6}. To derive the version stated here,
		let $A':=(A\times Y)\cap \graph{f}$. (Recall that we view the function $f:X\to Y$ as a subset $f\subseteq X\times Y$.) Then $A'\subseteq X\times Y$ is Borel by Lemma~\ref{lm:BorelF}, $f(A)=\pi_1(A')$ is Borel by the above product version  
		while the sequence of Borel right inverses $g_n':f(A)\to A'$ for $\pi_1$
		gives the required functions $g_n:=\pi_0\circ g_n'$.\end{proof} 
	
	One can view the second part of Theorem~\ref{th:LNU}
	as  a ``Borel version'' of the Axiom of Choice. If, say, $A=X$ and $f$ is surjective,  then finding one right inverse $g_0$ of $f$ amounts to picking exactly one element from each of the sets $X_y:=f^{-1}(y)$ indexed by $y\in Y$. Thus Theorem~\ref{th:LNU} gives that, if all dependences are ``Borel'' and each set $X_y$ is countable, then a Borel choice is possible. This result was generalised to the case where each $X_y$ is required to be only \emph{$\sigma$-compact} (a countable union of compact sets), see e.g.\ \Ke{Theorem 35.46.ii}.

	\subsection{Standard Borel Spaces}\label{se:SBS}

	A  \emph{standard Borel space} is a pair $(X,\C A)$ where $X$ is a set and there is a choice of a Polish topology $\tau$ on $X$ such that $\C A$ is equal to $\sigma_X(\tau)$, the Borel $\sigma$-algebra of~$(X,\tau)$. For a detailed treatment of standard Borel spaces, see e.g.\ \Ke{Chapter 12} or \cite[Chapter~3]{Srivastava98cbs}.

	We will denote the Borel $\sigma$-algebra on a standard Borel space $X$ by $\C B(X)$ or $\C B$, also abbreviating $(X,\C B)$ to~$X$. It is customary  not to fix a Polish topology $\tau$ on $X$ (which, strictly speaking,
	requires checking that various operations defined on standard Borel spaces do not depend on the choice of topologies).
	
	By the Borel Isomorphism Theorem (Theorem~\ref{th:BIT}) combined with Theorem~\ref{th:|Polish|}, every standard Borel space either consists of all subsets of a countable set $X$ or admits a Borel isomorphism into e.g.\ the interval $[0,1]$ of reals. Since this survey concentrates on the Borel structure, we could have, in principle, restricted ourselves to just these special Polish spaces. However, many constructions and proofs are much more natural and intuitive when written in terms of general Polish spaces.

	A useful property that is often used without special mention is that
	every Borel subset induces a standard Borel space. This follows from the following result.
	
	\begin{lemma}\label{lm:SubB} Let $(X,\tau)$ be a Polish space. Then for every $Y\in\C B(X,\tau)$ there is a Polish topology $\tau'\supseteq \tau$ on $X$ such that $Y$ is a closed set with respect to~$\tau'$ and  $\C B(X,\tau)=\C B(X,\tau')$.
	\end{lemma} 
	\begin{proof} See e.g.~\Ke{Theorem 13.1} or \Ts{Theorem 11.16}.\end{proof}
	
	\begin{remark}\label{rm:SubB} In fact, any two nested Polish topologies $\tau\subseteq \tau'$ on a common set $X$ generate the same Borel $\sigma$-algebra (so this conclusion could be omitted from the statement of Lemma~\ref{lm:SubB}). Indeed, the identity map $(X,\tau')\to (X,\tau)$ is continuous and thus Borel. By Theorem~\ref{th:LNU}, its unique right inverse, which is the identity map $(X,\tau)\to (X,\tau')$, is Borel. Thus $\C B(X,\tau')=\C B(X,\tau)$.\end{remark}

	\begin{cor}\label{cr:SubB} If $X$ is a standard Borel space and $Y\in\C B(X)$, then $(Y,\induced{\C B(X)}{Y})$ is a standard Borel space, where for $\C A\subseteq 2^X$ we denote $\induced{\C A}{Y}:=\{A\cap Y\mid A\in\C A\}$.\end{cor}
	
	\begin{proof} Fix a Polish topology $\tau$ on $X$ that generates $\C B(X)$, that is, $\sigma_X(\tau)=\C B(X)$. By Lemma~\ref{lm:SubB}, there is a Polish topology $\tau'\supseteq \tau$ for which $Y$ is closed. Let $\tau'':=\{U\cap Y\mid U\in\tau'\}$. As it is easy to see, $(Y,\tau'')$ is a Polish space and $\sigma_Y(\tau'')=\induced{\sigma_X(\tau')}{Y}$. By the second conclusion of Lemma~\ref{lm:SubB} (or by Remark~\ref{rm:SubB}), we have that $\sigma_X(\tau')=\sigma_X(\tau)$, finishing the proof.\end{proof}
	
	\section{Borel Graphs: Definition and Some Examples}
	\label{se:BorelGraphs}
	
	For a short discussion of bounded-degree Borel graphs, see Lov\'asz~\cite[Section 18.1]{Lovasz:lngl}.
	
	A \emph{Borel graph} is a triple $\C G=(V,E,\C B)$ such that $(V,E)$ is a graph (that is, $E\subseteq V^2$ is a symmetric and anti-reflexive relation), $(V,\C B)$ is a standard Borel space, and $E$ is a Borel subset of $V\times V$.
	(As it follows from~\eqref{eq:BProduct}, the Borel $\sigma$-algebra on $V\times V$ depends only on $\C B(V)$ but not on the choice of a compatible Polish topology $\tau$ on~$V$.)
	
	All our graph theoretic notation will also apply to Borel graphs and, when the underlying Borel graph $\C G$ is clear, we usually remove any reference to $\C G$ from notation (writing $N(x)$ instead of $N_{\C G}(x)$, etc). Note that we  use the calligraphic letter $\C G$ to emphasise that it is a Borel graph
	(whereas $G$ is used for general graphs).
	
	\begin{remark}
		\label{rm:VChoose2}
		If one prefers, then one can work with the edge set as a subset of ${V\choose 2}$, the set of all unordered pairs of distinct elements of~$V$.
		The standard $\sigma$-algebra on ${V\choose 2}$ (which also makes it a standard Borel space) is obtained by taking all those sets $A\subseteq {V\choose 2}$ for which  $\pi^{-1}(A)$ is Borel subset of $V^2$, where $\pi:V^2\to {V\choose 2}$ is the natural projection mapping $(x,y)$ to $\{x,y\}$. 
		In terms of Polish topologies, if we fix a topology $\tau$ on $V$, then we consider the \emph{factor topology} $\tau'$ on ${V\choose 2}$ with respect to $\pi$ (that is, the largest topology that makes $\pi$ continuous) and take the Borel $\sigma$-algebra of~$\tau'$.\ecpf
	\end{remark}

	Here are some examples of Borel graphs.

	\bex\label{ex:CtblG} For every graph $(V,E)$ whose vertex set is countable, the triple $(V,E,2^V)$ is a Borel graph.  Indeed, $(V,2^V)$ is a standard Borel space (take, for example, the discrete topology on~$V$). Then $\C B(V\times V)$ contains all singleton sets (as closed sets) and, being closed under countable unions, it contains all subsets of $V\times V$, in particular, the edge set~$E$.\ecpf\eex

	\bex\label{ex:BAction}
	Let $(\Ga;S)$ be a \emph{marked group}, that is $\Ga$ is a group generated by a finite set $S\subseteq \Ga$ which is \emph{symmetric}, that is, $S=S^{-1}$ where $S^{-1}:=\{\gamma^{-1}\mid \gamma\in S\}$. (We do not assume that $S$ is minimal in any sense.) Let $a:\Gamma\actson X$ be a (left) action of $\Gamma$ on a Polish space $X$ which is \emph{Borel}, meaning for the countable group $\Ga$ 
	that
	for every $\gamma\in \Ga$ the bijection $a(\gamma,\cdot):X\to X$, which maps $x\in X$ to $\gamma.x$, is Borel. Let the \emph{Schreier graph} $\C S(a;S)$  have $X$ as the vertex set and 
	\beq{eq:Schreier}
	\{(x,\gamma.x)\mid x\in X,\ \gamma\in S\}\setminus\DIAG_X
	\eeq 
	as the edge set. (Note that, regardless of the choice of a Polish topology on $X$,
	the diagonal $\DIAG_X=\{(x,x)\mid x\in X\}$ is a closed and thus Borel subset of~$X^2$.) Thus $\C S(a;S)$ is a Borel graph by Lemma~\ref{lm:BorelF}.\ecpf\eex
	
	For example, the irrational rotation graph $\C R_\alpha$ from Example~\ref{ex:Rotation} is a Borel graph, e.g.\ as the Schreier graph of the Borel action of the marked group $(\I Z,\{-1,1\})$ on $[0,1)$ given by $n.x:=x+n\alpha\pmod 1$ for $n\in\I Z$ and $x\in [0,1)$.
	
	By Corollary~\ref{cr:SubB}, if $\C G$ is a Borel graph and $Y$ is a Borel subset of $V$, then $\induced{\C G}{Y}$ is again a Borel graph. 
	Here is one case that often arises in the context of Example~\ref{ex:BAction} (and that we will need later in Section~\ref{se:Marks16}). 
	The \emph{free part} of a group action $a:\Ga\actson X$ is 
	\beq{eq:Free}
	\mathrm{Free}(a):=\{x\in X\mid \forall \gamma\in\Gamma\setminus\{e\}\  \ \gamma.x\not=x\}.
	\eeq
	As it is trivial to see, $x\in X$ belongs to the free part if and only if the map 
	$\Gamma\to X$ 
	that sends $\gamma\in \Gamma$ to $\gamma.x$ is injective.

	\begin{lemma}\label{lm:FreeB} The free part of a Borel action $a:\Gamma\actson X$ of a countable group $\Gamma$ is Borel.\end{lemma}
	\begin{proof} Using the definition in~\eqref{eq:Free}, we see that
		$$
		X\setminus \mathrm{Free}(a)=\bigcup_{\gamma\in\Ga\setminus\{e\}} \{x\in X\mid \ga.x=x\}=\bigcup_{\gamma\in\Ga\setminus\{e\}} \pi_0(\DIAG_X\cap \{(x,\gamma.x)\mid x\in X\})
		$$
		is a Borel set by Lemma~\ref{lm:BorelF} and the Lusin-Novikov Uniformization Theorem (Theorem~\ref{th:LNU}).\end{proof}

	\section{Basic Properties of Locally Finite Borel Graphs}
	\label{se:LocallyFinite}
	
	Recall that, unless stated otherwise,
	we consider \emph{locally finite} graphs only, that is, those graphs in which every vertex has finitely many neighbours. Note that we do not require that all degrees are uniformly bounded by some constant. This is already a very rich and important class in descriptive combinatorics, and is a very natural one from the point of view of finite combinatorics. 
	
	\begin{lemma}\label{lm:BorelG} Let $(V,\C B)$ be a standard Borel space and let $G=(V,E)$ be a locally finite graph. Then the following are equivalent.
		\begin{enumerate}[(i),nosep] 
			\item\label{it:BorelG1} The set $E\subseteq V^2$ is Borel (i.e.\ $(V,E,\C B)$ is a Borel graph).
			\item \label{it:BorelG2} For every Borel set $Y\subseteq V$, its neighbourhood $N_G(Y)$ is Borel.
			\item \label{it:BorelG3} For every Borel set $Y\subseteq V$, the $1$-ball $N^{\le 1}_G(Y)$ around $Y$ (i.e.\ the set of vertices at distance at most $1$ from $Y$) is Borel.
		\end{enumerate}
	\end{lemma}
	
	\begin{proof} Let us show that \ref{it:BorelG1} implies \ref{it:BorelG2}. Take any Borel $Y\subseteq V$. Note that $N(Y)$ is the projection of $Z:=(Y\times V)\cap E$ on the second coordinate.
		The projection is a continuous map and, as a map from $Z$ to $V$, has countable (in fact, finite) preimages. Thus $N(Y)=\pi_1(Z)$ is Borel by Theorem~\ref{th:LNU}. 
		
		The implication \ref{it:BorelG2}\ \ra\ \ref{it:BorelG3} trivially follows from
		$N^{\le1}(Y)=N(Y)\cup Y$ (and the $\sigma$-algebra $\C B$ being closed under finite unions).
		
		Let us show that \ref{it:BorelG3} implies \ref{it:BorelG1}. 
		Remark~\ref{rm:GenAlg} gives us a countable algebra $\C J$ on $V$  that generates~$\C B$.
		For $J\in\C J$, let $A_J$ be the union of $J\times (V\setminus N^{\le 1}(J))$ and its ``transpose'' $(V\setminus N^{\le 1}(J))\times J$. By~\eqref{eq:BProduct}, each set $A_J$ is Borel.  Recall that the diagonal $\DIAG_V=\{(x,x)\mid x\in V\}$ is a closed and thus Borel subset of~$V^2$. It is enough to prove that 
		\beq{eq:BorelG3}
		E=V^2\setminus \left(\DIAG_V\cup \left(\cup_{J\in \C J} A_J\right)\right),
		\eeq
		because this writes  $E$ as the complement of a countable union of Borel sets.
		
		By definition, each $A_J$ is disjoint from $E$ and thus the forward inclusion in~\eqref{eq:BorelG3} is obvious. Conversely, take any $(x,y)\in V^2\setminus E$. Suppose that $x\not=y$ as otherwise $(x,y)\in \DIAG_V$ and we are done. By 
		Lemma~\ref{lm:Sep}
		there is a set $J\in\C J$ which contains $x$ but is disjoint from the finite set~$N^{\le 1}(y)$. Then $y\not\in N^{\le 1}(J)$ and the pair $(x,y)$ belongs to $A_J$, giving the required.\end{proof}

	\begin{cor}\label{cr:Gr} Let $r\in \omega$. If $\C G=(V,E,\C B)$ is a locally finite Borel graph, then so is its $r$-th power graph $\C G^r$ (as a graph on the standard Borel space~$V$).\end{cor}
	\begin{proof} Trivially (or by a very special case of the K\"onig Infinity Lemma, see e.g.\ \cite[Lemma~8.1.2]{Diestel17gt}), the graph $\C G^r$ is locally finite. 
		
		In order to check that $\C G^r$ is a Borel graph, we verify Condition~\ref{it:BorelG3} of Lemma~\ref{lm:BorelG}. By the definition of $\C G^r$, we have  for every $A\subseteq V$ that $N^{\le 1}_{\C G^r}(A)=N^{\le r}_{\C G}(A)$. Also, we can construct $N^{\le r}_{\C G}(A)$ from $A$ by iteratively applying  $r$ times the $1$-ball operation in~$\C G$. By Lemma~\ref{lm:BorelG}\ref{it:BorelG3}, this operation preserves Borel sets. Thus $N^{\le 1}_{\C G^r}(A)$ is Borel for every Borel~$A\subseteq V$. We conclude that Condition~\ref{it:BorelG3} of Lemma~\ref{lm:BorelG} is satisfied for $\C G^r$ and thus this graph  is Borel.\end{proof}

	Here is an important consequence to Corollary~\ref{cr:Gr}. Recall that the connectivity relation $\C E_{G}$ of a graph $G$ consists of all pairs of vertices that lie in the same connectivity component of~$G$.
	
	\begin{cor}\label{cr:CEBorel}  If $\C G=(V,E,\C B)$ is a locally finite Borel graph, then $\C E_{\C G}$ is a Borel subset of $V^2$.\end{cor}
	
	\begin{proof} We have that $\C E_{\C G}$ is the union of the diagonal $\DIAG_V$ and the edge sets of $\C G^r$ over $r\ge 1$. As each $\C G^r$ is a Borel graph by Corollary~\ref{cr:Gr}, the set $\C E_{\C G}\subseteq V^2$ is Borel.\end{proof}
	
	\begin{remark}
		Addressing a question of the author, Chan~\cite{Chan20u} showed, under the set-theoretic assumption that the constructable continuum is the same as the continuum, that Lemma~\ref{lm:BorelG} fails for general locally countable graphs: 
		namely Property~\ref{it:BorelG2} (and thus also Property~\ref{it:BorelG3}) does not imply Property~\ref{it:BorelG1}. An example is given by the countable equivalence relation studied in
		\cite[Section~9]{Chan17jsl}. (Thus, combinatorially, the graph is a union of countable cliques.) However, it remains unclear if the above additional set-theoretic assumption is needed.  
	\end{remark}
	
	\subsection{Borel Colourings}
	\label{se:Chi}
	
	Let $\C G=(V,E,\C B)$ be a Borel graph (which we assume to be locally finite).
	
	A colouring $c:V\to X$ is \emph{proper} if no two adjacent vertices get the same colour. Since we will consider only the case of countable colour sets $X$ (when $2^X$ is the only $\sigma$-algebra making it a standard Borel space), we define $c: V\to X$ to be \emph{Borel} if the preimage under $c$ of every element of $X$ is Borel. When we view $c$ as a vertex colouring, this amounts to saying that each of (countably many) colour classes belongs to $\C B(V)$.
	
	\begin{lemma}\label{lm:CountableCol} Every locally finite Borel graph $\C G=(V,E,\C B)$ admits a Borel proper colouring $c:V\to\omega$.\end{lemma}
	
	\begin{proof} Fix a countable algebra $\C J=\{J_0,J_1,\dots\}$ generating $\C B(V)$, which exists by
		Remark~\ref{rm:GenAlg}. Define
		$$A:=\{(x,k)\in V\times \omega\mid x\in J_k\ \wedge\ N(x)\cap J_k=\emptyset\}.
		$$
		Thus $(x,k)\in A$ if the $(k+1)$-st set $J_k$ of $\C J$ separates $x$ from all its neighbours. 
		For every $x\in V$, there is at least one $k\in\omega$ with $(x,k)\in A$ by Lemma~\ref{lm:Sep}, and we define $c(x)$ to be the smallest such $k\in\omega$. 
		
		Clearly, if two distinct vertices $x$ and $y$ get the same colour $k$ then they are both in $J_k$ and cannot be adjacent as their neighbourhoods are disjoint from~$J_k$. Thus $c:V\to\omega$ is a proper vertex colouring.

		It remains to argue that the map $c:V\to\omega$ is Borel. For $k\in\omega$, define
		$$
		B_k:=\{x\in V\mid c(x)\ge k\}.
		$$
		Since 
		$c^{-1}(k)=B_{k}\setminus B_{k+1}$, it is enough so show that each $B_k$ is Borel.
		The complement of $B_k$ is exactly the image of $A\cap (V\times k)$ under the projection~$\pi_0$. Thus, by Theorem~\ref{th:LNU}, it suffices to show that $A\subseteq V\times\omega$ is Borel. As it is easy to see, $A=\cup_{k\in\omega} ((J_k\setminus N(J_k))\times \{k\})$. Each set in this countable union is Borel by Lemma~\ref{lm:BorelG}\ref{it:BorelG2}. So $A$ is Borel, finishing the proof.\end{proof}
	
	\hide{
		\begin{remark} Alternatively, a colouring $c$ satisfying Lemma~\ref{lm:CountableCol} can be  constructed from any Borel injection $I:V\to 2^{\omega}$ (that exists by Lemma~\ref{lm:I}) by letting $c(x)$ be the smallest prefix of $I(x)$ which distinguishes $x$ from every neighbour $y\in N(x)$. (Here, the set of colours will be the set of finite binary sequences, which is still a countable set.)
			\ecpf\end{remark}
	}
	
	The main idea of the proof of the following result is very useful in Borel combinatorics: we first show that we can cover $V$ by countably many  sets that are sufficiently ``sparse'' (namely, independent in this proof) and then apply some parallel algorithm (namely, greedy colouring) where we take these sets one by one and process all vertices of the taken set in one go.

	\begin{theorem}[Kechris et al.~\cite{KechrisSoleckiTodorcevic99}]
		\label{th:MaxInd} 
		Every locally finite Borel graph $\C G$ has a maximal independent set $A$ which is Borel.
	\end{theorem}
	
	\begin{proof} Let $c:V\to\omega$ be the proper Borel colouring returned by Lemma~\ref{lm:CountableCol}.
		We apply the greedy algorithm where we process colours $i\in\omega$ one by one and, for each $i$, add to $A$ in parallel all vertices of colour $i$ that have no neighbours in the current set~$A$.
		
		Formally, let $A_{0}:=\emptyset$ and, inductively for $i\in\omega$, define
		$$
		A_{i+1}:=A_{i}\cup (c^{-1}(i)\setminus N(A_{i})).
		$$
		Finally, define $A:=\cup_{i\in\omega} A_i$.

		As $c$ is a proper colouring, each set $A_i$ is independent by induction on~$i$. Thus $A$, as the union of nested independent sets, is independent. Also, the set $A$ is maximal independent. Indeed, if $x\not\in A$ then, with $i:=c(x)$, the reason for not adding $x$ into $A_{i+1}\subseteq A$ was that $x$ has a neighbour in $A_{i}$ and thus a neighbour in $A\supseteq A_{i}$. 
		
		An easy induction on $i$ shows by Lemma~\ref{lm:BorelG}\ref{it:BorelG2} that each $A_i$ is Borel. Thus $A=\cup_{i\in\omega} A_i$ is also Borel.\end{proof}

	Recall that a set $A$ of vertices in a graph $G$ is called \emph{$r$-sparse} if the distance in $G$ between every two distinct elements of $A$ is larger than~$r$.
	
	\begin{cor}\label{cr:MaxRSparse}
		Every locally finite Borel graph $\C G$ has a maximal $r$-sparse set $A$ which is Borel.
	\end{cor}
	\begin{proof} A set $A\subseteq V$ is (maximal) $r$-sparse in $\C G$ if and only if it is (maximal) independent in~$\C G^r$. Thus the required Borel set $A$ exists by Theorem~\ref{th:MaxInd} applied to
		$\C G^r$, which is a locally finite Borel graph by Corollary~\ref{cr:Gr}.\end{proof}

	The \emph{Borel chromatic number $\chi_{\C B}(\C G)$}  of an arbitrary Borel graph $\C G$
	is  defined as the smallest cardinality of a standard Borel space $Y$ for which
	there is a Borel proper colouring $c:V\to Y$.  Trivially, $\chi_{\C B}(\C G)$ is at least the usual \emph{chromatic number} $\chi(\C G)$, which is the smallest cardinality of a set $Y$ with $\C G$ admitting a proper vertex colouring $V\to Y$ (which need not be constructive in any way); for more on $\chi(G)$ for infinite graphs see e.g.\ the survey by Komj\'ath~\cite{Komjath11}.
	Since we restrict ourselves to locally finite graphs here, we have by Lemma~\ref{lm:CountableCol} that both $\chi_{\C B}(\C G)$ and $\chi(\C G)$ are in~$\omega\cup\{\omega\}$.

	\begin{theorem}[Kechris et al.~\cite{KechrisSoleckiTodorcevic99}]
		\label{th:ChiB}
		Every Borel graph $\C G=(V,E,\C B)$ with finite maximum degree $d:=\Delta(\C G)$ 
		satisfies $\chi_{\C B}(G)\le d+1$.
	\end{theorem}
	
	\begin{proof} One way to prove this result with what we already have is to iteratively keep removing Borel maximal independent sets from $\C G$ that exist by Theorem~\ref{th:MaxInd}. (Here we use Corollary~\ref{cr:SubB} to show that each new graph is Borel; alternatively, we could have removed only edges touching the current independent set while keeping $V$ unchanged.) Then the degree of each remaining vertex strictly decreases during each removal. Thus, after $d$ removals, every remaining vertex is isolated and, after $d+1$ removals, the vertex set becomes empty.
		
		Alternatively, we can take a countable partition  $V=\cup_{i\in \omega} V_i$ into Borel independent sets  given by Lemma~\ref{lm:CountableCol} and, iteratively for $i\in\omega$, 
		colour all vertices of the independent set $V_i$ in parallel, using the smallest available colour on each. Clearly, we use at most $d+1$ colours while an easy inductive argument on $i$ combined with Lemma~\ref{lm:BorelG}\ref{it:BorelG2} shows that the obtained colouring is Borel on each set~$V_i$.\end{proof}

	Here is a useful consequence of the above results. (Recall that we identify a non-negative integer $k$ with the set~$\{0,\dots,k-1\}$.)
	
	\begin{cor}\label{cr:RSparse} For every Borel graph $\C G=(V,E,\C B)$ of finite maximum degree $d$ and every integer $r\ge 1$ there is a Borel colouring $c:V\to k$ with $k:=1+d\sum_{i=0}^{r-1} (d-1)^i$ such that every colour class is $r$-sparse.\end{cor}
	
	\begin{proof} Apply Theorem~\ref{th:ChiB} to the $r$-th power $\C G^r$, which is a Borel graph by Corollary~\ref{cr:Gr} and, trivially, has maximum degree at most $k-1$.%
		\hide{
			Alternatively, if the reader prefers to work with $\C G$ directly, then this amounts to the following modifications of the proof of Theorem~\ref{th:ChiB}. First, define $U(x)$ to be the smallest prefix of $I(x)$ that distinguishes $x$ from every $y\not=x$ at distance at most $r$ from~$x$ in~$\C G$. Then point preimages of $U$ give a countable Borel colouring of $\C G$ with $r$-sparse sets. In the second part, when processing each $W_i$, we colour each vertex $x\in W_i$ by the smallest available colour (given the colouring of $\cup_{j\in i} W_j$)}\end{proof}

	The set of edges $E\subseteq V^2$ is a Borel set, so $E$ is itself a standard Borel space by Corollary~\ref{cr:SubB}. In particular, it makes sense to talk about Borel edge $k$-colourings, meaning symmetric Borel functions $E\to k$.
	(Alternatively, one could have defined a Borel edge $k$-colouring as a  symmetric Borel function $c:V^2\to \{-1\}\cup k$ with $c(x,y)\ge 0$ if and only if $(x,y)\in E$, thus eliminating the need to refer to Corollary~\ref{cr:SubB} here.)

	The above results on independent sets and vertex colouring extend to matchings and edge colourings as follows.
	
	\begin{lemma}\label{lm:CountableM}
		The edge set of every locally finite Borel graph $\C G=(V,E,\C B)$ can be partitioned into countably many Borel matchings.\end{lemma}
	
	\begin{proof} Corollary~\ref{cr:Gr} and Lemma~\ref{lm:CountableCol}
		give a proper Borel vertex colouring $c:V\to\omega$ of~$\C G^2$, the square of~$\C G$. Thus each colour class $V_i:=c^{-1}(i)$, $i\in \omega$, is 2-sparse in~$\C G$.
		It follows that, for each pair $i<j$ in $\omega$, the set $M_{ij}:=E\cap ((V_i\times V_j)\cup (V_j\times V_i))$
		is a matching. Moreover, since each $V_i$ is independent in $\C G$, the Borel matchings $M_{ij}$ over all $i<j$ in $\omega$ partition $E$, as required.\end{proof}

	\begin{theorem}\label{th:MaxM}
		Every locally finite Borel graph $\C G=(V,E,\C B)$ has a maximal matching $M\subseteq E$ which is a Borel subset of $V^2$.\end{theorem}
	\begin{proof} Let $M_i'$, $i\in\omega$, be the matchings returned by Lemma~\ref{lm:CountableM}. We construct $M$ greedily, by taking for each $i\in\omega$ all edges in $M_i'$ that are vertex disjoint from the current matching. 
		
		Formally, let $M_{0}:=\emptyset$ and, inductively for $i\in\omega$, define
		$$
		M_{i+1}:=M_{i}\cup (M_i'\setminus \big((\pi_0(M_{i})\times V)\cup (V\times \pi_0(M_{i}))\big).
		$$

		As in Theorem~\ref{th:MaxInd}, the set $M:=\cup_{i\in\omega} M_i$ is a maximal matching in~$\C G$. Also, each $M_{i+1}$ (and thus the final matching $M$) is Borel, which can argued by induction on $i$ (using Theorem~\ref{th:LNU} to show that $\pi_0(M_{i})$ is Borel).
	\end{proof}

	The \emph{Borel chromatic index $\chi_{\C B}'(\C G)$} 
	is the smallest  $k\in\omega\cup\{\omega\}$ such that there exists a Borel map $c:E\to k$ with no two intersecting edges having the same colour (equivalently, with each colour class being a matching). 
	Similarly to how Theorem~\ref{th:ChiB} was derived from Theorem~\ref{th:MaxInd}, the following result can be derived from Theorem~\ref{th:MaxM} by removing one by one maximal Borel matchings
	and observing that for every remaining edge the number of other edges that intersect it strictly decreases with each removal step.

	\begin{theorem}[Kechris et al.~\cite{KechrisSoleckiTodorcevic99}]
		\label{th:ChiB'} 
		Every Borel graph $\C G$ with finite maximum degree $d:=\Delta(\C G)$ 
		satisfies $\chi_{\C B}'(G)\le 2d-1$.
		\qed\end{theorem}

	\begin{remark}
		\label{rm:LineGraph} Lemma~\ref{lm:CountableM} and Theorems~\ref{th:MaxM}--\ref{th:ChiB'} can be also deduced by applying the corresponding results on independent sets to the \emph{line graph} $L(\C G)$ whose vertex set consists of unordered pairs $\{x,y\}$ with $(x,y)\in E$, where two distinct pairs are adjacent if they intersect. Let us just outline a proof that $L(\C G)$ is a Borel graph. 
		Recall the definition of the Borel $\sigma$-algebra on ${V\choose 2}$ from Remark~\ref{rm:VChoose2}.
		From this definition, it follows that the vertex set of the line graph $L(\C G)$ is a Borel subset of ${V\choose 2}$ and thus is itself a standard Borel space by Corollary~\ref{cr:SubB}.
		To show that the  line graph is Borel, we check Property~\ref{it:BorelG3} of Lemma~\ref{lm:BorelG}. By lifting all to $V^2$, it is enough to check a version of this property for every symmetric Borel set $A\subseteq E$. Now, $Y:=\pi_0(A)=\pi_1(A)$, the set of vertices covered by the edges in $A$, is Borel by Theorem~\ref{th:LNU} as the sizes of preimages under $\pi_0:A\to Y$ are finite. The $1$-ball of $A$ in the line graph corresponds to the set of edges of $\C G$ that intersect~$Y$. The latter set is equal to $E\cap ((Y\times V)\cup (V\times Y))$ and is thus Borel. Finally (assuming we have verified all steps above), we can conclude by Lemma~\ref{lm:BorelG} that $L(\C G)$ is Borel.\ecpf\end{remark}
	
	\begin{remark}
		Observe that there are locally countable Borel graphs that 
		do not admit a Borel proper vertex colouring with countably many colours (see e.g.\ \cite[Examples 3.13--16]{KechrisMarks:survey}); such graphs were completely characterised by Kechris et al.~\cite[Theorem 6.3]{KechrisSoleckiTodorcevic99} as containing a certain obstacle. On the other hand, Kechris et al.~\cite[Proposition~4.10]{KechrisSoleckiTodorcevic99} observed that,
		by the Feldman--Moore Theorem (Theorem~\ref{th:FeldmanMoore} here) 
		the statement of  Lemma~\ref{lm:CountableM} (and thus of Theorem~\ref{th:MaxM}) remains true also in the locally countable case.\ecpf\end{remark}
	
	\subsection{Local Rules}
	\label{se:Local}
	
	We will show in this section that all ``locally defined'' vertex labellings are Borel as functions. As a warm up,
	consider the degree function $\deg:V\to\omega$, which sends a vertex $x\in V$ to its degree $\deg(x)=|N(x)|$.
	
	\begin{lemma}\label{lm:deg} For every locally finite Borel graph $\C G=(V,E,\C B)$, the degree function $\deg_{\C G}:V\to\omega$ is Borel.\end{lemma}
	\begin{proof} 
		It is enough to show that for every $k\in\omega$ the set $D_k:=\{x\in V\mid \deg(x)\ge k\}$ is Borel, because the set of vertices of degree exactly $k$ is $D_k\setminus D_{k+1}$. 
		
		The most direct proof
		is probably to use a countable generating algebra $\C J=\{J_i\mid i\in\omega\}$ from Remark~\ref{rm:GenAlg}. Note that a vertex $x\in V$ has degree at least $k$ if and only if there are $k$ pairwise disjoint sets in $\C J$ with $x$ having at least one neighbour in each of them.
		Indeed, if $y_0,\dots,y_{k-1}\in N(x)$ are pairwise distinct then by Lemma~\ref{lm:Sep} there are $A_0,\dots,A_{k-1}\in\C J$ with $A_i\cap\{y_0,\dots,y_{k-1}\}=\{y_i\}$ for each $i\in k$ and if we let $B_0:=A_0$, $B_1:=A_1\setminus A_0$, $B_2:=A_2\setminus (A_0\cup A_1)$, and so on, then $B_0,\dots,B_{k-1}\in\C J$ have the required properties. 
		Thus, we can write $D_k$ as the countable union of the intersections $\cap_{m\in k} N(J_{i_m})$ over all $k$-tuples $i_0,\dots,i_{k-1}$ such that $J_{i_0},\dots,J_{i_{k-1}}$ are pairwise disjoint. By Lemma~\ref{lm:BorelG}\ref{it:BorelG2}, all neighbourhoods $N(J)$ for $J\in\C J$ and thus the set $D_k$ are Borel.
		
		Alternatively, fix any Borel proper edge colouring $c:E\to\omega$ which exists by Theorem~\ref{th:ChiB'}.
		Trivially, $\deg(x)\ge k$ if and only if there are at least $k$ distinct colours under $c$ at the vertex $x$. The set of vertices that belong to an edge of colour $i$ is $\pi_0(c^{-1}(i))$, which is Borel by Theorem~\ref{th:LNU}. Thus
		$$
		D_k=\bigcup_{i_0,\ldots,i_{k-1}\in\omega\atop i_0<\ldots<i_{k-1}}\ \bigcap_{m\in k}\ \pi_0(c^{-1}(i_m))
		$$
		is a Borel set.
		\hide{(Yet another proof of the lemma can be obtained along similar lines except fixing a 2-sparse Borel vertex colouring.)}
	\end{proof}

	In order to make the forthcoming general statement (Lemma~\ref{lm:LocalRule}) stronger and better suitable for applications, we consider a version where we may have some additional structure on graphs.
	Namely, a \emph{labelling} of a graph $G$ is any function $\ell$ from $V$ to some countable set; then we say that a pair $(G,\ell)$ is a \emph{labelled graph}.
	Vertex labellings allow us to encode many other types of structures on $G$ such as, for example, edge colourings (see Remark~\ref{rm:Reduction} for a reduction).  
	
	Let $\ell$ be a labelling of a  graph~$G=(V,E)$. For $r\in \omega$, let $\RR F_r$ be the function on $V$ which sends a vertex $x\in V$ to the isomorphism type of
	$$\induced{(G,c,x)}{N^{\le r}(x)}:=(\induced{G}{N^{\le r}(x)},\,\induced{c}{N^{\le r}(x)},\,x),$$
	the labelled graph induced by the $r$-ball $N^{\le r}(x)$ in $G$ rooted at $x$, where isomorphisms have to preserve also the root and the vertex labelling. Since we consider only locally finite graphs, $\RR F_r$ assumes countably many possible values and thus is an example of a labelling. 
	By a \emph{local rule of radius $r$} (or an \emph{$r$-local rule}) on $(G,\ell)$ we mean a function $\RR R$ on $V$ whose value at any $x\in V$ depends only on $\RR F_r(x)$. In other words, $r$-local rules are exactly those functions that factor through $\RR F_r$, that is, are representable as
	a composition $f\circ \RR F_r$ for some function~$f$. A function on $V$ is a \emph{local rule} if it is an $r$-local rule for some $r\in\omega$. Unless stated otherwise, we assume that local rules and labellings are functions from $V$ to $\omega$, that is, their values are non-negative integers.
	
	For example, the degree function $\deg$ or the number of triangles that contain a vertex are local rules of radius 1
	(that do not depend on the labelling). An example of a 1-local rule that uses the labelling $\ell$ is, say, 
	$x\mapsto|\ell(N(x))|$, the number of distinct $\ell$-labels on the neighbours of~$x$.

	Let $G$ be a graph with a labelling $\ell:V\to\omega$ such that, for every $x\in V$, its neighbours get pairwise distinct colours. 
	Fix some special element not in $V$, denoting it by~$\perp$. 
	For a non-empty sequence $S=(s_0,\dots,s_{j})\in \omega^{j+1}$ of labels, let us
	define a function $f_{S}:V\to V\cup\{\perp\}$ as follows. 
	Take any $x\in V$. If there is a \emph{walk} in $\C G$ of (edge) length $j$ (i.e.\ a sequence $(x_0,\dots,x_j)$ with $(x_i,x_{i+1})\in E$ for each $i\in j$) that starts with $x$ (i.e.\ $x_0=x$) and is \emph{$S$-labelled} (i.e.\ $\ell(x_i)=s_i$ for each $i\in j+1$) then let $f_{S}(x):=x_j$ be the final endpoint of this walk; otherwise let $f_{S}(x):=\,\perp$.
	By our assumption on $\ell$, there can be at most one such walk, so $f_S(x)$ is well-defined. 
	(Equivalently, we could have worked with partially defined functions,
	instead of using the special symbol~$\perp$.)
	For convenience, if $S=()$ is the empty sequence, then we define $f_{()}$ 
	to be the identity function on~$V$.

	\begin{lemma}\label{lm:fS} If $\C G$ is a locally finite Borel graph with a Borel labelling $\ell:V\to\omega$ that is injective on $N(x)$ for every $x\in V$, then for every $j\in\omega$ and every sequence $S=(s_0,\dots,s_{j})$ in $\omega^{j+1}$ the function $f_S:V\to V\cup\{\perp\}$ is Borel.
	\end{lemma}
	
	\begin{proof} Regardless of how we extend a Polish topology from $V$ to $V\cup\{\perp\}$, a subset $A$ of $V\cup\{\perp\}$ is Borel if and only if $A\cap V$ is Borel. Thus it is enough to check that the preimage $f_S^{-1}(V)$ is Borel and the restriction of $f_S$ to $f_S^{-1}(V)$ is a Borel function.
		
		We use induction on $j\in\omega$. If $j=0$, then $f_{(s_0)}$ is the identity function on the Borel set $\ell^{-1}(s_0)$ and assumes value $\perp$ otherwise; so $f_{(s_0)}$ is indeed Borel. Suppose that $j\ge 1$. Let $f:=f_{(s_0,\dots,s_j)}$ and $g:=f_{(s_0,\dots,s_{j-1})}$. 
		
		Observe that, for every $x\in V$, there is an $(s_0,\dots,s_j)$-labelled walk starting at $x$ if and only if there is an $(s_0,\dots,s_{j-1})$-labelled walk starting at $x$ and its endpoint $g(x)$ has a neighbour labelled~$s_j$. That is, 
		$$
		f^{-1}(V)= g^{-1}(V)\cap g^{-1}(N(\ell^{-1}(s_j))),
		$$
		and this set is Borel by induction and Lemma~\ref{lm:BorelG}\ref{it:BorelG2}. Let $Y:=g(g^{-1}(V))$ consist of the endpoints of all $(s_0,\dots,s_{j-1})$-labelled walks in~$\C G$. This set is Borel by Theorem~\ref{th:LNU} as the bijective image under the Borel map $g$ of the Borel set $g^{-1}(V)$. Let $Y'$ consist of those vertices in $Y$ that have a neighbour labelled $s_j$. Again by Theorem~\ref{th:LNU}, $Y'$ is Borel as the bijective image of 
		$$
		Y'':=\{(y,z)\in E\mid y\in Y\ \wedge\ \ell(z)=s_j\}=(Y\times \ell^{-1}(s_j))\cap E
		$$
		under the projection $\pi_0$ on the first coordinate. Moreover, the map $h:Y'\to Y''$ which is the (unique) right inverse of $\pi_0$ is Borel by the second part of Theorem~\ref{th:LNU}. The composition $\pi_1\circ h$ sends an element of $Y'$ to its unique neighbour labelled~$s_j$. Thus the function $f$ is Borel since, on the Borel set $f^{-1}(V)$, it is the composition $\pi_1\circ h\circ g$ of three Borel functions.\end{proof}

	\begin{lemma}\label{lm:LocalRule} Let $\C G=(V,E,\C B)$ be a locally finite Borel graph with a Borel labelling $\ell:V\to\omega$. Then every local rule $\RR R:V\to\omega$ on $(\C G,\ell)$ is a Borel function.\end{lemma}

	\begin{proof} Let the local rule $\RR R$ have radius~$r$. We can additionally assume that $\ell$ is a $2r$-sparse colouring of~$\C G$. Indeed, take any Borel proper vertex colouring $c:V\to\omega$ of $\C G^{2r}$ (which exists by Corollary~\ref{cr:Gr} and Lemma~\ref{lm:CountableCol}), replace $\ell$ by the labelling $(\ell,c):V\to \omega^2$, which maps a vertex $x\in V$ to $(\ell(x),c(x))$,  and update the local rule $\RR R$ to ignore the $c$-component of the labelling.
		
		To prove the lemma, it is enough to show that the function $\RR F_r$ on $V$
		is Borel because each preimage under $\RR R$ is a countable union of some preimages under $\RR F_r$. Note that $\RR F_0$ is Borel since the (countable) vertex partition defined by $\RR F_0$ is the same as the partition defined by the Borel function $\ell$. So assume that $r\ge 1$.
		
		Fix any particular feasible $\RR F_r$-value $\C F$ (a rooted labelled graph with each vertex at distance at most $r$ from the root). By relabelling vertices, assume that the vertex set of $\C F$ is~$k$ with $0$ being the root. Let $\ell':k\to\omega$ be the vertex labelling of~$\C F$. We assume that the function $\ell':k\to\omega$ is injective as otherwise $\RR F_r^{-1}(\C F)$ is empty and thus trivially Borel.

		For each $i\in k$, let $(s_0,\dots,s_j)$ be the sequence of labels on some fixed shortest path $P_i$ from the root $0$ to $i$ in~$\C F$ and let $f_i:=f_{(s_0,\dots,s_j)}$, be the function
		defined before Lemma~\ref{lm:fS} with respect to the labelled graph~$(\C G,\ell)$.

		Observe that, for $x\in V$, the $r$-ball $\RR F_r(x)$ is isomorphic to $\C F$ if and only if
		all the following statements hold:
		\begin{enumerate}[(a),nosep]
			\item\label{it:LocalRuleA} For every $i\in k$, we have $f_i(x)\not=\,\perp$ (that is, $\C G$ has a walk starting from $x$ labelled the same way as~$P_i$ is).  
			\item\label{it:LocalRuleC} For every distinct $i,j\in k$, the pair $(i,j)$ is an edge in $\C F$ if and only if $(f_i(x),f_j(x))$ is an edge of~$\C G$.
			\item\label{it:LocalRuleE} For every $i\in k$ with distance at most $r-1$ from the root in $\C F$ and every $y\in N_{\C G}(f_i(x))$ there is $j\in k$ with $y=f_j(x)$ and $j\in N_{\C F}(i)$.
		\end{enumerate}
		
		Indeed,  the map $h:k\to V$ that sends $i$ to $f_i(x)$ preserves the labels by Property~\ref{it:LocalRuleA} since we considered labelled walks when defining each~$f_i$. Furthermore, $h$ is injective; in fact, even the composition $\ell\circ h$ is injective since the endpoints of the paths $P_i$ have distinct $\ell'$-labels as distinct vertices of~$\C F$. Now, Property~\ref{it:LocalRuleC}
		states that $h$ is a graph isomorphism from $\C F$ to the subgraph induced by $\{f_0(x),\dots,f_{k-1}(x)\}$ in $\C G$. Finally, Property~\ref{it:LocalRuleE} states that 
		the breadth-first search of depth $r$ from $x$ in $\C G$ does not return any vertices not accounted by~$\C F$.
		
		Let $X$ consist of those $x\in V$ that satisfy Property~\ref{it:LocalRuleA}. By Lemma~\ref{lm:fS}, the set $X=\cap_{i\in k} f_i^{-1}(V)$ is Borel.
		
		Note that, for any distinct $i,j\in k$, the set
		\begin{eqnarray*}
			Y_{i,j} &:=& \{x\in X\mid (f_i(x),f_j(x))\in E\}\\
			&=&\pi_0\big(\,\{(x,x_i,x_j)\in X^3\mid x_i=f_i(x)\}\\
			&\cap& \{(x,x_i,x_j)\in X^3\mid \ x_j=f_j(x)\}\ \cap \ (X\times E)\,\big)
		\end{eqnarray*}
		is Borel by Lemma~\ref{lm:BorelF} (applied to the functions $f_i$ and $f_j$) and Theorem~\ref{th:LNU} (applied to the projection $\pi_0$, whose preimages are finite here as the graph $\C G$ is locally finite). Thus the set $Y$ of the elements in $X$ that satisfy Property~\ref{it:LocalRuleC} is Borel since it is the intersection of the sets $Y_{i,j}$ over all edges $\{i,j\}$ of $\C F$ and the sets $X\setminus Y_{i,j}$ over non-edges $\{i,j\}$ of~$\C F$.
		
		Finally, the set $\RR F_r^{-1}(\C F)$ that we are interested in is Borel by Lemma~\ref{lm:fS} as the countable intersection with $Y$ of $f_S^{-1}(\perp)$ over all sequences $S$ of labels of length at most $r$ that do not occur on walks in $\C F$ that start with the root.\end{proof}
	
	\begin{remark}
		\label{rm:Reduction}
		Note that many other types of structures on $\C G$, such as
		a Borel edge labelling $\ell':E\to\omega$, can be encoded via some vertex labelling $\ell$ to be used as the input to the local rule $\RR R$ in Lemma~\ref{lm:LocalRule}. Namely, fix  a Borel 2-sparse colouring $c:V\to\omega$ of $\C G$ and let the label $\ell(x)$ of a vertex $x$ be defined as the finite list $(c(x),\ell'(x,y_0),\dots,\ell'(x,y_{d-1}))$ where 
		$y_0,\dots,y_{d-1}$ are all neighbours of $x$ listed increasingly with respect to their $c$-colours. In fact, the same reduction also works for labellings $\ell':E\to\omega$ that need not be \emph{symmetric} (meaning that $\ell'(x,y)=\ell'(y,x)$ for all $(x,y)\in E$). More generally,
		any countably valued function defined on subsets of uniformly bounded diameter in $\C G$ can be locally encoded by a vertex labelling.
		\ecpf\end{remark}

	In some cases, a single local rule does not work but a good assignment can be found by designing a sequence of local rules of growing radii that eventually stabilise at every vertex. The author is not aware of any commonly used name for functions arising this way so, for the purposes of this paper, we  make up the following name (inspired by the term \emph{finitary factor} from probability). For a labelled graph $(G,\ell)$, let us call a function $\RR R:V\to \omega$  \emph{finitary} (with respect to $(G,\ell)$)  if there are local rules $\RR R_i$, $i\in\omega$, on $(G,\ell)$ such that the sequence of functions $\RR R_i:V\to \omega$ eventually stabilises to $\RR R$ everywhere, that is, for every $x\in V$ there is $n\in \omega$ such that for every $i\ge n$ we have $\RR R(x)=\RR R_i(x)$. Note that we do not require that $\RR R_i$ ``knows'' if its value at a vertex $x$ is the eventual value or not.

	\begin{example}\label{ex:PerfectM}
		Let $c:V\to \omega$ be a 2-sparse colouring of a graph $G$. For $x\in V$ let $\RR R(x)\in \omega$ be $0$ if the component of $x$ has no perfect matching and be the largest $i\in\{1,\dots,\deg(x)\}$ such that the component of $x$ has a perfect matching that matches $x$ with the $i$-th element of $N(x)$ (where we order each neighbourhood by looking at the values of $c$, which are pairwise distinct by the 2-sparseness of $c$). Then $\RR R:V\to \omega$ is finitary on $(G,c)$ as the following local rules $\RR R_r$, $r\in\omega$, demonstrate. Namely, $\RR R_r(x)$ is the largest $i\in \{1,\dots,\deg(x)\}$ such that there is a matching $M$ in $\induced{G}{N^{\le r}(x)}$ that covers every vertex in $N^{\le r-1}(x)$ and matches $x$ to the $i$-th element of $N(x)$; if no such $M$ exists then we let $\RR R_r(x):=0$. Trivially, for every $x\in V$, if we increase $r$ then $\RR R_r(x)$ cannot increase. Thus the values $\RR R_r(x)$, $r\in\omega$, eventually stabilise and, moreover, this final value can be easily shown to be exactly~$\RR R(x)$. On the other hand, the function $\RR R$ is not $r$-local for any $r\in\omega$: e.g.\ a vertex at distance at least $r$ from both endpoints of a finite path cannot decide by looking at distance at most $r$ if a perfect matching exists or not.
	\end{example}
	
	\begin{cor}\label{cr:Finitary} For every locally finite Borel graph $\C G=(V,E,\C B)$ with a Borel labelling $\ell:V\to\omega$, every function $\RR R:V\to\omega$ which is finitary with respect to $(\C G,\ell)$ is Borel.
	\end{cor}
	\begin{proof} Fix local rules $\RR R_i:V\to\omega$, $i\in\omega$, that witness that $\RR R$ is finitary. By Lemma~\ref{lm:LocalRule}, each $\RR R_i$ is a Borel function. The function $\RR R:V\to\omega$ is
		Borel since a pointwise limit of Borel functions is Borel (\Co{Proposition 2.1.5}). Alternatively, the last step follows from observing that, for every possible value $j\in \omega$,
		its preimage $\RR R^{-1}(j)
		= \cup_{i\in\omega} \cap_{m\ge i} \RR R_m^{-1}(j)$ is the $\liminf$ (and also, in fact, $\limsup$) of the preimages~$\RR R_i^{-1}(j)$, $i\in\omega$.\end{proof}

	\subsection{Graphs with a Borel Transversal}
	\label{se:Smooth}

	In this section we present one application of Lemma~\ref{lm:LocalRule} (namely, Theorem~\ref{th:Smooth} below) which, informally speaking, says that if we can pick exactly one vertex inside each graph component in a Borel way then any satisfiable locally checkable labelling problem has a Borel satisfying assignment.

	We define
	a \emph{locally checkable labelling problem} 
	(or an \emph{LCL} for short) 
	on labelled graphs to be a $\{0,1\}$-valued local rule $\RT$ that takes as input graphs with $\omega^2$-valued labellings. 
	A labelling $a:V\to \omega$ \emph{satisfies} (or \emph{solves}) the LCL $\RT$ on a labelled graph $(G,\ell)$ if $\RT$ returns value $1$ on every vertex when applied to the labelled graph $(G,(\ell,a))$. (Recall that the labelling $(\ell,a):V\to \omega^2$ labels a vertex $x$ with $(\ell(x),a(x))$.)
	We will identify the function $\RT(G,(\ell,a)):V\to\{0,1\}$ with the set of $x\in V$ for which it assumes value $1$ and, if the ambient labelled graph is understood, abbreviate this set to $\RT(a)$. Thus $a$ satisfies $\RT$ if and only if $\RT(a)=V$. 
	We call the requirement that a vertex $x$ belongs to $\RT(\cdot)$ the \emph{constraint at~$x$}. Labellings $a:V\to \omega$ on which we check the validity of an LCL will usually be called \emph{assignments}.

	Let us give a few examples of LCLs.
	First, checking that $a:V\to \omega$ is a proper colouring can be done by the 1-local rule which returns 1 if and only if the colour of the root is different from the colour of any of its neighbours (and that the colour of the root is in the set $n=\{0,\dots,n-1\}$ if we additionally want to require that $a$ uses at most $n$ colours). Note that the LCL ignores the labelling $\ell$ in this example.
	Second, suppose that a labelling $a$ encodes some edge colouring $c:E\to\omega$ as in Remark~\ref{rm:Reduction}. Then a 2-local rule can check if $c$ is proper: each vertex $x$ checks that all colours on ordered pairs $(x,y)\in E$ are pairwise distinct and that  $c(x,y)=c(y,x)$ for every $y\in N(x)$; for this the knowledge of $
	N^{\le 2}(x)$ is enough. Another example is checking that $a:V\to\{0,1\}$ is the indicator function of a maximal independent set: this can be done by the obvious 1-local rule.
	
	Let $\C G$ be a Borel graph. A \emph{transversal} for $\C G$ is a subset $X\subseteq V$ such that $X$ has exactly one vertex from every connectivity component of $\C G$ (i.e.\ $|X\cap [x]\OptionalCG |=1$ for every $x\in V$).  In the language of Borel equivalence relations, the existence of a Borel transversal can be shown (see e.g.~\cite[Proposition 6.4]{KechrisMiller:toe} or~\Ts{Proposition 20.5}) to be equivalent to the statement that the connectivity relation $\C E\OptionalCG $ of $\C G$ (which is Borel by Corollary~\ref{cr:CEBorel}) is \emph{smooth}, meaning that there is a countable family of Borel sets $Y_n\subseteq V$, $n\in \omega$,
	such that, for all $x,y\in V$, we have $(x,y)\in \C E\OptionalCG $ if and only if every $Y_n$ contains either both or none of~$x$ and $y$. 
	
	Here is an example of a large class of Borel graphs that have a Borel transversal.
	
	\begin{lemma}\label{lm:FinCompSmooth}
		Every (locally finite) Borel graph $\C G$ with all components finite  admits a Borel transversal.
	\end{lemma}
	
	\begin{proof} Fix any Borel total order $\preceq$ on the vertex set $V$ which exists by Lemma~\ref{lm:I} and define a transversal $X$ by picking the $\preceq$-smallest element in each component. Note that $X$ is the countable union of the sets $X_r$, $r\in\omega$, where $X_r$ consists of the $\preceq$-smallest vertices inside connectivity components of diameter at most~$r$. Each $X_r$ is Borel by Lemma~\ref{lm:LocalRule} as its indicator function can be computed by an $(r+1)$-local rule. 
		Thus the constructed transversal $X$ is also Borel. \end{proof}

	Having a Borel transversal is a very strong property from the point of view of Borel combinatorics as the forthcoming results state. The first one is a useful auxiliary lemma stating, informally speaking, that if we can pick exactly one vertex in each component in a Borel way then we can enumerate each component in a Borel way. 
	
	\begin{lemma}\label{lm:SmoothOrdering}
		Let  $\C G=(V,E,\C B)$ be a locally finite Borel graph, admitting a Borel transversal~$X\subseteq V$. Then there are Borel functions $g_i:X\to V$, $i\in\omega$, such that $g_0$ is the identity function on $X$ and, for every $x\in X$, the sequence $(g_i(x))_{i\in m}$ bijectively enumerates $[x]_{\C G}$ and satisfies $\dist_{\C G}(x,g_i(x))\le \dist_{\C G}(x,g_j(x))$ for all $i<j$ in~$m:=|[x]_{\C G}|$.\end{lemma}
	
	\begin{proof} 
		Fix a 2-sparse Borel colouring $c:V\to\omega$ which exists by Corollary~\ref{cr:RSparse}.
		
		The main idea of the proof is very simple: for each selected vertex $x\in X$ 
		we order the vertices in the connectivity component of $x$ first by the distance to $x$ and then by the $c$-labellings of the shortest paths from $x$ to them, and  let $g_i(x)$ be the $(i+1)$-st vertex in this order on~$[x]\OptionalCG $.

		Formally, let $\preceq$ be the total ordering of  $\omega^{<\omega}$, the set of all finite sequences of non-negative integers, first by the length and then lexicographically. Note that $\preceq$ is a \emph{well-order} on  $\omega^{<\omega}$ (that is, every non-empty subset of $\omega^{<\omega}$ has the $\preceq$-smallest element). For $S\in\omega^{<\omega}$, let $f_{S}:V\to V\cup\{\perp\}$ be the function defined before Lemma~\ref{lm:fS} (which sends $x\in V$ to the other endpoint of the $S$-coloured walk starting at $x$, if it exists). By Lemma~\ref{lm:fS}, each function $f_{S}$ is Borel. 

		We define functions $g_i$ inductively, with $g_0:=\induced{f_{()}}{X}$ being the identity function on~$X$. Suppose that $i\ge 1$. 
		For $x\in V$, let $g_i(x)$ be equal to $f_{S}(x)$ where $S\in\omega^{<\omega}$ is the $\preceq$-smallest sequence with $f_{S}(x)\in [x]\OptionalCG \setminus\{g_0(x),\dots,g_{i-1}(x)\}$ and let $g_i(x):=x$ if no such $S$ exists (i.e.\ if we have already exhausted the whole component of~$x$).
		
		Let us argue by induction on $i\in\omega$ that the function $g_i$ is Borel. As $g_0$ is clearly Borel, take any $i\ge 1$. For $S\in\omega^{<\omega}$, let $X_S:=f_{S}^{-1}(V\setminus \cup_{j\in i} g_j(X))\cap X$.
		For any given $S\in\omega^{<\omega}$, the set of $x\in X$ for which we use $f_S$  when defining $g_i(x)$ is exactly $X_S\setminus (\cup_{R\prec S} X_{R})$, since we have to exclude the already labelled vertices $g_0(x),\dots,g_{i-1}(x)$ and then any $R\prec S$ as it takes precedence over~$S$. Each set $X_S$ for $S\in\omega^{<\omega}$ is Borel by Theorem~\ref{th:LNU} and Lemma~\ref{lm:fS}. It follows that the function $g_i$ is Borel.
		The other claimed properties of the constructed functions $g_i$ are obvious from the definition.\end{proof}

	\begin{theorem}\label{th:Smooth} Let  $\C G=(V,E,\C B)$ be a locally finite Borel graph, admitting a Borel transversal~$X\subseteq V$. 
		Let $\ell:V\to \omega$ be a Borel labelling, $n\in \omega$ and $\RT$ be an LCL. 
		If there is an assignment $V\to n$ that solves $\RT$  on $(\C G,\ell)$,
		then there is a Borel assignment $V\to n$ that solves~$\RT$ on $(\C G,\ell)$.\end{theorem}
	
	\begin{proof} As in the proof of Lemma~\ref{lm:LocalRule}, by replacing $\ell$ with $(\ell,d)$ for some 2-sparse Borel colouring $d$ (and letting the updated LCL $\RT$ ignore the $d$-component), we can additionally assume that $\ell$ is a 2-sparse colouring of~$\C G$. 
		Let $g_i:X\to V$ for $i\in\omega$ be the Borel functions returned by Lemma~\ref{lm:SmoothOrdering}.

		Since we need to use these functions $g_i$ as inputs to in our local rules, we encode them by a vertex labelling~$I$ as follows. Namely, we define $I:V\to \omega$ to map $y\in V$ to the smallest $i\in\omega$ with $y\in g_i(X)$. Thus $I$ bijectively enumerates each component of $\C G$ by an initial interval of~$\omega$.
		This function is Borel by Theorem~\ref{th:LNU} since $I^{-1}(0)=X$ and
		$$
		I^{-1}(i)=g_i(X)\setminus \{x\in X\mid |[x]\OptionalCG|\le i\},\quad \mbox{for each }i\ge 1.
		$$ 
		(Note that, for every $i\in\omega$, 
		the set of $x\in X$ whose connectivity component has at most $i$ vertices  is Borel by Lemma~\ref{lm:LocalRule} since its indicator function can be computed by an $i$-local rule on $(\C G,\I 1_X)$.)

		Let $\rho$ be the radius of the local rule~$\RT$.
		
		For every $r\in\omega$, we define an $r$-local rule $\RR A_r:V\to \{-1\}\cup n$
		which on $(\C G,(\ell,I))$ works as follows. Given $y\in V$, explore
		$N^{\le r}(y)$, the $r$-ball around $y$. If it contains no vertex of $X$, then let $\RR A_r(y):=-1$ (which could be interpreted that the vertex $y$ does not yet make any guess of its value in the set~$n$). 
		Otherwise, let $x$
		be the (unique) vertex from $X$ that we have encountered. Let 
		$$
		k=k(r,y):=r-\dist(x,y)\mbox{\ \  and\ \  }Y=Y(r,y) :=N^{\le k}(x).
		$$
		If $y\not\in Y$, then we define $\RR A_r(y):=-1$. Suppose that $y\in Y$.
		Note that $Y=\{y_0,\dots,y_{m-1}\}$, where $m:=|Y|$ and $y_i$ for $i\in m$ is defined as the unique vertex of $Y$ labelled $i$ by $I$ (which is, of course,~$g_i(x)$). Since $Y\subseteq N^{\le r}(y)$, the local rule $\RR A_r(y)$ can identify all these vertices.
		Among all assignments $A:Y\to n$, pick one, denoting it $A=A(r,y)$, that satisfies $\RT$ on every vertex of 
		$$
		Y'=Y'(r,y):=N^{\le k-\rho}(x)
		$$ 
		and, if there is more than one choice then choose the one 
		for which the sequence $(A(y_0),\dots,A(y_{m-1}))$ is lexicographically smallest.
		(Note that there is always at least one choice of $A$ by our assumption that a global solution exists.) Finally, let $\RR A_r(y):=A(y)$, be the value of $A$ on the vertex $y\in Y$. Note that in order to compute $A$ (given $y_i$'s), we need to know only the labelled graph induced by $N^{\le \rho}(Y')\subseteq Y\subseteq N^{\le r}(y)$. Thus $\RR A_r$ is indeed a local rule of radius~$r$.
		
		Informally speaking, these rules $\RR A_r$ are constructed so that each vertex $x$ from the  transversal $X$ takes growing balls around itself and properly labels each with the lexicographically smallest assignment that looks satisfiable, given the current local information. 
		Every other vertex $y$ in the component of $x$ has to follow these choices once $y$ discovers enough information to compute $x$'s choice for the value at~$y$.
		
		Let us show that the values of $\RR A_r$, $r\in\omega$, eventually stabilise to some element of $n$ on each vertex~$y\in V$, by using induction on $I(y)$. Take any $x\in X$. Define $y_i:=g_i(x)$ for $i\in \omega$. Note that $y_0$ is the special vertex $x\in X$. Take any $i\in\omega$. First, observe that, trivially, $\RR A_r(y_i)\not=-1$ for all $r\ge 2\,\dist(y_0,y_i)+\rho$. By induction pick $r_0\in\omega$ such that $\RR A_r$ stabilises from $r_0$ on each of $y_0,\dots,y_{i-1}$, that is, for all $r\ge r_0$ the restrictions of $\RR A_r$ to $\{y_0,\dots,y_{i-1}\}$ are equal to each other. Consider the values $\RR A_r(y_i)$ for $r\ge r_0'$ where 
		$$
		\mbox{$r_0':=r_0+d$ and $d:=\max\{\dist(y_i,y_j)\mid j\in i\}$.}
		$$ 
		Using the notation from the definition of $\RR A_r$, it holds that $k(r,y_i)=r-\dist(y_0,y_i)$ is at least as large as $k(r_0,y_j)=r_0-\dist(y_0,y_j)$ for each $j\in i$. Thus, for every $j\in i$, we have that $Y'(r,y_i)\supseteq Y'(r_0,y_j)$, as these sets are just balls around the special vertex $x=y_0$ of radii $k(r,y_i)-\rho$ and $k(r_0,y_j)-\rho$ respectively.
		Thus, when we compute $A=A(r,y_i)$, the constraints that this assignment has to satisfy include all constraints for~$A(r_0,y_j)$. In the other direction, it holds for all $j\in i$ again by our choice of $d$ that $Y'(r+d,y_j)\supseteq Y'(r,y_i)$, that is, $A(r+d,y_j)$ has to satisfy all constraints that are imposed on~$A(r,y_i)$. Since we always go for the lexicographically minimal assignment, we conclude that $A(r,y_i)$ coincides with $\RR A_{r_0}$ on $\{y_0,\dots,y_{i-1}\}$. 
		Thus, for all $r\ge r_0'$, when we compute $A(r,y_i)$, we can equivalently view its values on $\{y_0,\dots,y_{i-1}\}$ as fixed and, given this, we minimise the value at~$y_i$. This value cannot decrease when we increase $r$ (because  any increase of the radius $r$ just adds some extra constraints on $A(r,y_i)$). So the values at each $y\in V$ eventually stabilise as they come from the finite set $n$, as desired.

		We define the final assignment $a$ as the one to which the local rules $\RR A_r$, $r\in\omega$, stabilise. It is finitary and thus Borel by Corollary~\ref{cr:Finitary}.

		It remains to check that the constructed assignment $a:V\to n$ solves the LCL~$\RT$. Take any~$y\in V$. Its $\rho$-ball $Z:=N^{\le \rho}(y)$ is finite and thus there is $r_0\in\omega$ such that $\induced{a}{Z}=\induced{\RR A_r}{Z}$ for each $r\ge r_0$. Let $x$ be the unique element of $[y]\OptionalCG \cap X$. Take any 
		$$
		r\ge \max(r_0,2\,\dist(x,y))+\rho.
		$$ 
		When we compute $\RR A_r(y)$, we have that $k(r,y)=r-\dist(x,y)$ is at least $\dist(x,y)+\rho$, so $y\in Y'(r,y)$, that is, the constraint at $y$ is one of the constraints that the partial assignment $A(r,y)$ has to satisfy. For each element $z\in N^{\le\rho}(y)$, we have $k(r,y)\ge k(r_0,z)$ and thus $Y(r,y)\supseteq Y(r_0,z)$. Recall that the final assignment $a$ coincides with $A(r,y)$ on $N^{\le \rho}(y)=Z$ by the choice of~$r_0$. Since the LCL $\RT$ of radius $\rho$ is satisfied by $A(r,y)$ at $y\in Y'(r,y)$, the assignment $a$ also satisfies the constraint at~$y$. As $y\in V$ was arbitrary, the constructed assignment $a$ solves the problem~$\RT$.\end{proof}
	
	\hide{
		Locally countable Borel graphs are somewhere halfway between locally finite Borel graphs and CBERs.  Some of the results that we presented here may be extended to locally countable graphs. For example, 
		Conley and Miller~\cite[Proposition~1]{ConleyMiller16mrl} proved that for countably local Borel graphs with a Borel transversal their chromatic number and Borel chromatic number coincide, which is extension of the special case (proper vertex colourings) of Theorem~\ref{th:Smooth}.
	}

	\begin{remark} Bernshteyn~\cite{Bernshteyn20u}  gave an example showing that Theorem~\ref{th:Smooth} is false when $n=\omega$, that is, when we consider assignments $V\to\omega$ that can assume can assume infinitely many values. (Also, Bernshteyn~\cite{Bernshteyn20u} presents an alternative proof of Theorem~\ref{th:Smooth}, via the Uniformization Theorem for compact preimages.)\end{remark}

	\subsection{Borel Assignments without Small Augmenting Sets}
	\label{se:ElekLippner}

	Suppose that we look for Borel assignments $a:V\to\omega$ on a labelled graph $(G,\ell)$ that solve an LCL $\RT$ where we want  many vertices to satisfy another LCL~$\RR P$.
	As before, we denote the set of vertices where $a$ 
	satisfies $\RR P$ by $\RR P(G,(\ell,a))$ (or by $\RR P(a)$ if the context is clear), calling it the \emph{progress set}.
	For an assignment $a$ that solves the LCL $\RT$, an \emph{$r$-augmenting set}
	is a set $R\subseteq V$ such that $R$ is \emph{connected} (meaning that $\induced{G}{R}$ is connected), $|R|\le r$, and  there is an assignment $b:R\to\omega$ such that $\replace{a}{b}$ solves $\RT$ and satisfies $\RR P(G,(\ell,a))\subsetneq \RR P(G,(\ell,\replace{a}{b}))$. Here, 
	\beq{eq:Replace}
	\replace{a}{b}:=b\cup (\induced{a}{(\pi_0(a)\setminus \pi_0(b))}
	\eeq
	denotes the function obtained from putting the functions  $a$ and $b$ together, with $b$ taking preference on their common domain $\dom(a)\cap \dom(b)$. 
	We call the new assignment $\replace{a}{b}$ an \emph{$r$-augmentation} of~$a$.
	Thus, an $r$-augmentation strictly increases the progress set by changing the current assignment on a connected set of at most $r$ vertices without violating any constraint of~$\RT$.

	As an example, suppose that $\RT$ states that $a$ encodes a set $M$ of edges which is a matching (that is, $\Delta(M)\le 1$) and $\RR P(x)=1$ means that a vertex $x\in M$ is matched by $M$. Similarly to Remark~\ref{rm:Reduction}, one can encode a matching $M$ by letting $a(x):=\deg(x)$ if $x$ is unmatched and otherwise letting $a(x)\in \deg(x)$ specify the unique $M$-match of $x$ by its position in $N(x)$ with respect to  the ordering of $N(x)$ coming from a fixed $2$-sparse Borel colouring. Then the LCL $\RT$ can be realised by a $2$-local rule (checking that the value at $x$ is consistent with the value at each $y\in N(x)$) while $\RR P$ is the 1-local rule that outputs 0 at $x$ if and only if $a(x)=\deg(x)$. Note that if we change the encoding so that $a(x)=0$ means that $x$ is unmatched while $a(x)\in \{1,\dots,\deg(x)\}$ encodes the $M$-match of $x$ otherwise, then the progress function $\RR P$, which verifies that $a(x)=1$, becomes 0-local. 
	One special example of an augmentation here is to replace $M\subseteq E$ by the symmetric difference
	$M\bigtriangleup P$, where $P\subseteq E$ is a (usual) augmenting path (that is, a path in $G$ whose endpoints are unmatched and whose edges alternate between $E\setminus M$ and $M$). Under either of the above encodings, the set $V(P)$ is augmenting here. 

	The following standard result states that, informally speaking, we can eliminate all $r$-augmen\-ta\-tions in a Borel way, additionally including a Borel ``certificate'' that  we changed at most $r$ assignment values per every vertex added to the progress set.
	Elek and Lippner~\cite{ElekLippner10} presented a version of it for matchings (when $\RT$ checks that the current assignment encodes a matching $M$ and augmenting sets are limited to augmenting paths). Their proof extends with obvious modifications to the general case and is presented here.

	\begin{theorem}
		\label{th:ElekLippner}
		Let $\RT$ and $\RR P$ be LCLs for labelled graphs. 
		Let $\C G$ be a locally finite Borel graph with a Borel labelling $\ell:V\to\omega$ and a Borel assignment $a_0:V\to\omega$ that solves~$\RT$ on~$(\C G,\ell)$. Then for every $r\ge 1$ there is a Borel assignment $a:V\to\omega$ such that $a$ solves~$\RT$  on~$(\C G,\ell)$ and admits no $r$-augmentation with respect to~$\RR P$. Additionally, there are $r$ Borel maps 
		$$f_j: \RR P(\C G,(\ell,a))\setminus \RR P(\C G,(\ell,a_0))\to V,\quad j\in r,
		$$ 
		such that $f_j\subseteq \C E_{\C G}$ for each $j\in r$ and every vertex $x\in V$ with $a(x)\not=a_0(x)$ is in the image of at least one~$f_j$.
	\end{theorem}

	\begin{proof} Let $t\in \omega$ be such that both $\RT$ and $\RR P$ can be computed by  a $t$-local rule.
		
		As in the proof of Lemma~\ref{lm:LocalRule}, we can assume that $\ell:V\to \omega$ is an $(r+2t)$-sparse colouring. Fix a sequence $(X_i)_{i\in\omega}$ where each $X_i$ is a non-empty subset of $\omega$ of size at most $r$ such that each such set appears as $X_i$ for infinitely many values of the index~$i$. Given $a_0$, we inductively define Borel assignments $a_1,a_2,\ldots:V\to\omega$, each solving the LCL~$\RT$. (Also, we will define some auxiliary functions $f_{i,j}$ that will be used to construct the final functions $f_j$, $j<r$.)
		
		Informally speaking, each new $a_{i+1}$ is obtained from $a_{i}$ by doing simultaneously all $r$-augmenta\-tions that can be supported on a connected set whose $\ell$-labels are exactly~$X_i$. Every two such sets are far away from each other by the sparseness of~$\ell$, so all these augmentations can be done in parallel without conflicting with each other.
		
		Suppose that $i\ge 0$ and we have already defined $a_{i}$, a Borel assignment $V\to\omega$ that solves~$\RT$. Let $\C X_i$ be the family of subsets $S\subseteq V$ such that $\induced{\C G}{S}$ is connected, $|S|\le r$, and $X_i=\ell(S)$, that is, the set of $\ell$-values seen on $S$ is precisely~$X_i$. Note that, since $\ell$ is $(r+2t)$-sparse, $\ell$ is injective on each $S\in\C X_i$ and the distance in $\C G$ between any two sets from $\C X_i$ is larger than $2t$; in particular, these sets are pairwise disjoint.
		Let $\C Y_i$ consist of those $S$ in $\C X_i$ which are augmenting for~$a_{i}$.
		Do the following for every $S\in\C Y_i$. First, take the lexicographically smallest function $b_S:S\to\omega$ such that the assignment 
		$
		\replace{a_{i}}{b_S}
		$ 
		(which is obtained from $a_i$ by letting the values of $b_S$ supersede it on $\dom(b_S)=S$) satisfies $\RT$ at every vertex 
		and has a strictly larger progress set than $a_{i}$ has, 
		that is,
		\beq{eq:Strict}
		\RT(\replace{a_i}{b_S})=V \mbox{ \ and \ } \RR P(\replace{a_{i}}{b_S}) \supsetneq \RR P(a_{i}).
		\eeq
		Let $P_S:=\RR P(\replace{a_{i}}{b_S})\setminus \RR P(a_{i})\not=\emptyset$. This set measures the progress made by the augmentation on~$S$. Note that $P_S$ is finite as a subset of $N^{\le t}(S)$. Let $(f_{S,0},\dots,f_{S,r-1})\in (S^{P_S})^r$ be the lexicographically smallest sequence with each $f_{S,i}$ being a function from $P_S$ to $S$ such that their combined images cover~$S$, that is, $\cup_{i\in r} \pi_1(f_{S,i})=S$. 
		(When defining the lexicographical order on $(S^{P_S})^r$, we can use, for definiteness, the total ordering of $S\cup P_S\subseteq N^{\le t}(S)$ given by the values of the $(r+2t)$-sparse colouring $\ell$ which is injective on this set.)
		Note that $r$ functions are enough as $|S|\le r$ while the ``worst'' case is when $P_S\not=\emptyset$ is a singleton.
		Having processed each $S\in\C Y_i$ as above, we define 
		$$
		a_{i+1}:=\replace{a_{i}}{\cup_{S\in\C Y_i}\, b_S}.
		$$ 
		In other words, $a_{i+1}$ is obtained from $a_i$ by replacing it by $b_S$ on each $S\in\C Y_i$. (Recall that these sets are pairwise disjoint.) Likewise, let
		\beq{eq:Fij}
		f_{i,j}:= \cup_{S\in \C Y_i} f_{S,j},\quad j\in r.
		\eeq

		Let us check, via induction on $i\in\omega$, the claimed properties of each constructed assignment~$a_{i+1}$, namely, that $a_{i+1}$ is Borel and solves~$\RT$. Note that $a_{i+1}$ is defined by a local rule on $(\C G,(\ell,a_{i}))$ of radius $r+2t$: by looking at $N^{\le r}(x)$ we can check if $x\in S$ for some $S\in\C X_i$ and, if such a set $S$ exists, then $N^{\le 2t}(S)\subseteq N^{\le r+2t}(x)$ determines whether $S\in\C Y_i$ and the value of $a_{i+1}$ on~$x$. (Note that we may need to look at distance as large as $2t$ from $S$ because the new values of $a_{i+1}$ on $S$ can affect the values of $\RT$ and $\RR P$ on $N^{\le t}(S)$ which in turn can depend on the values of $a_{i}$ on~$N^{\le 2t}(S)$.) Thus by induction and Lemma~\ref{lm:LocalRule}, the assignment $a_{i+1}:V\to\omega$ is Borel. Let us show that $a_{i+1}$ satisfies~$\RT$. Take any $x\in V$. If $\induced{a_{i+1}}{N^{\le t}(x)}=\induced{a_{i}}{N^{\le t}(x)}$, that is, $a_{i+1}$ and $a_{i}$ coincide  on every vertex at distance at most $t$ from $x$, then $\RT$ returns the same value on $x$ for $a_{i+1}$ as for $a_{i}$, which is 1 by induction. Otherwise there is $y\in N^{\le t}(x)$ which changes its value when we pass from $a_{i}$ to $a_{i+1}$. Let $S$ be the
		unique element of $\C Y_i$ that contains~$y$. 
		One of the requirements when we defined $b_S:S\to\omega$ was that $a':=\replace{a_{i}}{b_S}$ satisfies $\RT$. In particular,  $a'$ satisfies $\RT$ at the vertex~$x$. Now $a_{i+1}$ and $a'$ are the same at $N^{\le t}(x)$ because every element of $\C Y_i\setminus\{S\}$ is at distance more than $2t$ from $S$ and, by $S\cap N^{\le t}(x)\not=\emptyset$, at distance more than $t$ from~$x$. 
		Thus $a_{i+1}$ satisfies $\RT$ at $x$ since $a'$ does.
		As $x\in V$ was arbitrary, $a_{i+1}$ solves~$\RT$.
		
		Note that, since $t$ is an upper bound also on the radius of $\RR P$, the same argument as above when applied to every element of $N^{\le t}(S)$ shows by $\RR P(a_{i}\leadsto b_S)\supsetneq \RR P(a_{i})$ that 
		\beq{eq:Progress}
		\RR P(a_{i+1})\cap N^{\le t}(S)=\RR P(\replace{a_{i}}{b_S})\cap N^{\le t}(S)\supsetneq \RR P(a_{i})\cap N^{\le t}(S),\quad\mbox{for every $S\in\C Y_i$}.
		\eeq

		Let us show that, for every $x\in V$, the values $a_i(x)$, $i\in\omega$, stabilise eventually. Suppose that $a_{i+1}(x)\not= a_{i}(x)$ for some $i\in\omega$. Then $x\in S$ for some augmenting set $S\in \C Y_i$. When we change $a_{i}$ to $b_S$ on $S$, the progress set strictly increases, so
		take any $y\in \RR P(\replace{a_{i}}{b_S})\setminus \RR P(a_{i})$.  Of course, $y\in N^{\le t}(S)$.
		By~\eqref{eq:Progress}, we have $y\in \RR P(a_{i+1})\setminus \RR P(a_{i})$.
		This means that, every time the assignment at $x$ changes, some extra vertex from $N^{\le r+t}(x)$ is added to the progress set. This can happen only finitely many times by the local finiteness of $\C G$
		(and since no vertex is ever removed from the progress set by~\eqref{eq:Progress}). Thus the constructed assignments $a_{i}$ stabilise at every vertex, as claimed. 
		
		Define $a:V\to\omega$ by letting 
		$a(x)$ be the eventual value of $a_i(x)$ as $i\to\infty$.
		
		This assignment $a:V\to\omega$ is finitary and thus Borel by Corollary~\ref{cr:Finitary}.
		Also, it solves the LCL~$\RT$ (because, for every $x\in V$, all values of $a$ on $N^{\le t}(x)$ are the same as the values of the $\RT$-satisfying
		assignment $a_i$ for sufficiently large $i$).
		
		Let us show that no connected set $S\subseteq V$ of size at most $r$ can be augmenting for~$a$. This property depends on the values of $a$ on the finite set $N^{\le 2t}(S)$. Again, there is $i_0\in\omega$ such that $a$ and $a_i$ coincide on this set for all $i\ge i_0$. Since the set $\ell(S)$ appears in the sequence $(X_i)_{i\in\omega}$ infinitely often, there is $j\ge i_0$ with $X_j=\ell(S)$. Thus $S$ is in $\C X_j$ but not in $\C Y_j$ (otherwise at least one value on $S$ changes when we define~$a_{j+1}$). Thus $S$ is not augmenting for $a_j$ and, consequently, is not augmenting for~$a$. We conclude that the constructed Borel assignment $a$ admits no $R$-augmentation.

		Finally, let us define $f_{j}:=\cup_{i\in\omega} f_{i,j}$ for $j\in r$, where  each $f_{i,j}$ was defined in~\eqref{eq:Fij}.
		Note that when we define $f_{i,j}$ on some $P_S$ then, by~\eqref{eq:Progress}, every vertex of $P_S$ moves to the progress set at Stage~$i$ and stays there at all later stages. Thus each $f_j$, as a subset of $V\times V$, is a function.
		
		Note that each function $f_{i,j}$ moves any vertex in its domain by bounded distance in $\C G$ (namely, at most $r+t$). So it can be encoded by a vertex labelling $\ell_{i,j}$ on $V$ where for every $x\in V$ we specify if $f_{i,j}$ is defined on $x$ and, if yes, the sequence of the $\ell$-labels on the shortest (and then $\ell$-lexicographically smallest) path from $x$ to~$f_{i,j}(x)$. Each $\ell_{i,j}$ can clearly be computed by a local rule so it is Borel by Lemma~\ref{lm:LocalRule}. Lemma~\ref{lm:fS} implies that each function $f_{i,j}$ is Borel. Thus each function $f_j$ for $j\in r$ is Borel.
		Also, every vertex $x\in V$ where $a$ differs from $a_0$ belongs to some $S\in \C Y_i$ for some $i\in\omega$ and, by construction, $x$ is covered by the image of $f_{i,j}$ for some $j\in r$. Thus the images of $f_0,\dots, f_{r-1}$ cover all vertices where $a$ differs from~$a_{0}$, while the domain of each $f_j$ is $\RR P(a)\setminus \RR P(a_0)$ by construction.
		
		This finishes the proof of Theorem~\ref{th:ElekLippner}.\end{proof}
	

	\section{Negative Results via Borel Determinacy}
	\label{se:Marks16}
	
	The greedy bound $\chi(G)\le d+1$, where $d:=\Delta(G)$ is the maximal degree of $G$, is not in general best possible and can be improved for many finite graphs. For example, Brooks' theorem~\cite{Brooks41} states that, for a connected graph $G$, we have $\chi(G)\le d$ unless $G$ is a clique or an odd cycle. (See also Molloy and Reed~\cite{MolloyReed14} for a
	far-reaching generalisation of this result.)
	
	In contrast to these results, Marks~\cite{Marks16} showed rather surprisingly that, for every $d\ge 3$, the greedy upper bound $d+1$ on the Borel chromatic number is best possible, even if we consider acyclic graphs only. This was previously known for $d=2$:  the irrational rotation graph $\C R_\alpha$ of  
	Example~\ref{ex:Rotation}, is a $2$-regular acyclic Borel graph whose Borel chromatic number is $3$. 
	
	\hide{
		Before this result of Marks~\cite{Marks16}, Conley and Kechris~\cite[Theorem~0.7]{ConleyKechris13} showed that a specific $2n$-regular acyclic Borel graph, namely that the free part of the shift graph of $\Gamma\actson 2^\Gamma$ where $\Gamma$ is
		the free group with $n$ generators, has Borel chromatic number at least
		$\sqrt{2n-1}/(n+\sqrt{2n-1})$. As Lyons and Nazarov~\cite[Section~5]{LyonsNazarov11} observed (for more details see~\cite[Theorem~5.46]{KechrisMarks:survey}) that the known results on the independence ratio of random $d$-regular graphs allow to improve this bound to $d/(2\log d)$ when we take $[0,1]$ instead of $2=\{0,1\}$ in the above construction. (In fact, both of these papers have much stronger results than stated, namely that any measurable independent set has small measure with respect to the uniform probability measure on $2^\Gamma$ or $[0,1]^\Gamma$.)
	}
	
	We present a slightly stronger version which follows directly from Marks' proof.
	
	\begin{theorem}[Marks~\cite{Marks16}] 
		\label{th:Marks16}
		For every $d\ge 3$ there is a Borel acyclic $d$-regular graph $\C G=(V,E,\C B)$ with a Borel proper edge colouring $\ell:E\to d$ such that for every Borel
		map $c:V\to d$ there is an edge $(x,y)\in E$ with $c(x)=c(y)=\ell(x,y)$. (In particular,
		$\chi_{\C B}(\C G)\ge d+1$.)
	\end{theorem}
	
	\begin{proof} We follow the presentation from Marks~\cite{Marks13u}, generally adding more details. While the proof can be concisely written (the whole note~\cite{Marks13u} is only 1-page long), it is quite intricate.

		Let 
		$$
		\Gamma:=\langle \gamma_0,\dots,\gamma_{d-1}\,|\,\gamma_0^2=\dots=\gamma_{d-1}^2=e\rangle
		$$ 
		be the group freely generated by $d$ involutions $\gamma_0,\dots,\gamma_{d-1}$, that is, $\Gamma=\I Z_2*\dots*\I Z_2$ is the free product of $d$ copies of $\I Z_2$, the cyclic group of order~$2$. 
		
		Let $(G,\lambda)$
		be the \emph{(right) edge coloured Cayley graph} of $(\Gamma;\gamma_0,\dots,\gamma_{d-1})$
		whose vertex set is $\Gamma$ and whose edges are given by right multiplication by the involutions $\gamma_0,\dots,\gamma_{d-1}$, that is, for each $\beta\in\Gamma$ and $i\in d$ we connect $\beta$ and $\beta\gamma_i$ by an edge, which gets colour~$i$ under~$\lambda$. (Note that 
		the colour of the edge $\{\beta,\beta\gamma_i\}$ does not depend on the choice of an endpoint since $\gamma_i^2=e$.)  The graph $(G,\lambda)$ is isomorphic to the infinite edge $d$-coloured  $d$-regular tree.
		
		The group $\Ga$ naturally acts on itself. We consider the left action $a:\Gamma\actson\Gamma$ where the action $a(\gamma,\cdot)$ of $\gamma\in \Ga$ is just the left multiplication by~$\gamma$ which maps $\beta\in\Gamma$ to $\gamma\beta\in\Ga$. Note that we take the left multiplication for the action but the right multiplication when defining the Cayley graph~$G$. This ensures that the automorphisms of the graph $G$ that preserve the edge colouring $\lambda$ are precisely the left multiplications by the elements of $\Ga$:
		\beq{eq:Aut}
		\mathrm{Aut}(G,\lambda)
		=\{a(\gamma,\cdot)\mid\gamma\in\Ga\}.
		\eeq

		We view the elements of $\omega^\Gamma$ as functions $\Gamma\to\omega$ and  call them \emph{labellings}. (The proof, as it is written, also works if we replace $\omega^\Gamma$ by $A^\Gamma$ for some finite set $A$ of size $(d-1)^2+1$.) The standard Borel structure on $\omega^\Gamma$ comes from the product topology where we view $\omega^\Gamma$ as the product of countably many copies of the discrete space~$\omega$. 
		
		The group $\Ga$ naturally 
		acts on this space via the \emph{(left) shift action} $s:\Ga\actson \omega^\Gamma$ defined as follows. For $\gamma\in\Gamma$ and $x\in \omega^\Gamma$, the \emph{(left) shift} $\gamma.x\in \omega^\Gamma$ of $x$ is defined by 
		$$
		(\gamma.x)(\beta):=x(\gamma^{-1}\beta),\quad \beta\in \Gamma.
		$$
		In other words, we just pre-compose the labelling $x:\Gamma\to\omega$ with the map $a(\gamma^{-1},\cdot)$, the left multiplication by~$\gamma^{-1}$. (We take the inverse of $\gamma$ to get a left action, namely so that
		the identity $(\gamma\beta).x=\gamma.(\beta.x)$ always holds.%
		)
		For each $\gamma\in \Ga$, the $\gamma$-shift map $s(\gamma,\cdot):\omega^\Gamma\to \omega^\Gamma$ is Borel; in fact, it is a homeomorphism of the product space $\omega^\Gamma$ as it just permutes the factors. Thus the action is Borel.
		
		Let $\C S$ be the \emph{shift graph} on $\omega^\Gamma$ where a vertex $x\in\omega^\Ga$ is adjacent to every element in $\{\gamma_0.x,\dots,\gamma_{d-1}.x\}\setminus\{x\}$. In other words, $\C S$ is the Schreier graph $\C S(s;\{\gamma_0,\dots,\gamma_{d-1}\})$ as defined in Example~\ref{ex:BAction}.
		The Borel graph $\C S$ comes with the Borel edge colouring $E(\C S)\to 2^d$ where the colour of an edge $(x,y)\in E(\C S)$ is the (non-empty) set of  $i\in d$ such that $\gamma_i.x=y$.
		
		The graph $\C S$ has cycles and is not $d$-regular. (For example, the constant-0 labelling of $\Ga$ is an isolated vertex of $\C S$.)
		Let $X$ consist of those labellings $x\in \omega^\Gamma$ that give a proper vertex colouring of the Cayley graph~$G$, that is,
		\begin{eqnarray*}
			X&:=&\{x\in \omega^\Gamma\mid \forall \beta\in\Gamma\ \forall i\in d\ \ x(\beta)\not=x(\beta\gamma_i)\}\\
			&=&\bigcap_{\beta\in\Gamma} \bigcap_{i\in d} \bigcup_{k,m\in \omega\atop k\not=m}
			\{x\in \omega^\Gamma\mid x(\beta)=k\ \wedge \ x(\beta\gamma_i)=m\}.
		\end{eqnarray*}
		The second formula for $X$ makes it clear that this set is a Borel subset of~$\omega^\Gamma$.  Also, $X$ is an invariant set under the shift action $s$;  this follows from~\eqref{eq:Aut} since a proper colouring remains proper when pre-composed with an automorphism of the graph.
		
		The graph $\induced{\C S}{X}$ is neither acyclic nor $d$-regular. For example, there are exactly two proper $2$-colourings $\Ga\to\{0,1\}$ of the bipartite graph $G$ and they form an isolated edge in~$\induced{\C S}{X}$. So we need to do another (final) trimming.

		Let $Y:=X\cap \mathrm{Free}(s)$ be the intersection of $X$ with the free part of the shift action~$s$. 
		In other words, $Y$ consists of those proper vertex colourings of $G$
		that have no symmetries under the automorphisms of the edge-coloured graph~$(G,\lambda)$. 
		The set $Y$ is invariant under the action $s$, since $X$ and the free part $\mathrm{Free}(s)$ are. In particular, there are no edges connecting $X\setminus Y$ to $Y$ and the induced subgraph $\C G:=\induced{\C S}{Y}$ is $d$-regular and acyclic. The graph $\C G$ comes with the Borel edge $d$-colouring $\ell$, where, for each $x\in Y$ and $i\in d$, we colour the edge $\{x,\gamma_i.x\}$ by~$i$. 
		As we argued before, the set $X$ is Borel. Also, the free part of the Borel action $s$ is Borel by Lemma~\ref{lm:FreeB}. Thus $Y$ and the graph $\C G$ are Borel. The following two claims clearly imply that the edge coloured graph $(\C G,\ell)$ satisfies the theorem. (Note that,  for all $x\in X$ and $i\in d$, we have $\gamma_i.x\not=x$ because $(\gamma_i.x)(e)=x(\gamma_i^{-1})$ is different from $x(e)$ as $x\in X$ is a proper colouring of~$G$ and assigns distinct colours to the adjacent vertices $e$ and~$\gamma_i^{-1}=\gamma_i$.)

		\begin{claim}\label{cl:X} If $c:X\to d$  is a Borel map then  there is $x\in X$ and $i\in d$ with $c(x)=c(\gamma_i.x)=i$.
		\end{claim}

		\begin{claim}\label{cl:X-Y} There is a proper Borel $d$-colouring $c$ of the graph $\induced{\C S}{(X\setminus Y)}$.
		\end{claim}

		\bcpf[Proof of Claim~\ref{cl:X}.] 
		Given $c:X\to d$, we define for every $i\in d$ and $j\in\omega$  the game $G_{i,j}$ where two players, I and II,  take turns to construct $x:\Ga\to\omega$ which is a proper vertex colouring of~$G$. Initially, we start with the partial colouring $x$ which assigns value $j$ to the identity (that is $x(e)=j$) and is undefined on~$\Gamma\setminus\{e\}$. There are countably many rounds as follows.  For convenience, let us identify each element $\beta\in\Ga$ with the
		unique reduced word in $\gamma_0,\dots,\gamma_{d-1}$ representing~$\beta$. 
		In \emph{Round $r$} for $r=1,2,\dots\,$, first Player I chooses the values of $x$ at all reduced words of length~$r$ that begin with $\gamma_i$, and then Player II chooses the values of $x$ at all other reduced words of length~$r$ (that is, those that begin with $\gamma_m$ for some $m\in d\setminus\{i\}$). The restriction that applies to both players is that the current partial vertex colouring $x$ of $G$ is proper  at every stage. Since there are infinitely many available colours, there is always a non-empty set of responses for a player. So the game continues for $\omega$ rounds. Since a value of $x$, once assigned, is never changed later, a run of the game gives a fully defined map $x:\Ga\to\omega$ which, by the imposed restriction, is in fact an element of~$X$. Player I wins the game if and only if $c(x)\not=i$,  where $c:X\to d$ is the given Borel map.
		
		In general, by using the Axiom of Choice one can design games of the above type (when two players construct an infinite sequence in countably many rounds) where none of the players has a winning {strategy},
		that is, for every strategy of one player, there is a strategy of the other player that beats it. The groundbreaking result of Martin~\cite{Martin75} (with a simplified proof presented in~\cite{Martin85}) states that, for games with countably many choices in each round,  if the set of winning sequences is Borel then the game is \emph{determined}, i.e.\ one of the players has a winning strategy. Here, for the game $G_{i,j}$, the set of winning labellings for Player I is exactly $c^{-1}(i)$. Thus each game $G_{i,j}$ is determined.
		
		Let us show that for every $j\in\omega$ there is $i\in d$ such that Player II has winning strategy in~$G_{i,j}$. Suppose on the contrary that this is false. This implies by the Borel determinacy that, for every $i\in d$, Player I has a winning strategy for $G_{i,j}$; let us call this strategy~$S_i$. Now, we let these $d$ strategies play each against the others
		as follows. We start with $x(e):=j$. In \emph{Stage~$r$}, for $r=1,2,\dots$, we use the round-$r$ responses of the strategies $S_0,\dots,S_{d-1}$ in parallel to define $x$ on all reduced words of length~$r$. Induction on $r$ shows that, before this stage, we have a proper partial vertex colouring $x$ defined on $N^{\le r-1}_G(e)$ with $x(e)=j$. For every $i\in d$, this is a legal position of $G_{i,j}$ when  Player I is about to define $x$ on length-$r$ reduced words beginning with~$\gamma_i$. So we use the values specified by~$S_i$ on these words. Also, the new values assigned by two different strategies are never adjacent; in fact, they are $2r$ apart as the unique shortest path between them in the tree $G$ has to go via the identity~$e$. Thus, after Stage~$r$, the new partial labelling $x$ is a proper vertex colouring of $N^{\le r}_G(e)$ and we can proceed to Stage~$r+1$. The final labelling $x:\Gamma\to\omega$ is clearly a proper colouring of~$G$. Since each $S_i$ is a winning strategy, we have that $c(x)\not=i$. But this is impossible as $c$ has to assign some colour to~$x\in X$.

		By the previous paragraph and the Pigeonhole Principle, there are $i\in d$ and distinct $j_0,j_1\in\omega$ such that Player II has a winning strategy in $G_{i,j_0}$ and $G_{i,j_1}$.  Let these strategies be $T_0$ and $T_1$ respectively. 
		Now, we let $T_0$ play against $\gamma_i.T_1$, the ``$\gamma_i$-shifted'' version of~$T_1$, to construct a labelling~$x:\Ga\to\omega$ as follows. Initially, we let $x(e):=j_0$ and $x(\gamma_i):=j_1$, viewing it as \emph{Stage~$0$}. Iteratively for each $r=1,2,\dots$, \emph{Stage~$r$} uses the round-$r$ responses of $T_0$ and $\gamma_i.T$ in parallel
		to colour  all reduced words of length $r$ that  do not start with $\gamma_i$ and, respectively, all $\beta\in\Ga$ such that the reduced word of $\gamma_i\beta$ has length $r$ and does not start with~$\gamma_i$. The last set is represented precisely by reduced words of length $r+1$ that start with~$\gamma_i$. Thus, an induction on $r\ge1$ shows that the set $D_r$ of vertices on which the partial colouring $x$ is defined just before Stage~$r$ is represented by all reduced words of length $r-1$ and those of length $r$ that being with~$\gamma_i$. (This is true in the base case $r=1$ since the initial colouring is defined on~$D_1=\{e,\gamma_i\}$.) This is exactly the information that the strategy $T_0$ needs to know in Round~$r$. Note that $\gamma_i.D_r:=\{\gamma_i.\beta\mid \beta\in D_r\}$ is the same set $D_r$ so the round-$r$ response of $\gamma_i.T_1$ can also be computed. Thus we can play $T_0$ and $\gamma_i.T_1$ in every stage, without any conflicts between the assigned values. After $\omega$ stages,
		$\gamma_i.T_1$ colours all reduced words beginning with $\gamma_i$ and $T_0$ colours the rest of $\Gamma$ and
		we 
		get an everywhere defined function $x:\Ga\to\omega$ which is also a proper colouring of the Cayley graph~$G$. (Note that the only possibly conflicting edge $\{e,\gamma_i\}$ of $G$ gets distinct colours $j_0$ and $j_1$ before Stage~1.)

		Since $x$ can be represented as a run of the game $G_{i,j_0}$ where Player II applies the winning strategy $T_0$, it holds that $c(x)=i$. Also, $\gamma_i.x$ is a run of the game $G_{i,j_1}$, where the winning strategy~$T_1$ is applied. Thus $c(\gamma_i.x)=i$.  We see that this labelling $x$ satisfies Claim~\ref{cl:X}.\ecpf
		
		Next, we prove the remaining Claim~\ref{cl:X-Y} with its proof being fairly routine to experts.
		
		\bcpf[Proof of Claim~\ref{cl:X-Y}.] 
		Let a \emph{generalised cycle} be a finite sequence 
		$$
		(x_0,\dots,x_{m-1};\gamma_{i_0},\dots,\gamma_{i_{m-1}})\in (\omega^\Ga)^{m}\times \Ga^{m}
		$$ such that $m\ge 1$, $x_0,\dots,x_{m-1}$ are pairwise distinct, 
		$x_{j+1}=\gamma_{i_j}.x_j$ for every $j\in m$ where we denote $x_m:=x_0$, and  if $m=2$ then $i_0\not=i_1$. If $m\ge 3$ then this gives a usual cycle of length $m$ in the graph~$\C S$ (with a direction and a starting vertex specified). The cases $m=1$ and $m=2$ correspond to ``imaginary cycles": if we re-define the Schreier graph $\C S$ as the natural $d$-regular multigraph with loops, then these would correspond to loops and pairs of multiple edges respectively.
		
		Note that generalised cycles are minimal witnesses to non-freeness of the shift action $s:\Gamma\actson \omega^\Ga$ in the following sense.
		Suppose that $y\in \omega^\Gamma$ is not in the free part. Then there are $x_0\in [y]\OptionalG $ and non-identity $\gamma\in\Ga$ with $\gamma.x_0=x_0$. Writing $\gamma=\gamma_{i_{m-1}}\dots\gamma_{i_0}$ as the reduced word in $\{\gamma_0,\dots,\gamma_{d-1}\}$ and inductively on $j\in m-1$, letting $x_{j+1}:=\gamma_{i_j}.x_{j}$, we get all properties of a generalised cycle except vertices can repeat here. Now, if there are repetitions among $x_0,\dots,x_m$ then take two repeating vertices whose indices are closest and restrict to the
		subsequence between them.

		Let us briefly argue that we can choose a Borel set $\C C$ of vertex-disjoint generalised cycles in 
		$$
		\C S':=\induced{\C S}{(X\setminus Y)}
		$$ such that every component of $\C S'$ contains at least one. (This claim also directly follows from  \cite[Lemma 7.3]{KechrisMiller:toe}.) First, notice that the function $S$ which corresponds to each $x\in X\setminus Y$ the shortest length of a generalised cycle in $[x]_{\C S'}=[x]_{\C S}$ is Borel by Corollary~\ref{cr:Finitary} as a finitary function. Namely, it is the pointwise limit of $r$-local rules $\RR S_r:X\setminus Y\to \omega\cup\{\omega\}$, where for $x\in \omega^\Ga$ we define $\RR S_r(x)$ to be the minimum length of a generalised cycle inside $N^{\le r}(x)$ (which is not required to pass through $x$) and to be $\omega$ if none exists. Note that the function $S$ is invariant (that is, assumes the same value for all vertices in a graph component) even though $\RR S_r$ need not be. Now, for each $m\ge 1$, define the graph $\C G_m$ whose vertices are all generalised cycles of length $m$ in the components where the function $S$ assumes value $m$,
		where two vertices of $\C G_m$ are adjacent if the corresponding cycles have at least one common vertex. This graph,
		whose vertex set is a subset of the Polish space $(\omega^\Ga)^{m}\times \Ga^{m}$,
		can be routinely shown to be Borel with the tools that we have already presented. Also, the maximum degree of $\C G_m$ can be bounded by $m\, d^m$. Thus, by 
		Theorem~\ref{th:MaxInd}, we can choose a Borel maximal independent set $I_m$ in $\C G_m$. The union $\C C:=\cup_{m\ge 1} I_m$ satisfies the claim by the maximality of~$I_m$. (In fact, more strongly, we picked a maximal subset of vertex-disjoint shortest generalised cycles inside each component of~$\C S'$.)

		For every chosen generalised cycle $(x_0,\dots,x_{m-1};\gamma_{i_0},\dots,\gamma_{i_{m-1}})\in\C C$ and for each $j\in m$, define $c(x_j):=i_j$. This partially defined vertex colouring of $\C S'$ has the property that, for every edge $\{x,\gamma_i.x\}$ of $\C S'$, if $c(x)=i$ then $c(\gamma_i.x)\not=i$. (Note that $\C S'$ does not have any loops by definition, although we used ``imaginary loops''  when defining generalised cycles of length~$1$.)
		
		Now, for every uncoloured vertex $x$ of $\C S'$ take a shortest path $P$ from $x$ to a generalised cycle in $\C C$ and, if there is more than one choice of $P$ then choose one where the sequence of edge colours on $P$ is lexicographically smallest. Define $c(x):=i$, where $i$ is the smallest element of $\{0,\dots,d-1\}$ such that $(x,\gamma_i.x)$ is the first edge of~$P$. Note that $c(\gamma_i.x)$ is the colour of either  the second edge on $P$ (if $P$ has at least two edges) or the edge coming out of $\gamma_i.x$ in the (unique) generalised cycle in $\C C$ containing $\gamma_i.x$. In either case, $c(\gamma_i.x)$ cannot be $i$ as otherwise $P$ is not a shortest path. 
		
		This colouring $c$ is finitary on the labelled graph $(\C S',\ell')$, where $\ell'$ is a (Borel) vertex labelling encoding both the edge labelling of $\C S'$ as well as the initial partial colouring of all vertices on the generalised cycles from~$\C C$. Indeed, $c$ can be built as the nested union of partial local colourings $\RR C_r$, $r\in\omega$, where each vertex $x$ computes its (final) colour if $N^{\le r}(x)$ contains at least one vertex covered by $\C C$ and declares $\RR C_r(x)$ undefined otherwise. 
		By Corollary~\ref{cr:Finitary}, $c$ is Borel.
		Thus $c$ is the required
		colouring of~$\C S'$. Claim~\ref{cl:X-Y} is proved.
		\ecpf
		
		This finishes the proof of Theorem~\ref{th:Marks16}.\end{proof}
	
	Among many further results, Marks~\cite[Theorem~1.4]{Marks16} proved that the greedy upper bound for edge colouring of Theorem~\ref{th:ChiB'} is best possible even for acyclic Borel graphs that additionally admit a Borel bipartition.
	
	\begin{theorem}[Marks~\cite{Marks16}]
		\label{th:Marks16Edge} For every $d\ge 3$, there is a Borel acyclic $d$-regular graph $\C G$ such that $\chi_{\C B}(\C G)=2$ and $\chi_{\C B}'(\C G)=2d-1$.\qed\end{theorem} 
	
	\begin{remark} In fact, the bipartite graph in Theorem~\ref{th:Marks16Edge} constructed by Marks also does not admit a Borel perfect matching. Previously, such a graph for $d=2$ was constructed by Laczkovich~\cite{Laczkovich88}. Of course, it does not admit a Borel edge 2-colouring and satisfies Theorem~\ref{th:Marks16Edge} for $d=2$.
		\ecpf\end{remark}
	
	Thus the Borel chromatic number of a Borel graph cannot be  bounded by some function of its chromatic number. 
	However, some bounds can be shown under certain additional assumptions.  For example,  Weilacher~\cite{Weilacher20arxiv} showed by building upon some earlier results of Miller~\cite{Miller09} that if each component of a Borel graph $\C G$ is 2-ended then $\chi_{\C B}(\C G)\le 2\chi(\C G)-1$ (and that, under these assumptions, this bound is in fact best possible). The same bound $\chi_{\C B}(\C G)\le 2\chi(\C G)-1$  was established by Conley,
	Jackson, Marks, Seward and Tucker-Drob~\cite{ConleyJacksonMarksSewardTuckerdrob20arxiv} under the assumption that the Borel asymptotic dimension of $\C G$ is finite. The proofs from both papers can be presented so that we first find a Borel partition $V=A\cup B$ 
	with the induced graphs $\induced{\C G}{N^{\le 1}(A)}$ and $\induced{\C G}{B}$ having finite components only. Then, by Lemma~\ref{lm:FinCompSmooth} and Theorem~\ref{th:Smooth}, there are proper Borel colourings $a:N^{\le 1}(A)\to k$ and $b:B\to \{k,\dots,2k-1\}$ of these graphs, where $k:=\chi(\C G)$. We use $a$ on $A$ and $b$ on $B$, except we can save one colour by recolouring the independent set $b^{-1}(2k-1)$ where we use the colouring $a$ on its vertices with at least one neighbour in $A$ and assign colour 0 to the rest.
	
	\hide{Conley--Kechris~\cite[Theorems 5.1 and 5.11]{ConleyKechris13} showed that $\chi_{\C B}(G)\le \Delta(G)$ for every vertex-transitive graph with weakly 3-connected components or with one-ended components, etc.}

	\section{Borel Equivalence Relations}
	\label{se:CBER}

	An equivalence relation $\C E$ on a standard Borel space $X$ is called \emph{Borel} if it is Borel as a subset of~$X^2$. When we have some notion of isomorphism on a set of structures, it often leads to a Borel equivalence relation. For example, if  $X= R^{n\times n}$ encodes $n\times n$ matrices then the similarity relation can be shown to be Borel (by combining Proposition~20.3 and Example 20.6.(b) from~\cite{Tserunyan19idst}). 
	\hide{Or, when
		studying graphs with at most countably many vertices one can assume that their vertex set is an  initial segment of $\omega$ of the appropriate cardinality, so we can take for $X$ the set of pairs $(n,E)\in (\omega\cup\{\omega\})\times 2^{\omega\times\omega}$ with $E$ forming a graph on~$n$.}
	
	If we view Borel maps as ``computable'' then many ``computational'' questions translate to descriptive set theory problems. 
	For example, the existence of a Borel \emph{selector} for $\C E$ (a map $s:X\to X$ such that $\graph{s}\subseteq \C E $ and $s(X)$ is a transversal of~$\C E$) can be interpreted as being able to ``compute'' one canonical representative from each equivalence class. (To connect this to Section~\ref{se:Smooth}, observe that,  by e.g.\ \Ts{Proposition~20.3}, a Borel equivalence relation admits a Borel selector  if and only if it admits
	a Borel transversal.) The Jordan canonical form is an example of a Borel selector for the above matrix similarity relation (see \Ts{Example 20.6.(b)}). Also, if we have a Borel \emph{reduction} from $(X,\C E)$ to another Borel equivalence relation $(X',\C E')$, that is, a Borel map $r:X\to X'$ such that, for $x,y\in X$, we have $(x,y)\in\C E$ if and only if $(r(x),r(y))\in\C E'$, then we could say that the relation $\C E$ as not ``harder to compute modulo $r$'' than~$\C E'$. A lot of effort in this area went into understanding the hierarchy of possible Borel equivalence relations under the Borel (and some other kinds of) reducibility, see e.g. the survey by Hjorth~\cite{Hjorth10} that concentrates on this aspect.

	A promising field for applying combinatorial methods is the theory of \emph{countable} Borel equivalence relations (\emph{CBERs} for short), where each equivalence class is countable. 
	A more general result of Miller~\cite[Theorem C]{Miller09} implies that every CBER is the connectivity relation of some locally finite Borel graph. 
	So, various questions of descriptive set theory can be approached from the graph theory point of view.
	
	One example of an important and actively studied property is as follows. A CBER $(X,\C E)$  is called \emph{hyperfinite} if there are Borel equivalence relations $F_m\subseteq X\times X$, $m\in\omega$, such that $F_0\subseteq F_1\subseteq F_2\subseteq \ldots$\,, $\cup_{m\in\omega} F_m=\C E$
	and each $F_m$ has finite equivalence classes (or, equivalently, all of size at most $m$, see~\cite[Remark 6.10]{KechrisMiller:toe}). Slaman and Steel~\cite{SlamanSteel88} and Weiss~\cite{Weiss84} showed that the hyperfiniteness of $\C E$ is equivalent to being generated by some Borel action of~$\I Z$. The latter means that there is a Borel bijection $\phi:X\to X$ such that for every $x\in X$ the equivalence class $[x]_{\C E}$ of $x$ is exactly $\{\phi^{n}(x)\mid n\in\I Z\}$. Such a function $\phi$ is easy to find for finite equivalence classes (which generate a smooth equivalence relation by Lemma~\ref{lm:FinCompSmooth}). Thus the main point here is that every infinite class can be bijectively exhausted from any of its elements by applying Borel functions ``next" (namely, $\phi$) and ``previous'' (namely, the inverse $\phi^{-1}$). Although this reformulation of hyperfiniteness looks somewhat similar to smoothness, these two properties behave quite differently.
	For example, Conley, Jackson, Marks, Seward and Tucker-Drob~\cite{ConleyJacksonMarksSewardTuckerdrob20} showed that there is Borel graph $\C G$ satisfying Theorem~\ref{th:Marks16} (resp.\ Theorem~\ref{th:Marks16Edge}) such that its connectivity relation $\C E_{\C G}$ is hyperfinite. See e.g.~\cite[Section~6]{KechrisMiller:toe} for a detailed discussion of hyperfiniteness including the proof of the Slaman--Steel--Weiss Theorem.

	Although each CBER as a graph is just a union of countable cliques, the following theorem of Feldman and Moore~\cite{FeldmanMoore77} (see e.g.\ \cite[Theorem 1.3]{KechrisMiller:toe})
	gives a very useful symmetry breaking tool (in particularly, allowing us to identify vertices from the local point of view of any vertex $x$ by the edge colourings of shortest paths from $x$ and then apply results like an edge coloured version of Lemma~\ref{lm:fS}).

	\begin{theorem}[Feldman and Moore~\cite{FeldmanMoore77}]
		\label{th:FeldmanMoore} For every countable Borel equivalence relation $\C E\subseteq X^2$ 
		on a standard Borel space $X$  
		there is a Borel map $c:\C E\to\omega$ such that every colour class is a matching.\qed
	\end{theorem}
	
	Given the edge colouring $c:\C E\to\omega$ returned by the Feldman--Moore Theorem, we can encode each matching $c^{-1}(i)$ by an involution $\phi_i:X\to X$ that swaps every pair $x,y\in X$ with $c(x,y)=i$ and fixes every remaining element of~$X$. Thus if $\Ga$ is the group generated by the bijections $\phi_i:X\to X$ for $i\in\omega$, then the equivalence classes of $\C E$ are exactly the orbits of the action of $\Gamma$ on~$X$. Thus every CBER comes from a Borel action of a countable group, giving another very fruitful connection.

	The books by Gao~\cite{Gao09idst} and Kechris and Miller~\cite{KechrisMiller:toe} provide an introduction to Borel equivalence relations (from the point of view of group actions). See also the recent survey of results on CBERs by Kechris~\cite{Kechris19cber}.

	\section{Baire Measurable Combinatorics}
	\label{se:Baire}
	
	Recall that a subset $A$ of a Polish space $X$ has \emph{the property of Baire} if it is the symmetric difference of a Borel set and a set which is \emph{meager} (that is, a countable union of nowhere dense sets). Such a set $A$ is also often called \emph{Baire measurable}. An equivalent (and, from some points of view, more natural) definition  is that $A$ is the symmetric difference of an open set and a meager set.  Note that this property is not determined by the Borel $\sigma$-algebra $\C B(X)$ alone (that is, it depends in general on the topology on~$X$). For an introduction to these concepts from the descriptive set theory point of view, we refer to~\cite{Kechris:cdst} or~\cite{Tserunyan19idst}.
	\hide{By restricting ourselves to Polish spaces, we gain a number of powerful tools such as, for example, the Baire Category Theorem (see e.g.\ \Ke{Theorem~8.4} or \Ts{Theorem~6.12}) which states that a non-empty open subset of a Polish space cannot be meager.}

	Meager (resp.\ Baire measurable) sets can be considered as topologically ``negligible'' (resp.\ ``nice''). 
	One can show by using the Axiom of Choice that there are subsets $\I R$ without the property of Baire,
	see e.g.~\Ke{Example 8.24} or~\cite[Chapter~5]{Oxtoby:mc}.
	So various questions to find satisfying assignments $a:V\to\omega$ in a Borel graph $\C G$ with each preimage in $\C T$ are meaningful and interesting.
	A typical strategy is to construct a Borel assignment apart of a meager set $Y$ of vertices. Ideally, the remaining set $Y$ is \emph{invariant} (i.e.\ $Y=[Y]$); then $\induced{a}{Y}$ can be defined independently of the rest of the vertex set, e.g.\ by using the Axiom of Choice. The reader should be aware that, in general locally finite Borel graphs, the neighbourhood of a meager set (resp.\ a set with the property of Baire) need not be meager (resp.\ have the property of Baire); however, these properties do hold in many natural situations (e.g.\ when $\C G$ is the Schreier graph of a marked group acting by homeomorphisms). So, usually, one works with (partially defined) Borel assignments.
	
	One method of how to deal with this technical issue is the following useful lemma of Marks and Unger~\cite[Lemma~3.1]{MarksUnger16}: for every locally finite Borel graph $\C G$ and any function $f:\omega\to \omega$ there are Borel sets $A_n$, $n\in \omega$, such that each $A_n$ is $f(n)$-sparse and the complement of their union, $V\setminus \cup_{n\in\omega} A_n$, is a meager and invariant set. It is used by Marks and Unger~\cite[Theorem 1.3]{MarksUnger16} to prove the following ``topological'' version of Hall's marriage theorem. Let us say that a bipartite graph $G$ with a bipartition $V=B_0\cup B_1$ satisfies $\mathrm{Hall}_{\e,n}$ if for every finite set $X$ in a one part we have $|N(X)|\ge |X|$ (which is the usual Hall's marriage condition) and, additionally, if $|X|\ge n$ and $X$ is connected in $\C G^2$ then $|N(X)|\ge(1+\e)|X|$.
	
	\begin{theorem}[Marks and Unger~\cite{MarksUnger16}]
		\label{th:BaireM} If $\e>0$, $n\in\omega$, $V$ is a Polish space, and $\C G=(V,E,\C B)$ is a locally finite Borel graph with a Borel bipartition $V=B_0\cup B_1$ satisfying $\mathrm{Hall}_{\e,n}$
		then $\C G$ has a Borel matching such that the set of unmatched vertices is meager and invariant.\end{theorem}
	
	In brief, the proof of Theorem~\ref{th:BaireM} proceeds as follows. Given $\e$ and $n$, choose a fast growing sequence $f(0)\ll f(1)\ll f(2)\ll \dots$ and let $A_i$, $i\in\omega$, be the sets returned by~\cite[Lemma~3.1]{MarksUnger16} for this function~$f$. Starting with the empty matching $M_0:=\emptyset$ and $\C G_0:=\C G$,
	we have countably many stages indexed by $i\in\omega$. At Stage~$i$, every vertex $x$ of $\C G_i$ in the $f(i)$-sparse set $A_i$ picks a neighbour $y_x\in N(x)$ such that $\C G_i$ has a perfect matching containing the edge $\{x,y_x\}$, say we take the largest such $y_x$ with
	respect to some fixed Borel total order on~$V$.
	Define $M_{i+1}:=\{\,\{x,y_x\}\mid x\in A_i\cap V(\C G_i)\}$ and let $\C G_{i+1}$ be obtained from $\C G_{i}$ by removing all vertices matched by $M_{i+1}$. A combinatorial argument
	shows by induction on $i\in\omega$ that $\C G_i$ satisfies $\mathrm{Hall}_{\e_i,f(i)}$, where $\e_i:=\e-\sum_{j\in i} 8/f(j)>0$. In particular, each graph $\C G_i$ satisfies the usual Hall's marriage condition. Thus, by Rado's theorem (Theorem~\ref{th:Rado42}), $y_x$ exists for every~$x$ (and, by induction, we can carry out each stage). Moreover, the function $x\mapsto y_x$ of Stage~$i$ is Borel by induction on $i$ as it can be computed by a finitary rule on $(\C G_i,M_i)$, as it was demonstrated in Example~\ref{ex:PerfectM}. (In fact, we do not need to refer to Rado's theorem at all: when defining $y_x$, we can instead require that the graph obtained from $\C G_i$ by removing the adjacent vertices $x$ and $y_x$ satisfies 
	Hall's marriage condition.)
	Finally, $M:=\cup_{i\in\omega} M_i$ has all the required properties.

	Note that the matching $M$ returned by Theorem~\ref{th:BaireM} can be extended to cover all vertices, using Rado's theorem (Theorem~\ref{th:Rado42}); this application of the Axiom of Choice is restricted to a meager set. Thus if we encode the final perfect matching via a vertex labelling $\ell:V\to\omega$ as in Remark~\ref{rm:Reduction} then each preimage of $\ell$ has the property of Baire.

	A notable general result that is very useful in Baire measurable combinatorics is that every 
	CBER $\C E$ on a Polish space $X$ can be made hyperfinite by removing an $\C E$-invariant meager Borel subset of $X$; for more details see e.g.\ \cite[Theorem~12.1]{KechrisMiller:toe} 

	\section{$\mu$-Measurable Combinatorics}
	\label{se:Measurable}
	
	Suppose that we have a Borel locally finite graph $\C G=(V,E,\C B)$ and a measure $\mu$ on~$(V,\C B)$.
	Let $\C B_\mu$ be the \emph{$\mu$-completion of $\C B$} which is the smallest $\sigma$-algebra on $V$ containing  Borel sets and all \emph{$\mu$-null} sets (i.e.\ arbitrary subsets of Borel sets of $\mu$-measure 0). Recall from the Introduction that the sets in $\C B_\mu$ are called \emph{$\mu$-measurable} or just \emph{measurable}. 
	
	Of course, the presence of a measure $\mu$
	makes the set of questions that can be asked and the tools that can be applied much richer. 
	For example, the problems that we considered in this paper also make sense in the measurable setting, where we look for assignments $a$ that are \emph{measurable} as functions from $(V,\C B_\mu)$ to $\omega$, meaning here that $a^{-1}(i)\in\C B_\mu$ for each $i\in\omega$.

	Another studied possibility is to consider the so-called approximate versions: for example, the \emph{approximate $\mu$-measurable chromatic number} is the smallest $k$ such that for every $\e>0$ there is a Borel set of vertices $A\subseteq V$ of measure at most $\e$ such that the graph induced by $V\setminus A$ can be coloured with at most $k$ colours in a Borel way. For example, the approximate measurable chromatic number of the irrational rotation graph $\C R_\alpha$ from Example~\ref{ex:Rotation} is~2 because the Lebesgue measure of one colour class, namely $X_2=[0,c)$, in the constructed Borel 3-colouring of $\C R_\alpha$ can be chosen to be arbitrarily small. The survey by Kechris and Marks~\cite{KechrisMarks:survey} gives an overview of such results as well.

	For the reader who, inspired by this paper, would like to read more on the topic, let us point some technical subtleties that are sometimes not mentioned explicitly in the literature. The measure $\mu$ is (almost) always assumed to be \emph{$\sigma$-finite}, meaning that $V$ can be covered by countably many sets of finite $\mu$-measure. This implies  many important properties such as the regularity of $\mu$  
	(\Co{Proposition 8.1.2}), being able to talk about the product of $\mu$ with other measures without the complications of going through the so-called complete locally determined products (see \cite[Chapter 25]{Fremlin03mt2}),
	etc.  Also, 
	there is a simple trick (see e.g.\ \cite[Proposition~3.2]{GrebikPikhurko20}) that
	allows us to construct another  measure $\nu$ on $(V,\C B)$ with $\C B_\nu\subseteq\C B_\mu$ such that $\nu$ is \emph{quasi-invariant} (meaning that the saturation $[N]\OptionalCG $ of any $\nu$-null set $N\subseteq V$ is a $\nu$-null set). Thus it is enough to find a measurable satisfying assignment when the measure is quasi-invariant. The advantage of the quasi-invariance of $\nu$  is that the neighbourhood and the saturation of any $\nu$-null (resp.\ $\nu$-measurable) set is also $\nu$-null (resp.\ $\nu$-measurable). The reader should also be aware that the measurability of an edge labelling $c:E\to\omega$ is understood as the measurability of the vertex labelling $\ell:V\to \omega$ that encodes $c$ under some fixed local rule as in Remark~\ref{rm:Reduction}. This is equivalent to requiring that, there is a Borel $\mu$-null set $N\subseteq V$ such that $\induced{c}{E\cap (V\setminus N)^2}$ is Borel. Note that it is not a good idea to define the measurability of $c:E\to\omega$ with respect the product measure $\mu\times \mu$ on $V^2\supseteq E$: if $\mu$ is atomless
	then $(\mu\times \mu)(E)=0$ by Tonelli's theorem and every subset of $E$ is $(\mu\times\mu)$-measurable. 
	
	A particularly important case is when $\mu$ is a \emph{probability measure} (that is, $\mu(V)=1$) which is moreover \emph{invariant}, meaning that every Borel map $f:V\to V$ with $\graph{f}\subseteq \C E_{\C G}$ preserves the measure~$\mu$. 
	(The reader should be able to show that it is enough to check the above property only for involutions $f$ with $\graph{f}\subseteq E$, that is, functions that come from matchings in~$\C G$.)
	In this case, the quadruple $(V,E,\C B,\mu)$ is now often called a \emph{graphing}. (The reader should also be aware of another different usage of this term, where a \emph{graphing} of an equivalence relation $\C E$ means a Borel graph $\C G$ whose connectivity relation $\C E_{\C G}$ coincides with~$\C E$.)
	
	The invariance of $\mu$ is a measure analogue of the obvious fact from finite combinatorics that any bijection preserves the sizes of finite sets. It has many equivalent reformulations such as, for example, the Mass Transport Principle (see e.g.\ \Lo{Section 18.4.1}).

	For an invariant probability measure $\mu$ on the vertex set $V$, one can define a new measure $\eta$ on edges
	where the $\eta$-measure of a Borel subset $A\subseteq E$ is defined to be $\int_{V}|\{y\in V\mid (x,y)\in A\}|\, \dd\mu(x)$, the average $A$-degree. Then, in fact, the invariance of $\mu$ is equivalent to $\eta$ being \emph{symmetric} (meaning that $\eta(A)$ is always the same as  the $\eta$-measure of the ``transpose'' $\{(x,y)\mid (y,x)\in A\}$ of $A$), 
	see e.g.\ \Lo{Section 18.2}. Interestingly, the $\eta$-measurability of an edge labelling $c:E\to\omega$ now coincides with the $\mu$-measurability of the vertex labelling $\ell:V\to\omega$ that encodes~$c$ under the Borel reduction of Remark~\ref{rm:Reduction}.
	
	Graphings can serve as the local limits of bounded degree graphs, roughly speaking as follows.
	An \emph{$r$-sample} from a graphing is obtained by sampling a random vertex $x\in V$ under the probability measure $\mu$
	and outputting $\RR F_r(x)$, the isomorphism type of the rooted graph induced by the $r$-ball around~$x$. Also,
	each finite graph $(V,E)$ can be viewed as a graphing $(V,E,2^V,\nu)$ where $\nu$ is the uniform measure on the finite set~$V$. Then the \emph{local} convergence can be described by a metric where two graphings are ``close'' if the distributions of their samples are ``close'' to each other, see~\Lo{Section 18} for an introduction to graphings as limit objects.
	
	Also, 
	the Schreier graph of any Borel probability measure-preserving action of a marked group is a graphing. It contains a lot of information about the action and thus is an important object of study in \emph{measured group theory}. Even for such a simple group as the integers~$\I Z$, its measure-preserving actions (which are specified by giving just one measure-preserving transformation) form a very beautiful and deep subject that is the main focus of the classical ergodic theory. 
	It is hard to pick a good starting introductory point for this vast area. The reader
	is welcome to consult various surveys and textbooks, 
	e.g.~\cite{AbertGaboriauThom17or,Furman11,Gaboriau10,
		KerrLi19egid,Loh20etmgh,
		Shalom05}, and pick one (or its part) that looks most interesting.

	A special but important case of a graphing is the Schreier graph $\C S$ of the shift action of a marked group $(\Gamma;S)$ on the product measure space $X^\Gamma$ where $X$ is a standard probability space, e.g.\ $\{0,1\}$ or $[0,1]$ with the uniform measure. Measurable labellings of $\C S$ are exactly the so-called \emph{factors of IID} labellings with vertex seeds from $X$ of the Cayley graph of~$(\Ga,S)$. This connection gives a way of applying methods of descriptive combinatorics for constructing  various invariant processes on vertex-transitive countable graphs. This is a very active area of discrete probability (see e.g.\ the book by Lyons and Peres~\cite{LyonsPeres16ptn})
	where even the case of trees has many tantalising unsolved questions (see e.g.\ Lyons~\cite{Lyons17}). 
	
	As a showcase of how the results that we have proved can be used in the measurable setting, let us present a very brief outline of the following ``measurable'' version of Hall's marriage theorem by Lyons and Nazarov~\cite[Remark~2.6]{LyonsNazarov11} (whose detailed proof can be found in~\cite[Theorem~3.3]{GrabowskiMathePikhurko:expansion}).
	
	\begin{theorem}[Lyons and Nazarov \cite{LyonsNazarov11}]
		\label{th:LN} Let $\e>0$ and let $\C G=(V,E,\C B,\mu)$ be a bipartite graphing with a Borel bipartition $V=B_0\cup B_1$ such that $\mu(B_0)=\mu(B_1)=1/2$ and for every measurable $X$ inside a part it holds that 
		$$
		\mu(N(X))\ge \min((1+\e)\mu(X),1/4+\e).
		$$ 
		Then $\C G$ has a Borel matching that covers all vertices except a null set.
	\end{theorem}
	
	In order to prove Theorem~\ref{th:LN}, we start with the empty matching $M_0$ and, iteratively for each $i\in\omega$, augment $M_i$ to $M_{i+1}$ in a Borel way using augmenting paths of (edge) length at most~$2i+1$, until none remains. This can done by the result of Elek and Lippner~\cite{ElekLippner10}. (Alternatively, we can use
	Theorem~\ref{th:ElekLippner} here, applying  all possible $(2i+2)$-augmentations to $M_i$, with the rest of the proof being the same.) It is a nice combinatorial exercise to show that for every $\e>0$ there is $c>0$ such that if a finite bipartite graph $G$ with both parts of the same size $n$ is an \emph{$\e$-expander} (i.e.\ the neighbourhood of any set $X$ in a part has size at least $\min((1+\e)|X|, (1/2+\e)n)$), then any matching without augmenting paths of length at most $2i+1$ covers all except at most $(1-c)^in$ vertices of $G$. The proof of this statement from \cite{LyonsNazarov11} extends from finite graphs to the measurable setting since all inequalities used by it come from double counting and,  by the invariance of~$\mu$, also apply when the relative sizes of sets are formally replaced by their measures. Thus the measure of the set $X_i$ of vertices unmatched by the Borel matching $M_i$ in $\C G$ is at most $(1-c)^i$. When we do augmentations to construct $M_{i+1}$, we change the current matching on at most $2i+2$ vertices per one new matched vertex by the second part of Theorem~\ref{th:ElekLippner}. Again, by the invariance of~$\mu$, the measure of vertices where $M_i$ and $M_{i+1}$ differ is at most $(2i+2) (1-c)^i$. Define the final Borel matching 
	$$
	M:=\textstyle \liminf_i M_i=\cup_{i\in\omega}\cap_{j\ge i} M_j
	$$ 
	to be the pointwise limit of~the matchings~$M_i$ where they stabilise. Let $X$ be the (Borel) set of vertices unmatched by $M$. For every $i\in\omega$ the following clearly holds: if a vertex is not matched by $M$ then it is not matched by $M_i$ or it witnesses at least one change (in fact, infinitely many changes) after Stage~$i$. Thus the Union Bound gives that $\mu(X)\le (1-c)^i+ \sum_{j\ge i} (2j+2)(1-c)^j$. Since this inequality is true for every $i$ and its right-hand side can be made arbitrarily small by taking $i$ sufficiently large, we conclude that $X$ has measure 0, as desired. (The reader may have recognised the last two steps as a veiled application of the Borel--Cantelli Lemma.)

	\section{Borel Colourings from \Local\ Algorithms}
	\label{se:Bernshteyn20arxiv}
	
	Under rather general settings, local rules are equivalent to the so-called deterministic \Local\ algorithms that were introduced by Linial~\cite{Linial87,Linial92}. Their various variants have been actively studied in theoretical computer science; for an introduction, we refer the reader to the book by Barenboim and Elkin~\cite{BarenboimElkin13dgc}. Here we briefly discuss this connection and present a result of Bernshteyn~\cite{Bernshteyn20arxiv} that efficient \Local\ algorithms can be used to find satisfying assignments that are Borel.
	
	Suppose that we search for an assignment $a:V\to\omega$ that solves a given LCL $\RT$ on a graph $G$ (for example, we would like $a$ to be a proper vertex colouring with $k$ colours) where, just for the clarity of presentation, we assume that $\RT$ is defined on unlabelled graphs.
	A \emph{deterministic \Local\ algorithm with $r$ rounds} is defined as follows. 
	Each vertex $x$ of $G$ is a processor with unlimited computational power. There are $r$ synchronous rounds. In each round, every vertex can exchange any amount of information with each of its neighbours. After $r$ rounds, every vertex $x$ has to  output its own value $a(x)$, with all vertices using the same algorithm for the communications during the rounds and the local computations.

	Clearly, the final value $a(x)$ is some function of the $r$-ball of $x$ so the produced assignment is given by some $r$-local rule. Conversely, for every $r$-local rule $\RR A$, a possible strategy is that each vertex $x$ collects all current information from all its neighbours in each round and computes $\RR A(x)$ at the end of $r$ rounds, by knowing everything about its whole $r$-ball $N^{\le r}(x)$.  Given this equivalence, we will use mostly the language of local rules.

	Such a rule need not exist if there are symmetries. For example, if the input graph is vertex-transitive then all vertices produce the same answer so there is no chance to find, for example, a proper vertex colouring. Let us assume here that each vertex $x$ is given
	the order $n$ of the graph and its unique identifier $\ell(x)\in n$. Thus we evaluate the rule $\RR A$, that may depend on $n$, on the labelled graph $(G,\ell)$.
	
	The corresponding algorithmic question is, for a given family of graphs $\CH$ (which we assume to be closed under adding isolated vertices) and an LCL $\RT$, to estimate $\Det_{\RT,\CH}(n)$, the smallest $r$ for which there is an $r$-local rule $\RR A$ such that,
	for every graph $G\in \CH$ with $n$ vertices and every bijection $\ell:V\to n$, 
	the assignment $\RR A(G,\ell)$ solves~$\RT$ on~$G$.  If there is some input $(G,\ell)$ as above which admits no $\RT$-satisfying assignment, then we define $\Det_{\RT,\CH}(n):=\omega$. 
	Note that the value of $\Det_{\RT,\CH}(n)$ will not change if we modify the above definition by allowing to take any graph $G\in\CH$ with at most $n$ vertices and any injection $\ell:V\to n$ (because we can always add isolated vertices to $G$ and extend $\ell$ to a bijection).
	
	Note that if $\Det_{\RT,\CH}(n)$ is finite then it is at most $n-1$. Indeed, by being able to see at distance up to $n-1$ (and knowing $n$), each vertex $x$ knows its injectively labelled component $\induced{(G,\ell)}{[x]}$; thus a possible $(n-1)$-local rule $\RR A$ is that $x\in V$ computes the lexicographically smallest (under the ordering of $[x]$ given by the values of $\ell$) $\RT$-satisfying assignment $a:[x]\to \omega$ and outputs~$a(x)$ as its value.

	The following result is a special case of~\cite[Theorem 2.10]{Bernshteyn20arxiv} whose proof, nonetheless, contains the main idea.

	\begin{theorem}[Bernshteyn~\cite{Bernshteyn20arxiv}]
		\label{th:LocBorel} Let $\CH$ be a family of graphs with degrees bounded by $d$ which is closed under adding isolated vertices. Let $\RT$ be an LCL on unlabelled graphs such that
		$\Det_{\RT,\CH}(n)=o(\log n)$ as $n\to\infty$. Then every Borel graph $\C G=(V,E,\C B)$ such that $\induced{\C G}{N^{\le r}(x)}\in \CH$ for every $x\in V$ and $r\in \omega$, admits a Borel assignment $a:V\to\omega$ that solves~$\RT$ on~$\C G$.\end{theorem}

	\begin{proof} Let $t$ be the radius of~$\RT$. Fix some  sufficiently large $n$, namely we require that 
		$1+d\, \sum_{i=0}^{s-1} (d-1)^{i}\le n$, where 
		$ s:=2(\Det_{\RT,\CH}(n)+t)$. 
		This is possible since $s=o(\log n)$ by our assumptions. Take any local rule $\RR A$ of radius $r:=\Det_{\RT,\CH}(n)$ that works for all graphs from $\CH$ on at most $n$ vertices. Fix a Borel $2(r+t)$-sparse colouring $c:V\to n$ of $\C G$ which exists by Corollary~\ref{cr:RSparse} (and our choice of $n$).
		
		Apply the rule $\RR A$ to $(\C G,c)$, viewing $c:V\to n$ as the identifier function, to obtain a labelling $a:V\to \omega$. (Thus, informally speaking, we run the algorithm pretending that the graph $\C G$ has $n$ vertices.)		
		Note that $a$ is well-defined for every vertex $x\in V$ since, by our assumption on $\C G$, the subgraph induced by the $r$-ball $N^{\le r}(x)$ belongs to $\CH$ and is injectively labelled by the $2r$-sparse $n$-colouring~$c$. By Lemma~\ref{lm:LocalRule}, the function $a$ is Borel, so it is remains to check that it solves the LCL $\RT$. Take any $x\in V$. Consider 
		$
		H:=\induced{(\C G,c)}{N^{\le r+t}_{\C G}(x)}$,
		the labelled subgraph induced in $\C G$ by the $(r+t)$-ball around~$x$.   
		If we apply the rule $\RR A$ to $H$ then we obtain the same assignment $a$ on $N^{\le t}_H(x)=N^{\le t}_{\C G}(x)$, 
		because for every vertex in this set its $c$-labelled $r$-balls in $H$ and $\C G$ are the same. 
		The graph $H$ of diameter at most $2(r+t)$ is injectively labelled by the $2(r+t)$-sparse colouring $c$ and, in particular, has at most $n$ vertices.
		By the correctness of $\RR A$, the assignment $a$ satisfies the $\RT$-constraint at $x$ on $H$ (and also on $\C G$). Thus $a:V\to\omega$ is the required Borel assignment.\end{proof}

	Let us point some algorithmic results when $\CH$ consists of graphs with maximum degree bounded by a fixed integer~$d$ while $n$ tends to~$\infty$. The deterministic \Local\ complexities of a proper vertex $(d+1)$-colouring, a proper edge $(2d-1)$-colouring, a maximal independent set and a maximal matching are all $O(\log^* n)$, where $\log^* n$ is the \emph{iterated logarithm} of $n$,  the number of times needed to apply the logarithm function to $n$ to get a value at most~$1$. (
	For references and the best known bounds as functions of $(n,d)$, we refer the reader to~\cite[Table~1.1]{ChangKopelowitzPettie19}.)
	These are exactly the problems for which we showed the existence of a Borel solution in Theorems~\ref{th:MaxInd}--\ref{th:ChiB'}. 
	In fact, Theorems~\ref{th:MaxInd}--\ref{th:ChiB'} for bounded degree graphs are, by Theorem~\ref{th:LocBorel}, direct consequences  of the above mentioned results on the existence of efficient local algorithms. While there seems to be a large margin (between the running time of $O(\log^* n)$ for known algorithms and the $o(\log n)$-assumption of Theorem~\ref{th:LocBorel}), in fact, Chang, Kopelowitz and Pettie~\cite{ChangKopelowitzPettie19} showed that, for LCLs on bounded-degree graphs, if there is a \Local\ algorithm with $o(\log n)$ rounds that solves the problem then there is one with $O(\log^* n)$ rounds. (In fact, the proof of the last result is similar to the proof of Theorem~\ref{th:LocBorel}: fix a large constant $r$, generate a proper colouring $a$ of the $r$-th power of the input order-$n$ graph 
	using $O(\log^*n)$ rounds  and then ``simulate'' the $o(\log n)$ algorithm using the values of $a$ in lieu of the vertex identifiers.)
	
	On the other hand, deterministic algorithms using only $o(\log n)$ rounds seem to be rather weak, e.g.\ for colouring problems when the number of colours is even slightly below the trivial greedy bound. One (out of many) results demonstrating this is by Chang et al.~\cite[Theorem~4.5]{ChangKopelowitzPettie19} who showed that vertex $d$-colouring of trees (of maximum degree at most $d$) requires $\Omega(\log n)$ rounds.  
	\hide{Page 32 of BE says Linial87 showed that $\Omega(n)$ rounds are needed for 2-colouring trees (bounded degree?) but I cannot find this result there.}%
	The last result should be compared with Theorem~\ref{th:Marks16} here.
	
	Let us also briefly discuss another general (and much more difficult) transference result of Bernshteyn~\cite{Bernshteyn20arxiv} that randomised \Local\ algorithms that require $o(\log n)$ rounds on $n$-vertex graphs
	give measurable assignments.
	In an \emph{$r$-round randomised \Local\ algorithm} $\RR A$, each vertex $x$ generates at the beginning  its own random seed $s(x)$, a uniform element of $m$, independently of all other choices. Vertices  can share any currently known seeds during each of the $r$ communication rounds. Thus each obtained value $\RR A(x)$ can also depend on the generated seeds inside~$N^{\le r}(x)$. Equivalently, once the function $s:V\to m$ has been generated, the resulting assignment $\RR A$ is computed by some local rule on $(G,s)$.
	Given an LCL $\RT$, we say that the algorithm \emph{solves $\RT$} on an $n$-vertex graph $G=(V,E)$ if, for every vertex $x$ of $G$ the probability that $\RT(x)$ fails is at most~$1/n$. Let $\Rand_{\RT,\CH}(n)$ be the smallest $r$ such that some $r$-round randomised \Local\ algorithm (for some $m=m(\RT,\CH,n)$) solves $\RT$ for every $n$-vertex graph in~$\CH$.
	Note that we do not need any identifier function here, since $s(x)$ uniquely identifies a vertex $x$ with probability $(1-1/m)^{n-1}$ which can be made arbitrarily close to 1 by choosing $m$ sufficiently large.
	Under these conventions, Bernshteyn~\cite[Theorem 2.14]{Bernshteyn20arxiv} proved, roughly speaking, that if $\Rand_{\RT,\CH}(n)=o(\log n)$, then the corresponding LCL on bounded-degree Borel graphs admits a satisfying assignment which is $\mu$-measurable for any probability measure $\mu$ on $(V,\C B)$ (resp.\ Baire measurable for any given Polish topology $\tau$ on $V$ with $\sigma_V(\tau)=\C B(V)$).
	This result has already found a large number of applications, see \cite[Section 3]{Bernshteyn20arxiv}.

	\section{Borel Results that Use Measures or Baire Category}
	\label{se:Borel+}

	Of course, every result that an LCL admits no $\mu$-measurable (resp.\ no Baire measurable) solution on some Borel graph $\C G$ automatically implies that no Borel solution exists for $\C G$ either. There is a whole spectrum of techniques for proving results of this type (such as 
	ergodicity 
	that was briefly discussed in Section~\ref{se:Intro} in the context of Example~\ref{ex:Rotation}). We refer the reader to the survey by Kechris and Marks~\cite{KechrisMarks:survey} that contains many further examples of this kind.
	
	Rather surprisingly, measures have turned to be useful in proving also the existence of full Borel colourings for some problems. Let us briefly discuss a few such (very recent) results.

	One is a result of Bernshteyn and Conley~\cite{BernshteynConley19arxiv} who proved a very strong
	Borel version of the theorem of Hajnal and Szemer\'edi~\cite{HajnalSzemeredi70}. Recall that the original Hajnal-Szemer\'edi Theorem states that if $G$ is a finite graph of maximum	degree $d$ and $k\ge d+1$ then $G$ has an \emph{equitable colouring} (that is, a proper colouring $c:V\to k$ such that every two colour classes of $c$ differ in size at most by $1$). For a Borel graph $\C G$,  let us call a Borel $k$-colouring, given by a partition $V=V_0\cup\dots\cup V_{k-1}$, \emph{equitable} if for every $i,j\in k$ there is a Borel bijection $g:V_i\to V_j$ with $\graph{g}\subseteq \C E\OptionalCG $.
	Of course, an obvious obstacle here is the existence of a finite component whose size is not divisible by~$k$. The main result of Bernshteyn and Conley~\cite[Theorem 1.5]{BernshteynConley19arxiv} is that this is the only obstacle. Very briefly, the proof in~\cite{BernshteynConley19arxiv} runs a Borel version of the algorithm of  
	Kierstead, Kostochka, Mydlarz and Szemer\'edi~\cite{KiersteadKostochkaMydlarzSzemeredi10} that finds equitable colourings in finite graphs, where a current proper $k$-colouring gets 
	``improved'' from the point of view of equitability
	via certain recolouring moves. This iterative procedure gives only a partial colouring as it is unclear how to colour the set $X$ of components that contain vertices that change their colour infinitely often. However, Bernshteyn and Conley~\cite{BernshteynConley19arxiv} proved via the Borel--Cantelli Lemma that if $\mu$ is an arbitrary probability measure on $(V,\C B)$ which is invariant, then $\mu(X)=0$. Note that the conclusion holds even though the definition of $X$ does not depend on $\mu$. This means that $\induced{\C G}{X}$ does not admit any invariant probability measure. This is, by a result of Nadkarni~\cite{Nadkarni90}, equivalent to $X$ being \emph{compressible}, meaning that there is a Borel set $A\subseteq V$ intersecting every component of $\induced{\C G}{X}$ and a Borel bijection $f: X\to X\setminus A$ with $\graph{f}\subseteq \C E_{\C G}$. (Note that the converse direction in Nadkarni's result is easy: if such a function $f$ exists and $\mu$ is an invariant probability measure on $(X,\C B)$, then the Borel sets $A,f(A),f(f(A)),\dots$ are disjoint and have the same $\mu$-measure, which is thus 0; however, then the saturation $X$ of the $\mu$-null set $A$ cannot have positive $\mu$-measure, a contradiction.) 
	Then a separate argument shows that the remaining compressible graph~$\induced{\C G}{X}$ admits a full Borel equitable colouring.

	Another result that we would like to discuss comes from a recent paper of Conley and Tamuz~\cite{ConleyTamuz20}. Call a colouring $c:V\to 2$ \emph{unfriendly} if every $x\in V$ has at least as many neighbours of the other colour as of its own, that is,
	$$
	|\,\{y\in N(x)\mid c(y)\not= c(x)\}\,|\ge | \,\{y\in N(x)\mid c(y)= c(x)\}\,|.
	$$
	Such a colouring trivially exists for every finite graph: take, for example, one that maximises the number of non-monochromatic edges (i.e.\ a max-cut colouring). Also, the Axiom of Choice (or the Compactness Principle) shows that every locally finite graph admits an unfriendly colouring. Interestingly, Shelah and Milner~\cite{ShelahMilner90} showed that there are graphs with uncountable vertex degrees that have no unfriendly colouring; the case of locally countable graphs is open.
	
	Conley and Tamuz~\cite{ConleyTamuz20} showed that every Borel bounded-degree graph admits a Borel unfriendly colouring provided the graph has \emph{subexponential growth} meaning that
	\beq{eq:SubExp}
	\forall \e>0\ \exists n_0\in\omega\ \forall n\ge n_0\ \forall x\in V\quad  |N^{\le n}(x)|\le
	(1+\e)^n,
	\eeq
	or, informally, that $n$-balls have size at most $(1+o(1))^n$ uniformly in~$n$.
	Like in~\cite{BernshteynConley19arxiv}, they use a family of augmenting moves so that, for every invariant probability measure $\mu$ on $(V,\C B)$, the set $X$ where the constructed colourings do not stabilise has $\mu$-measure 0. Here, each move is very simple: if more than half of neighbours of a vertex $x$ have the same colour as $x$, then we change the colour of~$x$. We do these moves in stages so that the moves made in one stage do not interfere with each other (similarly as in the proof of Theorem~\ref{th:ElekLippner}). Since the graph has subexponential growth, for every $x\in V$ there is a probability measure $\mu_x$ which is very close to being invariant and satisfies $\mu_x(\{x\})>0$. (For example, take the discrete measure that is supported on $[x]$
	and puts weight $(1-\e)^i\mu_x(\{x\})$ on each vertex at distance $i$ from $x$ for all $i\in\omega$.) Since there happens to be some leeway when applying the Borel--Cantelli Lemma in the invariant case, it can also be applied to $\mu_x$ provided $\e>0$ is sufficiently small. As $\{x\}$ has positive measure in $\mu_x$, it cannot belong to~$X$. Thus $X=\emptyset$, as desired. 
	
	This idea was used also by Thornton~\cite{Thornton20arxiv} to prove that every Borel graph of maximum degree $d$ and of subexponential growth admits a Borel orientation of edges so that each out-degree is at most $d/2+1$.

	It is not clear if compressibility helps for the last two problems. In particular, it remains an open problem if, for example, every 3-regular Borel graph admits an unfriendly Borel $2$-colouring.

	\section{Graphs of Subexponential Growth}
	\label{se:SubExp}

	Recall that the notion of subexponential growth was defined in~\eqref{eq:SubExp}.
	Unfriendly $2$-colouring is one example of 
	a problem which admits a Borel solution for every graph of subexponential growth but this becomes an open problem or a false statement when the growth assumption is removed. Let us just point to some general results which show that Borel graphs of subexponential growth are indeed more tractable.
	
	Cs{\'o}ka, Grabowski, {M\'ath\'e}, Pikhurko and Tyros~\cite{CsokaGrabowskiMathePikhurkoTyros:arxiv} showed that such graphs admit a Borel satisfying assignment for every LCL where the existence of global solution can be established by the symmetric Lov\'asz Local Lemma.
	The Local Lemma, introduced in a paper of Erd\H os and Lov\'asz~\cite{ErdosLovasz75}, is a very powerful tool for proving the existence of a  satisfying assignment. A special case of it is as follows. Suppose that  we have a collection of bad events $\{B\mid B\in V\}$,
	each having probability at most $p$ and being a function of some finite set $\var(B)$ of binary random variables. Define the \emph{dependency graph} $D$ on $V$, where two events are connected if they share at least one variable. The Local Lemma gives that if
	\beq{eq:LLL}
	p<\frac{(\Delta-1)^{\Delta-1}}{\Delta^\Delta},
	\eeq
	where $\Delta$ is the maximum degree of $D$,
	then there is an assignment of variables such that no bad event occurs. (Remarkably, Shearer~\cite{Shearer85} showed that the bound in~\eqref{eq:LLL} is, in fact, best possible.)

	The Borel version of this result from~\cite{CsokaGrabowskiMathePikhurkoTyros:arxiv} is quite technical to state; informally speaking it states that if bad events and variables are indexed by elements of some standard Borel space $V$ so that the corresponding dependency graph $\C D$ on $V$ is Borel and has subexponential growth, and~\eqref{eq:LLL} holds, then there is a Borel assignment of variables such that no bad event occurs. For other ``definable'' versions of the Local Lemma, see Bernshteyn~\cite{Bernshteyn19am,Bernshteyn20arxiv} and Kun~\cite{Kun13arxiv}.

	An efficient randomised algorithm that finds an assignment in finite graphs whose existence is guaranteed by the Local Lemma was found in a breakthrough work of 
	Moser and Tardos~\cite{Moser09,MoserTardos10}. Actually, one example of a good algorithm is very simple: start with any initial assignment and, as long as there is an occurrence of some bad event $B$, pick one such $B$ arbitrarily and re-sample all variables in $\var(B)$. 
	This can be adopted to the Borel setting using the ideas of the proof of Theorem~\ref{th:ElekLippner} as follows. Fix some sufficiently large~$r_0$. Take an $r_0$-sparse Borel colouring $c$ of the dependency graph $\C D$. For each colour $i$, generate a uniform random sequence of binary bits $b_i:=(b_{i,j})_{j\in\omega}$. (Note that we have to generate only countably many bits.) Let the initial assignment of variables be $x\mapsto b_{c(x),0}$. At each iteration, find via Theorem~\ref{th:MaxInd} a Borel set $I$ of currently occurring bad events in which every two have disjoint sets of variables and, moreover, $I$ is a maximal set with this property. (Note that the subexponential growth assumption implies that $\C D$ has finite maximum degree.) Reassign the value of every $x$ in the (disjoint) union $\cup_{B\in I} \var(B)$ to the next bit of the sequence $b_{c(x)}$ that has not been used by the variable $x$ yet. The problem with this naive adaptation is that variables which are far away in $\C D$ can depend on each other and the estimates of Moser and Tardos do not apply because long chains of interdependent bad events may have now very different probabilities.

	The key idea of~\cite{CsokaGrabowskiMathePikhurkoTyros:arxiv} is that, because of the subexponential growth assumption, if some variable $y\in V$ is resampled many times at some finite stage of this procedure then there are another variable $x$ and an integer $r\le r_0/2$ such that, among all resampled bad events $B$, the number of \emph{internal} ones (those with $\var(B)\subseteq N^{\le r}(x)$) is very small compared with the number of \emph{boundary} ones (those with $\var(B)$ intersecting both $N^{\le r}(x)$ and its complement). The internal resamples obey the Moser--Tardos estimates because, by $r\le r_0/2$, they never use the same random bit twice. On the other hand, the number $k$ of boundary resamples is so small that we can afford to take the Union Bound over all possible ways of how at most $k$ uncontrollable boundary events can pop up around $x$ during the run of the algorithm. This means that, there is an assignment of binary bits $b_{i,j}$ and some $n\in\omega$ such that \textbf{every} variable is resampled at most $n$ times. It follows by the finiteness of $\Delta(\C D)$ that this procedure stabilises for every variable. The final colouring, as a finitary function, is Borel by Corollary~\ref{cr:Finitary}.

	The following general application of the Borel Local Lemma from~\cite{CsokaGrabowskiMathePikhurkoTyros:arxiv} was observed by Bernshteyn~\cite[Theorem 2.15]{Bernshteyn20arxiv}:  if $\Rand_{\CH,\RT}(n)=O(\log n)$ and $\C G$ is a Borel graph of subexponential growth with every ball belonging to the graph family $\CH$ then there is a Borel assignment $a:V\to\omega$ solving the LCL $\RT$. (Recall that $\Rand_{\CH,\RT}(n)$, as defined in Section~\ref{se:Bernshteyn20arxiv}, is the smallest number of rounds in a randomised \Local\ algorithm that fails any one $\RT$-constraint with probability at most $1/n$ for every $n$-vertex graph in~$\C G$.) The proof idea is as follows. Fix large $n$ and a suitable randomised algorithm that uses $r:=\Rand_{\CH,\RT}(n)$ rounds for some~$m$. This gives an $r$-local rule $\RR A$ that can be evaluated on $V$ once we have some seed function $s:V\to m$. We view the values of $s$ as variables. For every $x\in V$, let the ``bad'' event $B_x$ state that the assignment returned by $\RR A$ fails the $\RT$-constraint at~$x$. If the radius of the local rule $\RT$ is $t$, then $B_x$ depends only on the values of $s$ in $N^{\le r+t}(x)$. In particular, $B_x$ and $B_y$ can share a variable only if the distance between $x$ and $y$ is at most $2(r+t)$. Thus the maximum degree of the corresponding dependency graph $\C D$ is at most the maximum size of a ball in $\C G$ of radius $2(r+t)$. This is at most $o(n)$ since $r=O(\log n)$, $t$ is a constant and $\C G$ has  subexponential growth. Also, the probability of each $B_x$ (for a random uniform function $s:N^{\le r+t}(x)\to m$) is at most~$1/n$. Thus the Borel Local Lemma from~\cite{CsokaGrabowskiMathePikhurkoTyros:arxiv} (which also works if the bits $b_{i,j}$ are $m$-ary instead of binary)  gives that there is a Borel assignment $s:V\to m$ with no bad event $B_x$ occurring. Then the evaluation of $\RR A$ on $(\C G,s)$ satisfies the LCL $\RT$ and is a Borel function by Lemma~\ref{lm:LocalRule}.
	
	Thus any LCL that can be solved by a randomised \Local\ algorithm of radius $O(\log n)$ on graphs of order $n\to\infty$ admit Borel solutions on any Borel graph of subexponential growth. 
	For some examples of such problems that are interesting from the point of view of Borel combinatorics, see Section~3 in~\cite{Bernshteyn20arxiv}. 
	One is the result of Molloy and Reed~\cite{MolloyReed14} who proved that, for $d\ge d_0$ with $k_d$ being the maximum integer with $(k+1)(k+2)\le d$, if each 1-ball in a graph $G$ of maximum degree $d$ can be properly coloured with $c\ge d-k_d$ colours, then in fact the whole graph $G$ can be properly coloured with $c$ colours. For large $d$, this is a far-reaching generalisation of Brooks' theorem~\cite{Brooks41} which, for $d\ge 3$, corresponds to the case $c=d$ of the above implication. Bamas and Esperet~\cite{BamasEsperet18arxiv,BamasEsperet19} proved that, for a fixed large $d$, a $c$-colouring of an $n$-vertex graph $G$ whose existence is guaranteed by the above result of Molloy and Reed, can in fact be found by a randomised \Local\ algorithm using $o(\log n)$ rounds. Putting all together, we conclude that every Borel graph $\C G$ of maximum degree $d\ge d_0$ and subexponential growth has Borel chromatic number at most $\max(d-k_d,\chi(\C G))$. Note that the last statement fails if we remove the growth assumption, even when we look at $d$-colourings only: indeed, the graph $\C G$ given by Theorem~\ref{th:Marks16} has Borel chromatic number $d+1$ while the whole graph $\C G$, as thus each of its 1-balls, is bipartite (as $\C G$ has no cycles at all).

	\section{Applications to Equidecomposability}
	\label{se:Equidec}
	
	In order to demonstrate how some of the above results are applied, let us very briefly discuss the question of equidecomposability, where the
	methods of descriptive combinatorics have been recently applied with great success.
	
	Two subsets $A$ and $B$ of $\I R^n$ are called \textit{equidecomposable} if it is possible to find a partition $A=A_0\cup\cdots\cup A_{m-1}$ and isometries $\gamma_0,\dots,\gamma_{m-1}$ of $\I R^n$ so that $\gamma_0.A_0,\dots,\gamma_{m-1}.A_{m-1}$
	partition the other set $B$ (or, in other words, we can split $A$ into finitely many pieces and rearrange them using isometries to form a partition of~$B$). The most famous result about equidecomposable sets is probably the \textit{Banach-Tarski Paradox}~\cite{BanachTarski24}: in 
	$\I R^3$, the unit ball and two disjoint copies of the unit ball are equidecomposable.
	
	Equidecompositions are often constructed to show  that certain kinds of isometry-invariant means do not exist. For example, the Banach-Tarski Paradox implies that every finitely additive isometry-invariant  mean  defined on all bounded subsets of~$\I R^3$ must be identically~0. We refer the reader to
	the monograph by Tomkowicz and Wagon~\cite{TomkowiczWagon:btp} on the subject. 
	
	The connection to descriptive combinatorics comes from a well-known observation that if one fixes the set of isometries $S=\{\gamma_0,\dots,\gamma_{m-1}\}$ to be used, then an equidecomposition between $A$ and $B$ is equivalent to a perfect matching in the bipartite graph
	\beq{eq:GABS}
	G
	:=\left(\,A\sqcup B,\big\{\,(a,b)\,\in A\times B\mid \exists \gamma\in S\ \gamma.a=b\,\}\,\big\}\,
	\right).
	\eeq
	
	A fairly direct application of Theorem~\ref{th:BaireM} of Marks and Unger~\cite{MarksUnger16} gives a new proof of the important result of Dougherty and Foreman~\cite{DoughertyForeman94}, whose original proof was very complicated, that doubling a ball in the Banach-Tarski Paradox can be done with pieces that have the property of Baire.  
	\hide{One of equivalent reformulation of the latter result is the following. Let us write $A\sim U_0\sqcup\ldots\sqcup U_{n-1}$ to mean that $U_0,\dots, U_{n-1}$ are pairwise disjoint open set such that $A$ is the topological closure of their union. Dougherty and Foreman~\cite{DoughertyForeman94} proved that  there are finitely pairwise disjoint open sets $U_0,\dots,U_{n-1}$ and isometries~$\gamma_0,\dots,\gamma_{n-1}$ of $\I R^3$ so that $\I B\sim 
		U_0\sqcup\ldots\sqcup U_{n-1}$ while $\I B\cup (\I B+3)\sim 
		\gamma_0.U_0\sqcup\ldots\sqcup \gamma_{n-1}.U_{n-1}$, where $\I B:=\{x\in\I R^3\mid \|x\|_2\le 1\}$ is the Euclidean unit ball.}

	Of course, doubling a ball is impossible with Lebesgue measurable pieces because equidecompositions have to preserve the Lebesgue measure (as it is invariant under isometries). 
	Interestingly, the obvious necessary condition for $A\subseteq \I R^n$, $n\ge 3$, to be equidecomposable with Lebesgue measurable pieces to, say, the cube $[0,1]^n$ (namely, $A$ is Lebesgue measurable of measure 1, and finitely many congruents of each of $A$ and $[0,1]^n$ are enough to cover the other set) was shown by Grabowski, M\'ath\'e and Pikhurko~\cite{GrabowskiMathePikhurko:expansion} to be sufficient. The proof carefully chooses 
	isometries $\gamma_0,\dots,\gamma_{m-1}$, applies Theorem~\ref{th:LN} of Lyons and Nazarov~\cite{LyonsNazarov11} to the graph $G$ defined in~\eqref{eq:GABS} (after removing a null set from $A$ and $B$ to make these sets and the graph $G$ Borel) and fixes the remaining null sets of unmatched vertices using the Axiom of Choice. The hardest part here was to show the existence of suitable isometries such that the bipartite graph $G$ is a ``measure expander'', although the current version of~\cite[Section~6.5]{GrabowskiMathePikhurko:expansion} points out a few different proofs of this step, 
	all relying on some version of the spectral gap property.
	
	The Borel version of the above result, say, if every two bounded Borel subsets of $\I R^n$, $n\ge 3$, with non-empty interior and the same Lebesgue measure are equidecomposable with Borel pieces, remains open. (The examples by Laczkovich~\cite{Laczkovich93,Laczkovich03} show that this statement is false for $n\le 2$.) However, if the sets have ``small'' boundary then the following very strong results can be proved in every dimension; in particular, they all apply to the famous \emph{Circle Squaring Problem} of Tarski~\cite{Tarski25}. (We refer the reader to corresponding papers for all missing definitions.)
	
	\begin{theorem}
		\label{th:SmallBoundary} Let $n\ge 1$ and $A,B\subseteq \I R^n$ be bounded sets with non-empty interior such that $\mu(A)=\mu(B)$ (where $\mu$ denotes the Lebesgue measure on $\I R^n$) and $\boxdim(\partial A),\boxdim(\partial B)<n$ (i.e.\ their topological boundaries have upper Minkowski dimension  less than $n$).
		
		\begin{enumerate}
			\item {\rm{(Laczkovich~\cite{Laczkovich92,Laczkovich92b}):}} The sets $A$ and $B$ are equidecomposable using translations.
			
			\item {\rm{(Grabowski, M\'{a}th\'{e} and Pikhurko~\cite{GrabowskiMathePikhurko17}):}} The sets $A$ and $B$ are equidecomposable using translations with pieces that are both Lebesgue and Baire measurable.
			
			\item {\rm{(Marks and Unger~\cite{MarksUnger17}):}} If, additionally, the sets $A$ and $B$ are Borel then they are equidecomposable using translations with Borel pieces.
			
			
		\end{enumerate}
	\end{theorem}

	Very briefly, some of the key steps in the above results are as follows. All papers assume $A,B\subseteq [0,1)^n$ and work modulo 1 (i.e.\ inside the torus $\I T^n:=\I R^n/\I Z^n$).
	Given $A$ and $B$, we first choose a large integer $d$ and then some vectors $\V x_0,\dots,\V x_{d-1}\in \I T^n$ (a random choice will work almost surely). In particular, we assume that $\V x_i$'s are linearly independent over the rationals. Let $\C G$ be the Schreier graph of the natural action of the additive marked group $\Ga\cong \I Z^d$ generated by $S:=\{\sum_{i\in d} \e_i \V x_i\mid \e_i\in \{-1,0,1\}\}$ on the torus $\I T^n$. 
	In other words, $\C G$ has $[0,1)^n$ for the vertex set with distinct $\V x,\V y\in [0,1)^n$ being adjacent  if their difference modulo 1 belongs to~$S$. 
	Thus each component of $\C G$ is a copy of the ($3^d-1$)-regular graph on~$\I Z^d$. We
	fix large $N$ and look for a bijection $\phi:A\to B$ such that for every $a\in A$ the distance in the graph $\C G$ between $a$ and $\phi(a)$ is at most $N$. If we succeed, then we have equidecomposed the sets $A$ and $B$ using at most $(2N+1)^d$ parts.

	The deep papers of Laczkovich~\cite{Laczkovich90,Laczkovich92,Laczkovich92b} show that the assumption $\boxdim(\partial A)<n$ translates into the set $A$ being really well distributed inside each component of $\C G$ and this property in turn shows that if $N$ is large then the required bijection $\phi$ exists by Rado's theorem  (Theorem~\ref{th:Rado42}). This crucially uses the Axiom of Choice (since Rado's theorem does). 
	
	The equidecompositions built in~\cite{GrabowskiMathePikhurko17} come from some careful augmenting local algorithms of growing radii (tailored specifically to $\I Z^d$-actions) and showing that the set where they do not stabilise is null (by the Borel--Cantelli Lemma) and meager (by adopting the proof of Theorem~\ref{th:BaireM}).
	
	Marks and Unger~\cite{MarksUnger17} approached this problem in a novel way via real-valued flows in the graph $\C G$. The advantage of working with flows (versus matchings) is that, for example, any convex combination of feasible flows is again a feasible flow.
	First, Marks and Unger showed that there is a real-valued uniformly bounded Borel flow $f$ from $A$ to $B$ (which can be viewed as a fractional version of the required bijection~$\phi$).  Secondly, they proved that one can \emph{round} $f$ to a flow $h$ (which has the same properties as $f$ and, additionally, assumes only integer values). Finally, they showed that the flow $h$ can be converted into the desired Borel bijection $\phi:A\to B$ via a local rule. In this approach the second step is the most difficult one. This step relies on the unpublished result of Gao, Jackson, Krohne, and Seward announced in~\cite{GaoJacksonKrohneSeward15arxiv}
	(for a proof see~\cite[Theorem 5.5]{MarksUnger17}) that $\C G$ (or, more generally, the Schreier graph coming from any free Borel action of $\I Z^d$) admits a Borel family $\C C$ of finite connected vertex sets that cover all vertices and whose boundaries in $\C G$ are sufficiently far from each other. 
	(In other words, the family $\C C$ provides a certificate of hyperfiniteness of $\C E_{\C G}$ that has an extra boundary separation property.) 
	One can arrange $\C C$ to arrive in $\omega$-many stages so that, for any newly arrived set $C\in\C C$, every previous set is either deep inside $C$ or far away from~$C$.
	The rounding algorithm from the point of view of any vertex $x$ is to wait until some $C\in\C C$ containing $x$ arrives and then round all $f$-values inside $C$  in agreement with the other vertices of $C$, making sure not to override any rounding made inside any earlier $C'\in\C C$ with $C'\subseteq C$. Thus the final integer-valued flow $h$ can be computed by a finitary rule on $\C G$ (depending of the real-valued Borel flow $f$) and is Borel
	by a version of Corollary~\ref{cr:Finitary}.
	
	More recently, M\'ath\'e, Noel and Pikhurko~\cite{MatheNoelPikhurko} strengthened the results from~\cite{MarksUnger17} by proving that, additionally, we can require that the pieces themselves  have boundary of upper Minkowski dimension less than $n$ (and, in particular, are Jordan measurable). Also, it is shown in~\cite{MatheNoelPikhurko} that if the sets $A$ and $B$ in Theorem~\ref{th:SmallBoundary} are, say, open then each piece can additionally be a Boolean combination of \emph{$F_\sigma$-sets} (i.e.\ countable unions of closed sets). 
	These improvements started with the new result that, in order to find an integer-valued flow $h$ inside some $C\in\C C$ that will be compatible with all future rounding steps, we need to know only some local information, namely, the ball around~$C$ of sufficiently large but finite radius. Thus the new rounding algorithm does not need to know the flow $f$ (which may depend on the whole component of $\C G$). 
	
	\section{Concluding Remarks}
	
	Due to the limitation on space, we have just very briefly touched on some very exciting topics, with each of Sections~\ref{se:CBER}--\ref{se:Measurable} deserving a separate introductory paper (perhaps even a few, as is the case of $\mu$-measurable combinatorics). Also, there are some other topics that we have not even mentioned (the Borel hierarchy, analytic graphs and equivalence relations; treeability; the cost of a measure-preserving group action; combinatorial cost; classical/sofic entropy and other invariants; 
	the local-global convergence of bounded degree graphs; ``continuous combinatorics'' on zero-dimensional Polish spaces, etc). Also, we have presented hardly any open questions; we refer the reader to e.g.\ the survey by Kechris and Marks~\cite{KechrisMarks:survey} that contains quite a few of them.
	
	Nonetheless, the author hopes that this paper will be helpful in introducing more researchers to this dynamic (in both meanings of the word) and exciting area.

\thankyou{The author is grateful to Jan Greb\'\i k, V\'aclav Rozho\'n and the anonymous referee for the very useful  comments and/or discussions.
}

\providecommand{\bysame}{\leavevmode\hbox to3em{\hrulefill}\thinspace}
\providecommand{\MR}{\relax\ifhmode\unskip\space\fi MR }
\providecommand{\MRhref}[2]{%
	\href{http://www.ams.org/mathscinet-getitem?mr=#1}{#2}
}
\providecommand{\href}[2]{#2}



\myaddress

\end{document}